\documentclass[12pt]{amsart}
\usepackage{amssymb,amsfonts,amsmath,amsopn,amstext,
amsthm,amscd,color,enumerate,graphicx,latexsym,mathabx,
mathrsfs,mathtools,skak,stackrel,todonotes,xy,url,
verbatim,tikz-cd}
\usepackage{hyperref}
\usepackage{cleveref}

\tikzcdset{scale cd/.style={every label/.append style={scale=#1},
    cells={nodes={scale=#1}}}}

\input xy
\xyoption{all}

\theoremstyle{remark}

\newtheorem{para}{\bf}[subsection]
\newtheorem{assumption}[para]{\bf Assumption}

\newtheorem{example}[para]{\bf Example}

\newtheorem{rem}[para]{\bf Remark}
\theoremstyle{definition}

\newtheorem{dfn}[para]{Definition}

\theoremstyle{plain}

\newtheorem{thm}[para]{Theorem}

\newtheorem{lemma}[para]{Lemma}
\newtheorem{conj}[para]{Conjecture}
\newtheorem{cor}[para]{Corollary}
\newtheorem{prop}[para]{Proposition}

\crefname{prop}{prop.}{props.}

\newenvironment{numequation}{\addtocounter{para}{1}
\begin{equation}}{\end{equation}}

\DeclareMathSymbol{\sqbullet}{2}{mathb}{"0D}

\newcommand{\bbA}{{\mathbb A}}
\newcommand{\bbB}{{\mathbb B}}
\newcommand{\bbC}{{\mathbb C}}

\newcommand{\bbF}{{\mathbb F}}
\newcommand{\bbG}{{\mathbb G}}

\newcommand{\bbN}{{\mathbb N}}

\newcommand{\bbQ}{{\mathbb Q}}

\newcommand{\bbZ}{{\mathbb Z}}

\newcommand{\bA}{{\bf A}}
\newcommand{\bB}{{\bf B}}

\newcommand{\bG}{{\bf G}}

\newcommand{\bN}{{\bf N}}

\newcommand{\bT}{{\bf T}}

\newcommand{\bZ}{{\bf Z}}

\newcommand{\bh}{{\bf h}}

\newcommand{\frb}{{\mathfrak b}}
\newcommand{\frc}{{\mathfrak c}}

\newcommand{\frg}{{\mathfrak g}}

\newcommand{\frl}{{\mathfrak l}}
\renewcommand{\frm}{{\mathfrak m}}
\newcommand{\frn}{{\mathfrak n}}

\newcommand{\frt}{{\mathfrak t}}

\newcommand{\frz}{{\mathfrak z}}

\newcommand{\frA}{{\mathfrak A}}

\newcommand{\frX}{{\mathfrak X}}
\newcommand{\frY}{{\mathfrak Y}}

\newcommand{\cA}{{\mathcal A}}

\newcommand{\cC}{{\mathcal C}}

\newcommand{\cF}{{\mathcal F}}

\newcommand{\cL}{{\mathcal L}}

\newcommand{\cO}{{\mathcal O}}

\newcommand{\cR}{{\mathcal R}}

\newcommand{\cT}{{\mathcal T}}
\newcommand{\cU}{{\mathcal U}}

\newcommand{\cW}{{\mathcal W}}
\newcommand{\cX}{{\mathcal X}}

\newcommand{\cZ}{{\mathcal Z}}

\newcommand{\ove}{{\overline e}}

\newcommand{\ovk}{{\overline k}}

\newcommand{\sA}{{\mathscr A}}

\newcommand{\sL}{{\mathscr L}}

\newcommand{\wtc}{{\widetilde c}}

\newcommand{\wtO}{\widetilde{O}}

\newcommand{\ua}{\underline{a}}

\newcommand{\uk}{\underline{k}}

\makeatletter
\newcommand*\bdot{{\mathpalette\bdot@{.9}}}
\newcommand*\bdot@[2]{\mathbin{\vcenter{\hbox{\scalebox{#2}{$\m@th#1\bullet$}}}}}
\makeatother

\newcommand{\foo}{\mapstochar\mathrel{\mspace{1.2mu}}\leadsto}

\newcommand{\ab}{{\rm ab}}
\newcommand{\Ad}{{\rm Ad}}

\newcommand{\Aff}{{\rm Aff}}
\newcommand{\alg}{{\rm alg}}
\newcommand{\an}{{\rm an}}

\newcommand{\cl}{{\rm cl}}

\newcommand{\co}{{\rm co}}

\newcommand{\coeq}{{\rm coeq}}

\newcommand{\Comod}{{\rm Comod}}

\newcommand{\diag}{{\rm diag}}

\newcommand{\dR}{{\rm dR}}

\newcommand{\End}{{\rm End}}

\newcommand{\eq}{{\rm eq}}

\newcommand{\Ext}{{\rm Ext}}
\newcommand{\Fil}{{\rm Fil}}
\newcommand{\Fp}{{\bbF_p}}
\newcommand{\Fpbar}{{\overline{\bbF}_p}}

\newcommand{\Frob}{{\rm Frob}}

\newcommand{\FS}{{\rm FS}}
\newcommand{\ft}{{\rm ft}}

\newcommand{\Gal}{{\rm Gal}}
\newcommand{\GL}{{\rm GL}}

\newcommand{\Hom}{{\rm Hom}}

\newcommand{\id}{{\rm id}}
\newcommand{\im}{{\rm im}}
\newcommand{\Ind}{{\rm Ind}}
\newcommand{\cInd}{{\rm c\text{-}Ind}}

\newcommand{\isom}{\stackrel{\simeq}{\lra}}

\newcommand{\la}{{\rm la}}

\newcommand{\lra}{\longrightarrow}

\newcommand{\Mod}{{\rm Mod}}

\newcommand{\op}{{\rm op}}
\newcommand{\ot}{\otimes}

\newcommand{\pr}{{\rm pr}}

\newcommand{\Psh}{{\rm Psh}}

\newcommand{\Qp}{{\bbQ_p}}
\newcommand{\ra}{\rightarrow}

\newcommand{\red}{{\rm red}}
\newcommand{\Res}{{\rm Res}}
\newcommand{\Rep}{{\rm Rep}}

\newcommand{\rig}{{\rm rig}}
\newcommand{\Rig}{{\rm Rig}}

\newcommand{\s}{\sqbullet}
\newcommand{\set}[2]{\left\{#1 \middle|  #2\right\}}

\newcommand{\Shv}{{\rm Shv}}

\newcommand{\sm}{{\rm sm}}

\newcommand{\Sob}{{\rm Sob}}

\newcommand{\Sp}{{\rm Sp}}
\newcommand{\Spa}{{\rm Spa}}
\newcommand{\Spec}{{\rm Spec}}
\newcommand{\Spf}{{\rm Spf}}

\newcommand{\sss}{{\rm ss}}

\newcommand{\st}{{\rm st}}
\newcommand{\Stk}{{\rm Stk}}

\newcommand{\sub}{\subset}

\newcommand{\Sym}{{\rm Sym}}

\newcommand{\Top}{{\rm Top}}

\newcommand{\unr}{{\rm unr}}

\newcommand{\vep}{{\varepsilon}}
\newcommand{\vphi}{{\varphi}}

\newcommand{\wot}{\widehat{\ot}}

\newcommand{\wt}{{\rm wt}}
\newcommand{\WT}{{\rm WT}}

\newcommand{\x}{{\times}}
\newcommand{\xra}{\xrightarrow}

\newcommand{\Z}{{\mathbb Z}}

\newcommand{\Zp}{{\mathbb Z_p}}

\begin{document}

\title{Block decompositions in the p-adic Langlands correspondence}
\author{Daniel Le}
\address{Department of Mathematics, Purdue University, 150 N. University Street, West Lafayette, IN 47907-2067}
\email{ledt@purdue.edu}
\author{Matthias Strauch}
\address{Department of Mathematics, Indiana University, Rawles Hall, Bloomington, IN 47405, U.S.A.}
\email{mstrauch@iu.edu}
\author{Zichuan Wang}
\address{Department of Mathematics, Indiana University, Rawles Hall, Bloomington, IN 47405, U.S.A.}
\email{wangzich@iu.edu}

\begin{abstract} Let $K/\Qp$ be a finite extension, $\bG$ a connected reductive group over $K$, and set $G = \bG(K)$, considered as a $p$-adic Lie group. We show first that the Bernstein center $\cC_G$ of the category of solid locally $\Qp$-analytic representations of $G$ is isomorphic to the center of the locally analytic distribution algebra $D^\la(G)$. We then consider the Emerton-Gee stack $\frX_{n,K}$ of rank-$n$ $(\vphi,\Gamma)$-modules over the Robba ring for $K$ and determine its connected components. The latter are in canonical bijection with the primitive idempotents of the ring $\cC_{\GL_n(K)}$. Moreover, under the assumption that the ring of global functions on $\frX_{n,K}$ has no non-zero locally nilpotent elements, we show that this ring is isomorphic to a Fr\'echet completion of $\cC_{\GL_n(K)}$. 
\end{abstract}

\maketitle

\tableofcontents

\section{Introduction}
\label{section1_intro}

In this paper we consider geometric structures in the $p$-adic Langlands program which are related by a localization functor whose existence has been conjectured by Emerton, Gee, and Hellmann in \cite{EmertonGeeHellmann}. Those results have consequences for the organization of the categories of representations under consideration -- on the `automorphic side' as well as on the `Galois side' -- into blocks. 

\vskip8pt

\subsection{Smooth representations, the Bernstein decomposition, and spaces of \texorpdfstring{$L$}{}-parameters}
\label{1.1_smooth}
Let $K$ be a local non-archimedean field of residue characteristic $p$, $\bG$ a connected reductive group over $K$, and set $G = \bG(K)$. To put our results into context we change the setting and begin by recalling some results in the local Langlands program for smooth representations on vector spaces over an algebraically closed field $E$ of characteristic zero.\footnote{More generally, one can consider representations on modules over a local ring $\Lambda$ with residue field of characteristic $\ell \neq p$. To do so is necessary if one wants to study deformations of such representations. This leads to the concept of a local Langlands correspondence in families. See \cite{Emerton-Helm_families} and \cite{Dat_Moduli-L-parameters} for more information and background.} Denote by $\Rep^\sm_E(G)$ the category of smooth representations of $G$ on $E$-vector spaces. When $E = \bbC$, Joseph Bernstein gave a description of the decomposition of this category into blocks: $\Rep^\sm_E(G) = \prod_\Omega \Rep_E(G)_\Omega$. Here $\Omega$ ranges over the inertial equivalence classes $[L,\rho]$ of pairs $(L,\rho)$, where $L$ is a Levi subgroup of a $K$-parabolic subgroup of $G$ and $\rho$ is a supercuspidal representation of $L$. Two such pairs $(L_1,\rho_1)$ and $(L_2,\rho_2)$ are called inertially equivalent if there is $g \in G$ and an unramified character $\psi$ of $L_2$ such that $L_2 = gL_1g^{-1}$ and $\rho_2 = \psi \ot \rho_1^g$ with $\rho_1^g(h) = \rho_1(g^{-1}hg)$ for $h \in L_2$. Any such inertial equivalence class $\Omega$ has the structure of a connected affine algebraic variety and the disjoint union of all these equivalence classes, $\Omega(G) = \coprod_\Omega \Omega$, is called the {\it Bernstein variety} of $G$.  

\vskip8pt

This decomposition is achieved by studying the center $\cC^\sm_G$ of the category $\Rep^\sm_E(G)$, also called the {\it Bernstein center}, which is defined as the ring of natural transformations of the identity functor of this category. In this way the Bernstein variety $\Omega(G)$ appears as a (coarse) parameter space for the irreducible representations of $G$, the ring of global functions on $\Omega(G)$ is the Bernstein center $\cC^\sm_G$, and the blocks of $\Rep^\sm_E(G)$ are in bijection with the connected components of $\Omega(G)$, and hence with the primitive\footnote{A nonzero idempotent $r$ in a nonzero ring $R$ is called primitive if it is impossible to write $r=r_1+r_2$ for nonzero idempotents $r_1,r_2\in R$ such that $r_1r_2=0$. If $R=\Gamma(X,\cO_X)$ for a qcqs scheme $X$, then primitive idempotents in $R$ correspond bijectively to connected components of $X$ because a connected component of $X$ is the same as a clopen subset that is connected. Although the Bernstein variety is not quasi-compact, its connected components are open, so we still have the bijective correspondence.} idempotents of $\cC^\sm_G$, cf. \cite[2.13]{Bernstein_Centre}, \cite[ch. VI]{Renard_book}. We note that $\Omega(G)$ has countably infinite connected components for any non-trivial $\bG$.

\vskip8pt

We turn now to the `Galois side', which is rather about representations of the Weil group $W_K$ of $K$.
Let $\widehat{G}=\widehat{\bG}(E)$ denote the Langlands dual group of $\bG$, and let ${}^{L}G:=\widehat{G}\rtimes W_K$ denote the Langlands $L$-group. Modifying \cite[\S7]{Vogan_TheLLC}, Haines in \cite[\S5]{Haines_stable_center} endowed the set of $\widehat{G}$-conjugacy classes of semisimple $L$-parameters $\varphi:W_K\to {}^LG$ of $G$ with the structure of an infinite disjoint union of connected affine algebraic varieties over $E$, which we denote by $\Omega^{\mathrm{st}}(G)$, adopting the notation in \cite[\S6]{Scholze-Shin_unitary}.
The ring of global functions on $\Omega^{\mathrm{st}}(G)$ is called the {\it stable Bernstein center} and is denoted by $\cC^{\sm,\st}_G$.
By \cite[Proposition 4.23]{Haines_stable_center}, an enhancement (LLC+) of the local Langlands correspondence for $G$ induces a quasi-finite map $f:\Omega(G)\to \Omega^{\st}(G)$ of algebraic varieties, which gives rise to a homomorphism $f^*: \cC_G^{\sm,\st}\to \cC^\sm_G$ that is injective when $\bG/K$ is quasi-split.
For $\bG = \GL_n$, it is known that $f$, and hence $f^*$, are isomorphisms \cite[Remark 4.25]{Haines_stable_center}, so we have a bijection between connected components of $\Omega(G)$ and $\Omega^\st(G)$.

\vskip8pt

More recently, following the work of Emerton-Helm, Hellmann \cite{Hellmann_iwahori}, Ben-Zvi, Chen, Helm, Nadler \cite{Ben-Zvi_Cat-Deligne-Langlands}, and Zhu \cite{Zhu_Stack-Langlands-parameters} considered categorical versions of the local Langlands correspondence. We recall here Hellmann's conjecture \cite[Conj. 1.1]{Hellmann_iwahori} where it is assumed that $\bG$ is split. Fix a Borel subgroup $\bB \sub \bG$, a maximal split torus $\bT$, as well as a generic character $\psi: N \ra E^\x$, where $N$ is the maximal unipotent subgroup of $B = \bB(K)$. As in \cite{Hellmann_iwahori}, we denote by $[X_{\widehat{\bG}}^{\rm WD}/\widehat{\bG}]$ the stack of Weil-Deligne representations valued in $\widehat{\bG}$. It is then conjectured that there is a fully faithful functor
\begin{numequation}\label{Hellmann-conj}
\cR^\psi_G: D^+\Big(\Rep^\sm_E(G)\Big) \lra D^+_{{\rm QCoh}}\Big([X_{\widehat{\bG}}^{\rm WD}/\widehat{\bG}]\Big) \;.    
\end{numequation}

Moreover, evidence for this conjecture is provided by giving a construction over the principal Bernstein block of $G=\GL_n(K)$, i.e., the block containing the trivial representation. 
This conjectural fully faithful functor induces a map from the ring of global functions of $[X_{\widehat{\bG}}^{\rm WD}/\widehat{\bG}]$ to the Bernstein center $\cC_G^{\sm}$.
In fact, $\Gamma([X_{\widehat{\bG}}^{\rm WD}/\widehat{\bG}],\cO_{[X_{\widehat{\bG}}^{\rm WD}/\widehat{\bG}]}) \cong \cC^{\sm,\st}_G$ by \cite[Cor. 1.8]{DHKM_L-parameters},\footnote{This result says that the GIT quotient $X_{\widehat{\bG}}^{\rm WD}/\hskip-3pt/\widehat{\bG}$ is isomorphic to $\Omega^\st(G)$.} and this conjectural map of rings should coincide with the map $f^*:\cC_G^{\sm,\st}\to \cC_G^\sm$ considered above, cf. \cite[Remark 3.7(b)]{Hellmann_iwahori}. If this is the case, then one obtains a canonical surjection from the set of connected components of $\Omega(G)$ to the set of connected components of $[X_{\widehat{\bG}}^{\rm WD}/\widehat{\bG}]$. 

\vskip8pt

\subsection{Locally analytic representations, the Bernstein center, and stacks of \texorpdfstring{$(\varphi,\Gamma)$}{}-modules} 
From now on $K$ is a finite extension of $\Qp$. This paper is about a situation similar to the one discussed in the preceding subsection, but now we consider $G = \bG(K)$ as a $p$-adic Lie group. Fix once and for all a finite extension $E/K$ and assume $\Sigma_K:=\Hom_{\Qp\mbox{-}{\rm algs}}(K,E)$ has cardinality $[K:\Qp]$. Denote by $\Rep_{E_\s}^{\la,\heartsuit}(G)$ the abelian category of solid locally analytic representations of $G$ as defined in \cite[3.3]{RR_2}.
We write $\cC_G$ for the Bernstein center of $\Rep_{E_\s}^{\la,\heartsuit}(G)$, which is our object of interest on the `automorphic side'. On the `Galois side' we consider the Emerton-Gee stack of $(\vphi,\Gamma)$-modules over the Robba ring for $K$. 

\vskip8pt

{\it Result on the Bernstein center for solid locally analytic representations.} Let $D^\la(G) = D^\la(G,E)$ be the locally $\Qp$-analytic distribution algebra of $G$, i.e., the strong dual of the topological $E$-vector space $C^\la(G,E)$ of locally $\Qp$-analytic $E$-valued functions on $G$.
In sec. \ref{section2_Bernstein}, we prove that the Bernstein center of $\Rep^{\la, \heartsuit}_{E_\s}(G)$ is canonically isomorphic to the center of $D^\la(G)$.

\vskip8pt

\begin{thm}[Theorem {\ref{thm:Bernstein_Dist}}] The canonical map 
\[Z(D^\la(G,E)) \lra \cC_G\;, \;\; z \mapsto (z\cdot\id_V:V\ra V)_V;,\]
is an isomorphism of $E$-algebras.
\end{thm}

\vskip8pt

If $\bG$ is {\it sufficiently isotropic} over $K$ (e.g. $\bG$ is split over $K$), cf. \cite[p. 55]{Kohlhaase_Invariant_dist}, then $\Res_{K/\Qp}(\bG)$ is sufficiently isotropic over $\Qp$, and Kohlhaase's results \cite[2.2.1, 2.2.4]{Kohlhaase_Invariant_dist} imply that the set of primitive idempotents in $Z(D^\la(G))$ is naturally identified with the set of primitive idempotents in $D^\la(Z(G))$.
Hence, the blocks in $\Rep^{\la, \heartsuit}_{E_\s}(G)$ are in bijection with connected components of $\Spec (D^\la(Z(G)))$.

\vskip8pt

{\it Result on connected components.} For the remainder of this introduction we have $G = \GL_n(K)$. The Amice transform defines an isomorphism 
$$
D^\la(Z(G),E) = D^\la(K^\times,E) \isom \Gamma(\cW_K)\otimes_{E} \Gamma(\bbG_m^{\alg})
$$ 
of $E$-algebras, where $\cW_K$ is the rigid space representing the functor $\Hom^{\rm cont}(\cO_K^\times, \bbG_m^{\an}(-)): \Rig_E^\op\to {\rm Set}$, and $\bbG_m^\alg = \Spec(E[T,T^{-1}])$.
In particular, by local class field theory, connected components of $\Spec(D^\la(Z(G)))$, and hence the blocks in $\Rep^{\la, \heartsuit}_{E_\s}(G)$, are indexed by characters of the roots of unity $(\cO_K^\times)_{\rm tors} = \mu(K)$ in $K$, which is finite.

\vskip8pt 

On the other hand, $\Gamma(\cW_K\times \bbG_m^\rig)$ is isomorphic to the ring $\Gamma(\frX_{1,K})$ of global functions on the rigid analytic moduli stack $\frX_{1,K}$ of $(\varphi,\Gamma_K)$-modules of rank $1$.
Therefore, Bernstein blocks in $\Rep^{\la, \heartsuit}_{E_\s}(\GL_n(K))$ are in bijection with connected components of $\frX_{1,K}$.
We study the connected components of various Emerton-Gee stacks.
Let $\Rig_E$ denote the category of rigid analytic spaces over $\Sp(E)$, which can be endowed with the Tate-\'{e}tale topology or the Tate-fpqc topology \cite[2.1]{ConradTemkinAnalytification}.
Let $\cX_n$ be the formal algebraic Emerton-Gee stack \cite{EmertonGee_modulicrystalline}, and let $\cX^\rig_n$ be its rigid analytic fiber as in sec. \ref{subsection3.2_genericfiber}, which is a Tate-fpqc Artin stack on $\Rig_E$ by \cite{Oh-Shimizu_moduli} and Theorem \ref{thm:Artin_etale=fpqc}.
Let $\frX_n$ be the rigid analytic Emerton-Gee stack \cite[5.1]{EmertonGeeHellmann} on the Tate-fpqc site $\Rig_E$.

\vskip8pt 

In sec. \ref{section3_Prelim}, we begin with preliminaries on stacks. Our first goal is to introduce and study the {\it rigid generic fiber} $\cX^\rig$ for any formal algebraic stack $\cX$ locally formally of finite type over $\cO_E$.\footnote{As above, and throughout this paper, we assume that $E$ is a finite extension of $K$ with $|\Hom_{\Qp\mbox{-}{\rm algs}}(K,E)|=[K:\Qp]$. The Emerton-Gee stack is a priori a stack over $\Spf(\Zp)$. The connected components of this stack over $\Spf(\Zp)$ (and its variants) are, however, in general not geometrically connected, as our discussion implicitly shows.} 
Such a notion has been introduced in \cite[(5.1.12)]{EmertonGeeHellmann}, where Proposition 5.1.13 in {\it loc. cit.} shows that $\cX^\rig$ is a Tate-fpqc rigid analytic Artin stack under the assumption that $\cX$ is of topologically finite type.
However, the Emerton-Gee stack $\cX_{n,K}$ is not topologically of finite type (\cite[6.5.2]{EmertonGee_modulicrystalline}), but locally formally of finite type over $\cO_E$.
We want to know that $\cX_{n,K}^\rig$ is a Tate-fpqc Artin stack on $\Rig_E$.
Therefore following \cite{Oh-Shimizu_moduli}, in sec. \ref{subsection3.2_genericfiber} we define the {\it rigid generic fiber} $\cX^\rig$ for any formal algebraic stack $\cX$ locally formally of finite type over $\cO_E$, which, by the arguments of \cite{Oh-Shimizu_moduli}, is an Artin stack on $\Rig_E$ in the {\it Tate-\'{e}tale} topology.
It is straightforward to see that the Tate-\'{e}tale topology on $\Rig_E$ is generated by \'{e}tale strongly surjective morphisms of rigid spaces \cite[2.1]{Hu96}, which is coarser than the Tate-fpqc topology on $\Rig_E$.
Using adic spaces (\cite{Scholze_diamond}, \cite{Bertapelle_Weil}, \cite{Wahle_Weil}) and Berkovich spaces (\cite{Ducros_families}, \cite{Daylies_descente}), we prove that any Tate-\'{e}tale quasi-analytic space (resp. Artin stack) on $\Rig_E$ is already a Tate-fpqc quasi-analytic space (resp. Artin stack) on $\Rig_E$ in Theorem \ref{thm:qan_etale=fpqc} (resp. Theorem \ref{thm:Artin_etale=fpqc}).
Therefore, the {\it rigid generic fiber} $\cX^\rig$ of a formal algebraic stack $\cX$ locally formally of finite type over $\cO_E$ is a Tate-fpqc Artin stack on $\Rig_E$.
\vskip8pt

In sec. \ref{subsection4.1_formalcomp}, we determine the set $\pi_0(\cX_{n,K})$ of connected components of the formal algebraic Emerton-Gee stack $\cX_{n,K}$ over $\Spf(\cO_E)$ whose groupoid of $A$-valued points is the groupoid of projective \'{e}tale $(\varphi,\Gamma_K)$-modules of rank $n$ with $A$-coefficients, for any $p$-adically complete $\cO_E$-algebra $A$, cf. \cite[\S1.3, \S2.2]{EmertonGee_modulicrystalline}. 

\vskip8pt

\begin{thm}[{Theorem \ref{thm:reduced-EG}}]\label{thm1}
The set of connected components of $\cX_{n,K}$ is in bijection with the set of connected components of $\cX_{1,K}$ via the determinant map, which is in bijection with the Pontryagin dual of the prime-to-$p$ roots of unity in $K$ via local class field theory.
\end{thm}

\vskip8pt

More precisely, using the local geometry of $\cX_{n,K}$ (\cite{CEGS_ANT, CEGS_LocalGeometry, Kansal_Codim-1-intersections}) we show that $\det:\cX_{n,K}\to \cX_{1,K}$ induces a bijection on connected components.\footnote{For the results \ref{thm1}, \ref{thm2}, \ref{thm3} we suppose that $p$ is odd, as we make use of other papers which make this assumption.}  
Hence, by \cite[7.2.17]{EmertonGee_modulicrystalline}, for a connected formal scheme $\Spf(A)$, two $(\varphi,\Gamma_K)$-modules $D,D' \in \cX_{n,K}(\Spf(A))$ lie in the same connected component if and only if 
$$
\overline{\det(D)}|_{\mu(K)^p} = \overline{\det(D')}|_{\mu(K)^p},
$$ 
where $\overline{\det(D)}$ and $\overline{\det(D')} \in \cX_{1,K}(\Spec(A/\varpi_EA))$ are the mod-$\varpi_E$ reductions of the determinants, and we use the local class field theory to identify the prime-to-$p$ torsions in $I_{K}^{\ab}\cong \cO_K^\times$ with the group of prime-to-$p$ roots of unity $\mu(K)^p$ in $K$.

\vskip8pt

In sec. \ref{subsection4.2_rigidcomp}, we determine the set $\pi_0(\cX_{n,K}^\rig)$ of connected components of the rigid generic fiber $\cX_{n,K}^\rig$ as Tate-fpqc stack on $\Rig_{E}$.

\vskip8pt

\begin{thm}[{Theorem \ref{thm:rigid-EG-general}}]\label{thm2}
The set of connected components of $\cX_{n,K}^\rig$ is in bijection with the set of connected components of $\cX_{1,K}^\rig$ via the determinant map, which is in bijection with the Pontryagin dual of the roots of unity in $K$ via local class field theory.
\end{thm}

\vskip8pt

More precisely, $\det^{\rig}:(\cX_{n,K})^\rig\to (\cX_{1,K})^\rig$ induces a bijection on the connected components.
Hence, for a connected affinoid $E$-algebra $A$, two \'{e}tale $(\varphi,\Gamma_K)$-modules $D,D'\in \cX_{n,K}^\rig(\Sp(A))$ lie in the same connected component if and only if $$\det(D)|_{(I_K^\ab)_{\rm tors}} = \det(D')|_{(I_K^\ab)_{\rm tors}}.$$
Since $(I_K^\ab)_{\rm tors}\cong \mu(K)$ for the group $\mu(K)$ of roots of unity in $K$, $\pi_0(\cX_{n,K}^\rig)$ is in bijection with the characters of $\mu(K)$.
Roughly, our proof relies on \cite{BIP_def}'s result on $\cO_E$-flatness and normality of certain rings $R$ closely related to local Galois deformation rings.
These rings $R$ are versal rings to certain connected substacks $\cX$ of $\cX_{n,K}$ at finite type points such that $\cX^\rig$ are connected components of $\cX_{n,K}^\rig$.
On the other hand, if locally the formal algebraic stack $\cX$ is smoothly presented as the coequalizer $\Spf(B)\rightrightarrows \Spf(A)$ of two affine formal schemes formally of finite type, then locally $\cX^\rig \simeq \coeq(\Spf(B)^\rig\rightrightarrows\Spf(A)^\rig)$ and the completed local rings of $\Spf(B),\Spf(A)$ at closed points are also versal rings to $\cX$ at the corresponding finite type points. As two versal rings at the same point are smoothly equivalent, $\Spf(B)$ and $\Spf(A)$ are $\cO_E$-flat and normal.
By \cite[7.4.1]{deJongCrystalline}, the ring of formal functions on $\cX$ is isomorphic with the ring of analytic functions on $\cX^\rig$ whose spectral norms are bounded by $1$.
In particular, $\cX^\rig$ is connected if and only if $\cX$ is connected.

\vskip8pt

Finally, since affinoid algebras are Jacobson, by passing to the associated graded of Kedlaya's slope filtration for $(\varphi,\Gamma_K)$-modules with coefficients in finite extensions of $\bbQ_p$ and by twisting, we show that the morphism $\pi_n:\cX_{n,K}^\rig \to \frX_{n,K}$ of ``forgetting the \'{e}tale lattice" induces a bijection on the connected components.

\vskip8pt

\begin{thm}[{Corollary \ref{cor:comp_Robba}}]\label{thm3}
The set of connected components of $\frX_{n,K}$ is in bijection with the set of connected components of $\cX_{n,K}^\rig$ via the morphism $\pi_n$, and hence with the Bernstein blocks in $\Rep_{E_\s}^{\la,\heartsuit}(\GL_n(K))$ via $\det$ and a canonical isomorphism which identifies $\Gamma(\frX_{1,K})$ with the Fr\'echet completion of $\cC_{\GL_1(K)}$.
\end{thm}

\vskip8pt

{\it Result on global functions.} We continue with $G = \GL_n(K)$. In sec. \ref{section5_globalfunc}, we determine (up to locally nilpotent functions) the ring of global functions $\Gamma(\frX_{n,K})$ in four steps, following \cite[Conjecture 5.1.36]{EmertonGeeHellmann}.

\vskip8pt

\begin{thm}[{Theorem \ref{thm:EG_fcns}}]
(a) There is a canonical morphism of $E$-algebras $\gamma:\cC_{G}=Z(D^\la(G,E))\to \Gamma(\frX_n,\cO_{\frX_n})$.

\vskip5pt

(b) Assume the Emerton-Gee stack $\frX_{n,K}$ has no nonzero locally nilpotent global functions (Assumption \ref{hypo:reduced}).
Then $\gamma$ induces an isomorphism between a Fr\'echet completion of $\cC_G$ and $\Gamma(\frX_n,\cO_{\frX_n})$.
\end{thm}

\vskip8pt

Firstly, let us consider the image of $\Gamma(\frX_{n,K})$ in $\Gamma(\frX_B)$ under the pullback map along the map $\beta_B:\frX_B\to \frX_{n,K}$ forgetting the choice of a triangulation on a trianguline $(\varphi,\Gamma_K)$-module.
By Corollary \ref{cor:comp_Robba} and \cite[Remark 5.3.15]{EmertonGeeHellmann}, $\beta_B(\frX_B)$ is Zariski-dense in $\frX_{n,K}$.
Hence, the image of $\Gamma(\frX_{n,K})$ in $\Gamma(\frX_B)$ is isomorphic to $\Gamma(\frX_{n,K})$ modulo locally nilpotent functions.
Passing from a triangulated $(\varphi,\Gamma_K)$-module to the associated graded $(\varphi,\Gamma_K)$-module, we get an isomorphism $\Gamma(\frX_B)\cong \Gamma(\frX_T)$, where $\frX_T \cong (\frX_{1,K})^n$.
Functions on $\frX_T$ can be made explicit: since
$
\Gamma(\frX_T) = \Gamma((\frX_{1,K})^n) = \Gamma(\cW_K^n\times \bbG_m^{\an,n}) = \Gamma((\bbG_m^{\an,n})_{\cW_K^n})\,,
$
any function $f \in \Gamma(\frX_T)$ can be uniquely expanded into a Laurent series 
$$
f = \sum_{\underline{m}\in \bbZ^n} c_{\underline{m}}\cdot \underline{\alpha}^{\underline{m}}\,,
$$
where $\underline{m}=(m_1,\dots,m_n)\in \bbZ^n$, $c_{\underline{m}}\in \Gamma(\cW_K^n)$, and we write $\alpha_1,\dots,\alpha_n$ for coordinates on $\bbG_m^{\an,n}$ so that $\underline{\alpha}^{\underline{m}}:= \alpha_1^{m_1}\cdots \alpha_n^{m_n}$ is the usual multi-index notation.
Secondly, we apply results obtained in \cite{BHS3, EmertonGeeHellmann} on local companion points and compactifications of $\frX_B$ to deduce {\it symmetries} satisfied by the coefficients $c_{\underline{m}}$ when $f$ comes from a function in $\Gamma(\frX_{n,K})$, e.g., $c_{\underline{m}}=0$ if $\underline{m}\in \bbZ^n \setminus \bbZ(1,\dots,1)$ by Lemma \ref{lem:diagonal}, and for $c_{(m,\dots,m)}$ with $m\in \bbZ$, it satisfies symmetries in Sen weights \eqref{eqn:sym_Sen} and symmetries in smooth parts \eqref{eqn:sym_smooth}.
Thirdly, we observe that the derivative map $\wt: \cW_K\to (\bbA_E^{1,\an})^{\Sigma_K}, \delta \mapsto (\wt_\sigma(\delta))_{\sigma\in \Sigma_K}$, is locally surjective after passing to a finite \'{e}tale cover, which quickly implies that the map 
$$
\Gamma(\WT_{n,E}\times_{\WT_{1,E}} \frX_{1,K}) \xra{(\omega_n,\det)^*} \Gamma(\frX_{n,K})
$$ is injective, where $\omega_n:\frX_{n,K} \ra \WT_{n,E}:=\Res_{K/\bbQ_p}(\mathfrak{gl}_n//\GL_n)_E \cong \prod_{\sigma:K\hookrightarrow E} \bbA^{n,\an}_E$ maps a $(\varphi,\Gamma)$-module $D$ to the characteristic polynomial of its Sen operator $\Theta_D$ \cite[(5.1.35)]{EmertonGeeHellmann}.
Finally, we deduce from the \'{e}tale-local surjectivity of $\wt$ and \eqref{eqn:sym_smooth} that the coefficient $c_m:=c_{(m,\dots,m)}\in \Gamma(\cW_K^n)$ descends along the map 
$$
(\wt_1,\dots,\wt_n,\det): \cW^n_K\to (\bbA^{n,\an}_E)^{\Sigma_K}\times_{(\bbA^{1,\an}_E)^{\Sigma_K}} \cW_K.
$$
Then, \eqref{eqn:sym_Sen} implies that $c_m\in \Gamma(\cW^n_K)$ further descends along the quotient map 
$$
(\bbA^{n,\an}_E)^{\Sigma_K}\times_{(\bbA^{1,\an}_E)^{\Sigma_K}} \cW_K \twoheadrightarrow ((\bbA^{n,\an}_E)^{\Sigma_K}\times_{(\bbA^{1,\an}_E)^{\Sigma_K}} \cW_K)//S_n^{\Sigma_K} \cong \WT_{n,E}\times_{\WT_{1,E}} \cW_K
$$
to a function $\frc_{m}\in \Gamma(\WT_{n,E}\times_{\WT_{1,E}} \cW_K)$. 
Consequently, the injective map 
$$
\Gamma((\WT_{n,E}\times_{\WT_{1,E}} \cW_K) \times \bbG_m^\an) \cong \Gamma(\WT_{n,E}\times_{\WT_{1,E}} \frX_{1,K}) \xra{(\omega_n,\det)^*} \Gamma(\frX_{n,K})
$$ 
is surjective (and hence bijective) modulo locally nilpotent functions on $\frX_{n,K}$.
Since the Fr\'echet completion of $\cC_G$ is $\Gamma(\WT_{n,E}\times_{\WT_{1,E}} \frX_{1,K})$ by \cite[2.2.1]{Kohlhaase_Invariant_dist}, we define 
$$
\gamma: \cC_G \hookrightarrow \Gamma(\WT_{n,E}\times_{\WT_{1,E}} \frX_{1,K}) \xrightarrow{(\omega_n,\det)^*}\Gamma(\frX_{n,K})
$$
to be the composite, referring the reader to sec. \ref{section6_comparison} for details.

\vskip8pt

{\it Connection via EGH functor.}
The conjectural fully faithful embedding \eqref{Hellmann-conj} has analogues in the Banach case \cite[6.1.15]{EmertonGeeHellmann} and the locally analytic case \cite[6.2.4]{EmertonGeeHellmann}, where on the `automorphic side' we consider finitely presented smooth $\cO_E$-linear representations of $G$ \cite[6.1.1]{EmertonGeeHellmann} and finitely presented solid locally analytic $E$-linear representations of $G$, respectively, while on the `Galois side' we consider bounded complexes of coherent sheaves on the moduli stacks $\cX_{n,K}$ and $\frX_{n,K}$ of $(\varphi,\Gamma_K)$-modules, respectively.
In the locally analytic case, we expect the conjectural functor described in \cite[6.2.4]{EmertonGeeHellmann} 
$$
\frA^\rig_{\GL_n(K)}: D^b_{\rm f.p.}\Big(\Rep^\la_{E_\s}(\GL_n(K))\Big) \lra D^b_{\rm coh}(\frX_{n,K})
$$
to be linear with respect to the action of the Bernstein center $\cC_G$ on the domain and the action of $\Gamma(\frX_{n,K})$ on the target via the morphism $\gamma:\cC_G \ra \Gamma(\frX_{n,K})$ defined above. 
This functor is, unlike the situation in the smooth case, faithful but not full in general, as one can see already in the case of $\GL_1(K)$ \cite[Remark 7.1.11]{EmertonGeeHellmann}. One can speculate that it is fully faithful on a subcategory of `tempered representations', whose Bernstein center is a Fr\'echet completion of $\cC_G$.

\section{The Bernstein center of \texorpdfstring{$\operatorname{Rep}^{\la, \heartsuit}_{E_\s}(G)$}{}}
\label{section2_Bernstein}

\subsection{Solid locally analytic representations}

\begin{para}\label{analytic-groups} {\it Analytic groups.} In this section, let $G$ be any locally $\Qp$-analytic group, by which we mean a group object in the category of $\Qp$-varieties of class $C^\omega$ in the sense of \cite[\S5.1.5]{Bourbaki_Varietes_1-7}. By \cite[III, p. 5]{Bourbaki-TopGen}, such a group is Hausdorff, and by \cite[8.7, 18.8]{SchneiderBookLie} it is strictly paracompact, which means that every open covering has a refinement by disjoint open subsets.
\end{para}

\vskip8pt

\begin{para}\label{loc-an-fcs-and-distributions} {\it Locally analytic functions and distributions.} Following \cite{RR_2}, we write $\Mod(E_\s)$ for the stable $\infty$-category of solid $E$-vector spaces, and write $\Mod^\heartsuit(E_\s)$ for the heart, which is the abelian category of solid $E$-vector spaces.
If $T$ is a topological $E$-vector space, we associate to it a condensed $E$-vector space $\underline{T}$ by $\underline{T}(S)=C^0(S,T)$ for any profinite set $S$, where $C^0(S,T)$ is the $E$-vector space of continuous functions from $S$ to $T$. The restriction of the functor $T \mapsto \underline{T}$ is fully faithful on the subcategory of Fr\'echet spaces.\footnote{Fr\'echet spaces are metrizable, and hence compactly generated. By \cite[Prop. 1.7]{Scholze_condensed} the functor $T \mapsto \underline{T}$ is fully faithful on the category of compactly generated topological spaces.} To simplify, we generally suppress the underline from the notation.

\vskip8pt

Let $C^\la(G,E)$ be the topological algebra of locally $\bbQ_p$-analytic $E$-valued functions on $G$.
Let $D^\la(G,E):=C^\la(G,E)'_b$ be the locally analytic distribution algebra of $G$, which is the strong dual of $C^\la(G,E)$.
Choose a compact open subgroup $G_0\subset G$.
We have 

\[C^{\la}(G,E)=\prod_{g\in G/G_0} C^\la(gG_0,E) \quad\text{ and }\quad  D^\la(G,E)=\bigoplus_{g\in G/G_0} \delta_gD^\la(G_0,E),\]

cf. \cite[\S2]{SchneiderTeitelbaum_JAMS_locandist} and \cite[\S5, E1]{NFA}.
Then, $\underline{C^\la(G,E)} = \prod_{g\in G/G_0} \underline{C^\la(gG_0,E)}$ because $\underline{(-)}$ is a right adjoint, and $\underline{D^\la(G,E)}=\bigoplus_{g\in G/G_0} \underline{\delta_gD^\la(G_0,E)}$ as the image of any continuous map from a profinite set $S$ to the locally convex direct sum is bounded, and thus lies in a finite direct sum by the proof of \cite[Proposition 7.12(iii)]{NFA}.
Both $C^\la(G,E)$ and $D^\la(G,E)$ are solid $E$-vector spaces, since $C^\la(G_0,E)$ and $D^\la(G_0,E)$ are solid by \cite[Proposition 3.7]{RR_1} and \cite[Proposition 16.10]{NFA}.
Moreover, $D^\la(G,E) \cong C^\la(G,E)^\vee$ and $D^\la(G,E)^\vee \cong C^\la(G,E)$ by \cite[2.1.10(2), 2.1.11]{RR_2}, where $V^\vee := \underline{\Hom}_E(V,E)$ for any $V\in \Mod^\heartsuit(E_\s)$.

\vskip8pt

Let $\frg$ denote the Lie algebra of $G_0$. 
Let $\cL_0$ be a $\bbZ_p$-lattice of $\frg$ such that $[\cL_0,\cL_0]\subset p\cL_0$.
For any $h\in \bbQ_{\geq 0}$, the lattice $\cL_h:=p^h\cL_0 \subset \frg_{\overline{\bbQ}_p}$ also satisfies $[\cL_h,\cL_h]\subset p\cL_h$, and we can integrate it to an analytic group $\bbG_{h} = \bbG_{\cL_h}$ over $\bbQ_p$ whose underlying space is the polydisc of radius $p^{-h}$, cf. \cite[2.1.1]{RR_2}.
Let $\stackrel{\circ}{\bbG}_h := \bigcup_{r\in \bbQ_{>0}} \bbG_{p^r \cL_h} = \bigcup_{r\in \bbQ_{>0}} \bbG_{h+r}$ be the wide open polydisc of radius $p^{-h}$, and define $\bbG^{(h^+)}:=\coprod_{g\in G_0/\stackrel{\circ}{\bbG}_h(\bbQ_p)}g \stackrel{\circ}{\bbG}_h$.
Let $C^h(G_0,\bbQ_p):=\cO(\bbG^{(h^+)})$
be the space of analytic functions with values in $\bbQ_p$ on the Stein $\bbQ_p$-analytic group $\bbG^{(h^+)}$, and let $D^h(G_0,E)$ denote the strong dual of $C^h(G_0,E):=C^h(G_0,\bbQ_p)\otimes_{\bbQ_p} E$, which is also the underlying topological space of the condensed $E$-linear dual of $\underline{C^h(G_0,E)}$ by \cite[2.1.2(1), 2.1.6, 3.1.6]{RR_2}.
By \cite[2.1.5 (3),(4)]{RR_2}, we have 
$$
D^\la(G_0,E)=\varprojlim_{h\in \bbQ_{\geq 0}} D^h(G_0,E)
$$
with injective transition maps $D^{h'}(G_0,E) \hookrightarrow D^{h}(G_0,E)$ for any rational $h'>h\geq 0$.
By the injectivity of transition maps, we deduce that the canonical injective map
\begin{align}
\label{eqn:Dlalim}
\begin{split}
D^\la(G,E) 
&= 
\bigoplus_{g\in G/G_0} \delta_g D^\la(G_0,E)     
=
\bigoplus_{g\in G/G_0}  \delta_g \varprojlim_{h} D^h(G_0,E) 
\\&\hookrightarrow
\varprojlim_h \left(\bigoplus_{g\in G/G_0}  \delta_g D^h(G_0,E) \right)
=
\varprojlim_h \left( D^\la(G,E)\otimes_{D^\la(G_0,E)} D^h(G_0,E) \right)
\end{split}
\end{align}
is actually an isomorphism $D^\la(G,E) \cong \varprojlim_h \left( D^\la(G,E)\otimes_{D^\la(G_0,E)} D^h(G_0,E) \right)$.
Let $D_h:=D^\la(G,E)\otimes_{D^\la(G_0,E)} D^h(G_0,E)$ for $h\geq 0$.
Then, $D^\la(G,E)\cong \varprojlim_h D_h$.

\end{para}

\vskip8pt

\begin{para}{\it Locally analytic representations.} Let $\Mod_{E_\s}(D^\la(G,E))$ be the stable $\infty$-category of solid $D^\la(G,E)$-modules, and 
$$
\Rep^{\la}_{E_\s}(G) \subset \Mod_{E_\s}(D^\la(G,E))
$$ 
be the stable $\infty$-category of solid locally analytic representations of $G$, i.e., the full subcategory consisting of those $V$ for which the map $V^{\mathrm{R}\la}\ra V$ is an equivalence \cite[3.2.3]{RR_2}. 
By \cite[Theorem A]{RR_2}, $\Rep^{\la}_{E_\s}(G)$ carries the induced $t$-structure from $\Mod_{E_\s}(D^\la(G,E))$ and it is equivalent to the derived category of its heart.
We call this heart the abelian category of solid locally analytic representations of $G$, and denote it by $\Rep^{\la,\heartsuit}_{E_\s}(G)$.

\begin{example}
Given a profinite set $S$, 
written as the inverse limit $S=\varprojlim_i S_i$ of finite sets $S_i$ over an index set $I$, the associated ``free solid $E$-vector space over $S$'' is defined by 
$$
E_\s[S] := (\varprojlim_i \cO_{E}[S_i])[1/p]\,.
$$
If $S$ is extremally disconnected, then $E_\s[S]$ is known to be a projective object in the category of solid $E$-vector spaces, and those form a collection of compact projective generators. 
By \cite[3.2.6(1), 3.3.4]{RR_2}, $\Rep^{\la,\heartsuit}_{E_\s}(G)$ admits a collection of generators  
$$
D_{h,S} := D_h \otimes_{E_\s} E_\s[S] = D^\la(G,E)\otimes_{D^\la(G_0,E)} D^h(G_0,E)\otimes_{E_\s} E_\s[S],
$$
with $h\in \bbQ_{\geq 0}$ and $S$ extremally disconnected sets.
\end{example}

\end{para}

\vskip8pt

\subsection{The Bernstein center and the center of the distribution algebra}

\begin{para}\label{central-distributions} {\it From the center of the distribution algebra to the Bernstein center.} 
For any central distribution $z \in Z(D^\la(G,E))$ and $V\in \Rep_{E_\s}^{\la,\heartsuit}(G)\subset \Mod_{E_\s}(D^\la(G,E))$, the action map
$$
\rho_V^{(z)}: V\xra{z\cdot\id_V}V
$$
is a $D^\la(G,E)$-linear endomorphism of $V$.
Let 
$$
\cC_G := \End(\id_{\Rep_{E_\s}^{\la,\heartsuit}(G)})
$$ 
be the Bernstein center of the abelian category of solid locally analytic $G$-representations on $E$-vector spaces.
Since the construction of $\rho_V^{(z)}$ is natural in $V$, we have an $E$-algebra homomorphism
$$
\rho^{(-)}: Z(D^\la(G,E))\longrightarrow \cC_G, \quad z\mapsto (\rho^{(z)}_V:V \to V)_{V\in \Rep_{E_\s}^{\la,\heartsuit}(G)}\,.
$$
\end{para}

\begin{thm}
\label{thm:Bernstein_Dist}
Consider the map
$$
z_{(-)}: \cC_G\ra D^\la(G,E)  \underset{\eqref{eqn:Dlalim}}{\cong} \varprojlim_h D_h
, \quad \eta\mapsto z_\eta:=\varprojlim_h (\eta_{D_h}(1\otimes 1))
$$
that sends a natural transformation $\eta:\id\Rightarrow\id$ to the element $\varprojlim_h (\eta_{D_h}(1\otimes 1))\in \varprojlim_h D_h$.

\vskip8pt

(i) The map $z_{(-)}$ factors through $Z(D^\la(G,E))$ to become an $E$-algebra homomorphism
$$
z_{(-)}: \cC_G\xrightarrow{\simeq} Z(D^\la(G,E))\,,
$$
such that one has $\eta_V = \rho_V^{(z_\eta)}$ for all $V\in \Rep_{E_\s}^{\la,\heartsuit}(G)$.

\vskip8pt

(ii) The latter map $z_{(-)}$ is the inverse to the $E$-algebra map $\rho^{(-)}: Z(D^\la(G,E)) \rightarrow \cC_G$ introduced above.
\end{thm}

\begin{proof}
(i) 
Given any natural transformation $\eta: \id \Rightarrow \id$ in the Bernstein center $\cC_G$, and for any $h\in \bbQ_{\geq 0}$, consider $z_h := \eta_{D_h}(1\otimes 1) \in D_h$.
Since $D^\la(G,E)$ is free as $D^\la(G_0)$-module, the transition map $D_{h'} \hookrightarrow D_{h}$ is a $D^\la(G,E)$-linear injection for any $h'>h$, and hence by the naturality of $\eta$, we have an element $z_\eta:= \varprojlim_h (z_h) \in D^\la(G,E)$ via the bijection \eqref{eqn:Dlalim}, and the projection map $\pr_h:D^\la(G,E)\to D_h$ is injective for any $h\geq 0$.

\vskip8pt

Consider $E_\s[S]$ for an extremally disconnected set $S=\varprojlim_i S_i$.
By definition, $E_\s[S]$ contains the $\cO_{E,\s}$-lattice $\varprojlim_i \cO_E[S_i]$, so $E_\s[S]$ is Smith \cite[Definition 3.2(2)]{RR_1}.
In particular, $E_\s[S]$ is quasiseparated, hence flat over $E_\s$ \cite[3.20, 3.21]{RR_1}.
By \cite[Remark 4.9]{RR_1}, $D^h(G_0,E)$ is a filtered union of Smith subspaces.
By \cite[Lemma 3.8(2), Proposition 3.12]{RR_1}, the solid $E$-tensor product of two Smith spaces is Smith.
Thus by flatness, $D^h(G_0,E)\otimes_{E_\s} E_\s[S]$ is a filtered union of Smith subspaces.
Finally, 
\begin{align*}
D_{h,S}
= D_h\otimes_{E_\s} E_\s[S] &= D^\la(G,E)\otimes_{D^\la(G_0,E)} \left(D^h(G_0,E)\otimes_{E_\s} E_\s[S]\right)  
\\&=
\bigoplus_{g\in G/G_0}\delta_g \left(D^h(G_0,E)\otimes_{E_\s} E_\s[S]\right)  
\end{align*}
is a filtered union of Smith subspaces, and hence quasiseparated \cite[3.20]{RR_1}.

\vskip8pt

Since any quasiseparated condensed set $X$ is a filtered colimit of compact Hausdorff spaces with closed embeddings as transition maps by \cite[Proposition 1.2(4)]{Scholze_analytic}, the unit map of the adjunction $(*)_{\rm top}: \mathrm{Cond}({\rm Set})\rightleftarrows \Top:\underline{(-)}$ induces a monomorphism $X\hookrightarrow \underline{X(*)_{\rm top}}$ \cite[p. 9]{Scholze_analytic}.
Hence, for any condensed set $Y$ and quasiseparated $X$, the map
\begin{equation}
\label{eqn:qs_inj}
\Hom_{\mathrm{Cond}({\rm Set})}(Y,X) \hookrightarrow \Hom_{\mathrm{Cond}({\rm Set})}(Y, \underline{X(*)_{\rm top}}) \cong  \Hom_{\Top}(Y(*)_{\rm top},X(*)_{\rm top})    
\end{equation}
sending a map $f:Y\to X$ of condensed sets to the map $f(*):Y(*)\to X(*)$ is injective.

\vskip8pt

For the quasiseparated solid locally analytic representation $V:=D_{h,S}$, we claim that the endomorphism $\eta_V:V\to V$ is induced by the action of $z_\eta \in D^\la(G,E)$. 
In particular, when $S=\{*\}$ so that $V=D_h$, from the injectivity of $\pr_h:D^\la(G,E)\to D_h$ we can deduce that $z_\eta \in Z(D^\la(G,E))$ is central.
By \eqref{eqn:qs_inj}, it suffices to show that $\eta_V(*):V(*)\to V(*)$ is given by the action of $z_\eta \in D^\la(G,E)$, from which we can check on each vector in $V(*)$.
Let $v\in V(*)$ be a nonzero vector.
By \cite[3.2.8]{RR_2}, $V^{\rm Rla} = \varinjlim_h \mathrm{R}\underline{\Hom}_{D^\la(G_0,E)}(D^h(G_0,E),V)$ as solid $G_0$-representations.
Taking the $0$-th cohomology sheaf, which commutes with the filtered colimit as the filtered colimit is exact, we get that
$$
V= V^\la =  \varinjlim_h \underline{\Hom}_{D^\la(G_0,E)}(D^h(G_0,E), V).
$$
Then, as $E$ is a compact object in $\Mod_{E_\s}$, we have
\begin{align*}
V(*) &= \Hom_{\Mod_{E_\s}^\heartsuit} (E,V)
= \Hom_{\Mod_{E_\s}^\heartsuit} (E,\varinjlim_h \underline{\Hom}_{D^\la(G_0,E)}(D^h(G_0,E), V))
\\&= \varinjlim_h \Hom_{\Mod_{E_\s}^\heartsuit} (E,\underline{\Hom}_{D^\la(G_0,E)}(D^h(G_0,E), V))
\\&= \varinjlim_h \Hom_{D^\la(G_0,E)}(D^h(G_0,E), V) 
\\&= \varinjlim_h \Hom_{D^\la(G,E)}(D^\la(G,E)\otimes_{D^\la(G_0,E)}D^h(G_0,E), V) = \varinjlim_h \Hom_{D^\la(G,E)}(D_h, V),
\end{align*}
under which $v\in V(*)$ corresponds to a map $f_v:D_h\to V$ of solid $D^\la(G,E)$-modules such that $f(*)(1\otimes 1) = v\in V(*)$.
The naturality of $\eta$ applied to $f_v$ shows that $\eta_V(*)(v)=z_\eta v$.

\vskip8pt

Since we have shown that $\eta_V$ is given by the action of $z_\eta\in Z(D^\la(G,E))$ on the generators $V=D_{h,S}$, by naturality, $\eta_V$ is given by the action of $z_\eta\in Z(D^\la(G,E))$ for every $V\in \Rep^{\la,\heartsuit}_{E_\s}(G)$.
By injectivity of $\pr_h:D^\la(G,E)\to D_h$, the map $z_{(-)}:\cC_G\to Z(D^\la(G,E))$ is an injective $E$-algebra homomorphism. 
\vskip8pt

(ii) The map $z_{(-)}:\cC_G\to Z(D^\la(G,E))$ is also surjective, and hence bijective, since every $z\in Z(D^\la(G,E))$ acts on $D^\la(G,E)$-modules by $\rho^{(z)}$, which is inverse to $z_{(-)}$.
\end{proof}

\section{Preliminaries on stacks}
\label{section3_Prelim}

\subsection{Some definitions}

\begin{para}{\it Groupoids and sheaves \cite[3.4.4]{Olsson_Alg-spaces-stacks}, \cite[2.4.3]{Moret-Bailly_Laumon}.} Let $\mathcal{C}$ be a site. A \textbf{groupoid in sheaves on $\mathcal{C}$} is a quintuple $(U,R,s,t,c)$, where $U,R\in \Shv(\cC)$ are sheaves on $\cC$ and $s,t:R\to U$, $c: R\times_UR \to R$ are morphisms of sheaves such that for any object $T\in \cC$, the evaluation $(U(T),R(T),s,t,c)$ on $T$-valued points defines a groupoid, cf. \cite[\href{https://stacks.math.columbia.edu/tag/0230}{Tag 0230}]{stacks-project}, that we denote by $[U(T)/R(T)]$. 
From a groupoid $(U,R,s,t,c)$ in sheaves on $\cC$, we define a category fibered in groupoids $[ U/_p R ] \to \mathcal{C}$ whose objects are pairs $(T,x)$ with $T\in \cC$ and $x\in U(T)$, and whose morphisms from $(T,x)$ to $(T',x')$ are pairs $(f,\phi)$ with $f\in \Hom_{\cC}(T,T')$ and $\phi \in \Hom_{[U(T)/R(T)]}(x,f^*(x'))$.
We denote by $[ U / R ]\to \mathcal{C}$ the stackification of $[U/_p R]$, and call it the \textbf{quotient stack}. 
\end{para}
\begin{para}{\it Sheaf attached to a stack \cite[3.19]{Moret-Bailly_Laumon}.}
\label{para:coarse_sheaf}
Given a stack $\frX$ on a site $\cC$, we can take the \textbf{coarse sheaf $F_{\frX}$ associated to $\frX$}, i.e., the sheafification of the presheaf $F_\frX^{\rm pre}$ 
$$
T\mapsto F_\frX^{\rm pre}(T):= \mathrm{Ob}(\frX(T))/\simeq \,,
$$
of isomorphism classes, where two objects $x,y\in \frX(T)$ are identified if there is a morphism $f:x\to y$ in $\frX(T)$.
Then, $F_{(-)}:\Stk(\cC)\to \Shv(\cC)$ is the left adjoint to the ``inclusion'' functor $i: \Shv(\cC)\to \Stk(\cC)$ that naturally attaches a stack to a sheaf, \cite[(3.4.1)]{Moret-Bailly_Laumon}. 
Indeed, for any $\frX\in \Stk(\cC)$ and $G\in \Shv(\cC)$, the projection $\frX\to F_{\frX}$ induces a natural isomorphism
$
\Hom_{\Shv(\cC)}(F_{\frX}, G)  \cong \Hom_{\Stk(\cC)}(\frX, i(G)),
$
because the fiber categories of the stack attached to a sheaf have only identity morphisms.
\end{para}

\begin{para}{\it Surjections \cite[3.6]{Moret-Bailly_Laumon}.}
Let $\cC$ be a site. 
A morphism $f:\frX\to \frY$ between stacks $\frX$ and $\frY$ on $\cC$ is called \textbf{surjective} if for any $U\in \cC$ and any object $y\in \frY(U)$, there exist a covering family $\{U_i\to U\}_{i\in I}$ and objects $x_i \in \frX(U_i)$, such that $f(x_i) \cong y|_{U_i}$ in $\frY(U_i)$.
Equivalently, a map $f:\frX\to \frY$ is surjective if and only if the induced map $F_f: F_{\frX}\to F_{\frY}$ on the coarse sheaves is an epimorphism in the category $\Shv(\cC)$ of sheaves on the site $\cC$.
\end{para}

\begin{para}{\it Slice categories.} 
Let $\cC$ be a site.
We view any object $U\in \cC$ as a category fibered in sets over $\cC$ by firstly viewing $U$ as the representable presheaf $\Hom_{\cC}(-,U): \cC^\op\to {\rm Set}$, and then using the equivalence \cite[\href{https://stacks.math.columbia.edu/tag/02Y2}{Tag 02Y2}]{stacks-project} between the category of presheaves of sets over $\cC$ and the category of categories fibered in sets over $\cC$. 
If $\cC$ is subcanonical, i.e., representable presheaves on $\cC$ are sheaves, then $U$ is a stack in sets over $\cC$, \cite[\href{https://stacks.math.columbia.edu/tag/0430}{Tag 0430}]{stacks-project}.
Explicitly, as a category fibered over $\cC$, the object $U$ is the slice category over $\cC$ whose objects are pairs $(V,x)$ with $V\in \cC$ and $x\in \Hom_\cC(V,U)$, and whose morphisms from $(V_1,x_1)$ to $(V_2,x_2)$ are morphisms $f\in \Hom_\cC(V_1,V_2)$ such that $x_2\circ f = x_1$.

\end{para}

\subsection{Generic fiber of formal algebraic stacks}
\label{subsection3.2_genericfiber}

\begin{para}
\label{para:FSO}
Consider the category ${\rm FS}_{\cO_E}$ of locally noetherian adic formal schemes $\cX$ over $\Spf(\cO_E)$ whose reduction $\cX_{\red}$ is a scheme locally of finite type over $k := \cO_E/\varpi \cO_E$, as in \cite[7.0.1]{deJongCrystalline}.
We call ${\rm FS}_{\cO_E}$ the category of formal schemes $\cX$ {\bf locally formally of finite type} over $\Spf(\cO_E)$.
If $\cX=\Spf(A)$ is affine over $\Spf(\cO_E)$, so that $A$ is noetherian and adic with the biggest ideal of definition $I$ of $A$, then $\cX\in {\rm FS}_{\cO_E}$ precisely when $A/I$ is a finitely generated $k$-algebra, which is equivalent to $A$ being a quotient of the ring $A_{n,m} := \cO_E\langle x_1,\dots, x_n \rangle [[y_1,\dots,y_m]]$ for some $n,m \geq 0$, cf. \cite[\S7.1]{deJongCrystalline}.
A formal algebraic space $F$ over $\Spf(\cO_E)$ is {\bf locally formally of finite type} if it admits a morphism $U\to F$ from some formal scheme $U\in \FS_{\cO_E}$, which is representable by schemes, \'{e}tale, and surjective.
A formal algebraic stack $\cX$ over $\Spf(\cO_E)$ is {\bf locally formally of finite type} if it admits a morphism $U\to \cX$ from a formal algebraic space $U$ locally formally of finite type, which is representable by algebraic spaces, smooth, and surjective.
\end{para}

\begin{para}
\label{para:formal-to-rigid}
Following \cite[(5.1.12)]{EmertonGeeHellmann}, \cite[\S7.1.7]{deJongCrystalline}, and \cite[Definition 4.31]{Oh-Shimizu_moduli}, to any formal algebraic stack $\cX$ locally formally of finite type over $\Spf(\cO_E)$, we associate its {\bf rigid analytic generic fiber} $\cX^\rig$ as follows.
We describe it on the dense subsite $\Aff_E \subset \Rig_E$ of affinoid rigid $E$-spaces with respect to Tate-fpqc and Tate-\'{e}tale topologies.
Let 
$$
\cX^\rig_{\rm pre}(\Sp(A)) := \varinjlim_{\frY} \cX(\frY), \quad \text{ for all } Y=\Sp(A)\in \Aff_E
$$
be the prestack on $\Rig_E$ defined by the filtered colimit of groupoids $\cX(\frY)$ as $\frY$ ranges over all models for $Y=\Sp(A)$ in the sense of \cite[\S7.1.6]{deJongCrystalline}.
We write $\cX_{\eta,et}$ for the Tate-\'{e}tale stackification of $\cX^{\rig}_{\rm pre}$ and write $\cX^\rig$ for the Tate-fpqc stackification of $\cX^{\rig}_{\rm pre}$, which is also the Tate-fpqc stackification of $\cX_{\eta,et}$.
\end{para}

\begin{para}
\label{para:presentation}
Let $\cX$ be a locally noetherian formal algebraic stack over $\Spec(\cO_E)$ that admits a map to $\Spf(\cO_E)$.
By \cite[Definition A.6]{EmertonGee_modulicrystalline}, there exists a smooth surjective map of $\cO_E$-stacks
$
\pi: \cU = \coprod_{i\in I} \Spf(A_i) \twoheadrightarrow \cX
$
representable by formal algebraic spaces, with $A_i$ being noetherian and adic $\cO_E$-algebras.
In particular, the fiber product $\cR:=\cU\times_\cX \cU$ is a locally noetherian formal algebraic space.
Denote the projection maps by $s,t:\cR \rightrightarrows \cU$, which are smooth.
We have a coequalizer diagram
\[
\begin{tikzcd}
    \cR \arrow[r, shift left=1ex, "s"] \arrow[r, shift right=1ex, "t"] & \cU \arrow[r, "\pi"] & \cX \,.
\end{tikzcd}
\]
Passing to the underlying reduced substacks as in \cite[Definition 3.27]{Emerton_formalstacks} and \cite[A.10]{EmertonGee_modulicrystalline}, we get by \cite[Remark 5.27]{Emerton_formalstacks} a coequalizer diagram of algebraic stacks over $\Spec(k)$
\[
\begin{tikzcd}
    \cR_\red \arrow[r, shift left=1ex, "s_\red"] \arrow[r, shift right=1ex, "t_\red"] & \cU_\red \arrow[r, "\pi_\red"] & \cX_\red \,,
\end{tikzcd}
\]
where $s_\red,t_\red,\pi_\red$ are smooth since they are obtained by base change from smooth maps by \cite[Lemma 3.37]{Emerton_formalstacks}.
\end{para}

We observe the following situation in which a formal algebraic stack over $\Spf(\cO_E)$ is locally formally of finite type.

\begin{lemma}
\label{lem:finite_type_red}
Suppose that $\cX$ is a locally noetherian formal algebraic stack over $\Spec(\cO_E)$ that admits a map to $\Spf(\cO_E)$, such that $\cX_\red$ is an algebraic stack locally of finite type over $k$.
Then, $\cX$ is locally formally of finite type.
\end{lemma}
\begin{proof} 
Since $\cX$ is locally noetherian, there is a smooth surjection 
$
\pi: \cU = \coprod_{i\in I} \Spf(A_i) \twoheadrightarrow \cX
$
representable by formal algebraic spaces, with $A_i$ being noetherian adic $\cO_E$-algebras.
Let $I_i\subset A_i$ be the largest ideal of definition. 
The reduction $\pi_\red: \cU_\red = \coprod_i \Spec(A_i/I_i)\twoheadrightarrow \cX_\red$ is smooth, hence locally of finite type.
Since $\cX_\red$ is locally of finite type over $k$, each $\Spec(A_i/I_i)$ is of finite type over $\Spec(k)$.
Thus each $\Spf(A_i)$ belongs to $\FS_{\cO_E}$, and $\cX$ is locally formally of finite type.
\end{proof}

\begin{para}
Motivated by the theory of \'{e}tale spaces and adic stacks developed in the \'{e}tale topology in \cite{Oh-Shimizu_moduli}, it is natural to compare the Tate-\'{e}tale and Tate-fpqc versions of quasi-analytic spaces and rigid analytic Artin stacks.
We show below in \ref{thm:qan_etale=fpqc} (resp. \ref{thm:Artin_etale=fpqc}) that quasi-analytic spaces (resp. Artin stacks) on $\Rig_{E,\acute{e}t}$ are quasi-analytic spaces (resp. Artin stacks) on $\Rig_{E,fpqc}$.\footnote{This is the rigid analogue of \cite[\href{https://stacks.math.columbia.edu/tag/076L}{Tag 076L}]{stacks-project} (resp. \cite[\href{https://stacks.math.columbia.edu/tag/076U}{Tag 076U}]{stacks-project}).}
By the arguments of \cite{Oh-Shimizu_moduli}, the Tate-\'{e}tale rigid generic fiber $\cX_{\eta,et}$ of a formal algebraic stack $\cX$ locally formally of finite type over $\Spf(\cO_E)$ is a rigid analytic Artin stack on $\Rig_{E,et}$, and hence a rigid analytic Artin stack on $\Rig_{E,fpqc}$.
Moreover, if $\cX=[\cU/\cR]$ for a smooth groupoid $(\cU,\cR,s,t,c)$ of formal algebraic spaces (necessarily) locally formally of finite type over $\Spf(\cO_E)$, then the Tate-fpqc rigid generic fiber $\cX^{\rig}=[\cU^\rig/\cR^\rig]$ is the Tate-fpqc quotient stack of the smooth groupoid $(\cU^\rig,\cR^\rig,s^\rig,t^\rig,c^\rig)$ in quasi-analytic spaces.

\end{para}

\begin{lemma}
\label{lem:etale_equivalence}
Let $(U,R,s,t,c)$ be an \'{e}tale equivalence relation in rigid spaces. 
Then, the natural map $p:U\to X:=U/R$ to the Tate-fpqc quotient sheaf is representable in rigid spaces, \'{e}tale, and surjective.
Moreover, $X$ is a quasi-analytic space over $\Rig_E$.
\end{lemma}
\begin{proof}
By Lemma \ref{lem:aux_1}, it suffices to show that the surjective map $p:U\to X$ is representable in rigid spaces and \'{e}tale.
Let $s:T\to X$ be a map from a rigid space $T\in \Rig_E$.
By the definition of sheafification, there is a Tate-fpqc cover $T'\to T$ in $\Rig_E$ such that the section $s|_{T'}$ factors through the map $p:U\to X$ by $s':T'\to U$.
We need to show that $U\times_{p,X,s} T$ is a rigid space.
Since the representability is Tate-local on $U$, it suffices to show that $U_0\times_{X}T$ is representable by a rigid space for all affinoid admissible open $U_0$ in $U$.
We claim that $U_0\times_X T' \to T'$ is a separated \'{e}tale morphism of rigid spaces.
Over $T'$, we have 
$
(U_0\times_X T)\times_T T' \cong X\times_{X\times X}(U_0\times T') \cong R\times_{(s,t),U\times U}(U_0\times T')
$, cf. \cite[\href{https://stacks.math.columbia.edu/tag/044V}{Tag 044V}]{stacks-project}, which is a rigid space.
Since $R\rightrightarrows U$ is an equivalence relation, $(s,t):R\to U\times U$ is a monomorphism, hence separated. 
Since $U_0$ is affinoid and thus separated, the projection $U_0\times T' \to T'$ is separated.
So, $U_0\times_X T' \to T'$ is separated.
As for \'{e}taleness, $U_0\times_{X,p\circ s'} T' \cong (U_0\times_X U) \times_{U,s'} T' \cong U_0\times_{U,s} (R\times_{t,U,s'}T') \to T'$ is the base change of the map $R\times_{t,U,s'}T'\to T'$, which is the base change of the \'{e}tale map $t:R\to U$.
Thus, $U_0\times_X T' \to T'$ is \'{e}tale.
We claim that the separated \'{e}tale map $U_0\times_X T' \to T'$, equipped with the natural descent datum relative to the Tate-fpqc cover $T'\to T$, descends to an \'{e}tale map of rigid spaces to $T$.
Applying Scholze's functor $(-)^{\diamond}$, which induces an equivalence of \'{e}tale sites \cite[Theorem 1.5]{Scholze_diamond}, and using the effective $v$-descent of separated \'{e}tale maps \cite[Proposition 9.7]{Scholze_diamond}\footnote{\cite[Proposition 9.7]{Scholze_diamond} is about separated \'{e}tale maps of perfectoid spaces, but we need the descent for separated \'{e}tale maps of locally spatial diamonds. If $C$ is a site in which finite limits are representable, the topos of sheaves on $C$ always satisfies effective descent \cite[4.2.11]{Olsson_Alg-spaces-stacks}. Thus we have by \cite[Proposition 10.11(iv)]{Scholze_diamond} effective $v$-descent of separated \'{e}tale maps of $v$-sheaves as a corollary to the $v$-descent for separated \'{e}tale maps of perfectoids \cite[Proposition 9.7]{Scholze_diamond}. Since a pro-\'{e}tale sheaf that is \'{e}tale over a locally spatial diamond is a locally spatial diamond \cite[Corollary 11.28]{Scholze_diamond}, the descended $v$-sheaf is a locally spatial diamond.}, we deduce that $U_0\times_X T\to T$ is a separated \'{e}tale map of rigid spaces\footnote{A priori, we only know that $U_0\times_X T$ is an analytic adic space over $T$. But since it is separated and \'{e}tale over $T$, it follows that Tate-locally on $T$, $U_0\times_X T$ is a quasi-separated adic space locally of finite type over $E$, hence comes from a rigid space, cf. \cite[1.1.11]{Hu96}. By gluing along admissible opens, we conclude that $U_0\times_X T$ is a rigid space.}, so $p:U\to X$ is representable and \'{etale}.
\end{proof}

{
}

\begin{thm}
\label{thm:qan_etale=fpqc}
Let $X_{et}$ be a Tate-\'{e}tale quasi-analytic space over $E$.
Let $p_{et}:U\to X_{et}$ be a map representable in rigid spaces, \'{e}tale, and Tate-\'{e}tale surjective, from a rigid space $U\in \Rig_E$.
Let $X_{fpqc}$ be the Tate-fpqc sheafification of $X_{et}$. 
Then, the induced map $p_{fpqc}:U\to X_{fpqc}$ is representable in rigid spaces, \'{e}tale, and Tate-fpqc surjective.
So, $X_{fpqc}$ is a Tate-fpqc quasi-analytic space over $E$.
In fact, $X_{et}\cong X_{fpqc}$.
\end{thm}
\begin{proof}
Let $R:=U\times_{X_{et}}U$.
Since $X_{et}$ is a Tate-\'{e}tale quasi-analytic space, $R$ is a rigid space equipped with natural projection maps $s,t:R\rightrightarrows U$ defining an \'{e}tale equivalence relation such that $(U/R)_{et}=X_{et}$.
Then, $X_{fpqc}$ is the Tate-fpqc sheafification of the Tate-\'{e}tale quotient sheaf $(U/R)_{et}$, which is the Tate-fpqc quotient sheaf $(U/R)_{fpqc}$.
By Lemma \ref{lem:etale_equivalence}, $X_{fpqc}$ is a Tate-fpqc quasi-analytic space that admits $p_{fpqc}$ as an \'{e}tale surjection over $\Rig_{E,fpqc}$.
Note that $p_{fpqc}$ is also surjective in the Tate-\'{e}tale topology, because for any $T\to X_{fpqc}$ with $T\in \Rig_E$, the base change $T\times_{X_{fpqc}} U\to T$ is a Tate-\'{e}tale cover of rigid spaces.
We claim that the natural map $\alpha: X_{et}\to X_{fpqc}$ is an isomorphism of presheaves on $\Rig_E$.
From $p_{fpqc}=\alpha \circ p_{et}$, it follows that $\alpha$ is an epimorphism of Tate-\'{e}tale sheaves.
The fact that $U\times_{(U/R)_{et}} U = R = U\times_{(U/R)_{fpqc}} U$ implies that $\alpha$ is a monomorphism of presheaves.
Hence, $\alpha:X_{et}\xrightarrow{\simeq} X_{fpqc}$ is an isomorphism.
\end{proof}

\begin{lemma}
\label{lem:Berkovich}
Let $f:S'\to S$ be a faithfully flat quasi-compact morphism of affinoid spaces over $E$.
Then, there exist finitely many affinoid spaces $\{V_j\}_{j=1}^r$, and an fpqc morphism $\psi:\coprod_{j=1}^r V_j \to S$ such that $\psi|_{V_j}$ factors as $V_j \xrightarrow{g_j} W_j \xrightarrow{h_j} S$ for some affinoid space $W_j$ with $g_j$ finite faithfully flat and $h_j$  \'{e}tale, for each $j$.
Moreover, there is an $S$-map $\psi': \coprod_j V_j\to S'$ such that $\psi = f \circ \psi'$ and the following diagram is commutative
$$
\begin{tikzcd}[column sep=25mm, row sep=20mm]
\coprod_{j=1}^r V_j \arrow[d] \ar[dr,"\psi"] \ar[d,"\psi'"] \arrow[r, "g", two heads] & \coprod_{j=1}^r W_j \arrow[d, "h", two heads] \\
S' \arrow[r, "f"] & S
\end{tikzcd}
$$
Finally, $g$ is a finite faithfully flat cover, and $h$ is a Tate-\'{e}tale cover.
\end{lemma}
\begin{proof}
First, view $f$ as an fpqc map of strict $E$-affinoid Berkovich spaces. 
Then, by Ducros' multisection theorem, cf. \cite[2.2.2]{Daylies_descente}, there exist a strict $E$-affinoid space $X$, a quasi-finite faithfully flat map $X\to S$, and an $S$-map $X\to S'$.
By \cite[2.2.3]{Daylies_descente}, we can find a finite $G$-covering $X=\bigcup_{j=1}^r V_j$ such that the induced map $\psi: \coprod_{j} V_j \to X\to S$, when restricted to each $V_j$, factors as $V_j \xrightarrow{g_j} W_j \xrightarrow{h_j} S$
with $g_j$ finite faithfully flat and $h_j$ quasi-\'{e}tale in the sense of \cite[5.2.6, 5.4.10, 5.4.11]{Ducros_families}.
For a quasi-\'{e}tale map of strict $E$-affinoid Berkovich spaces, its restriction on rigid points is an \'{e}tale map of rigid spaces \cite[Lemma A.10]{ALY_crelle}.
Since $\psi$ is a faithfully flat map of affinoid spaces, we conclude that on  rigid analytic points, the map $h = \coprod_j h_j:\coprod_j W_j \to S$ is a Tate-\'{e}tale cover of rigid spaces, i.e., \'{e}tale and surjective on the associated adic spaces.
\end{proof}

\begin{prop}
\label{prop:fppf=et+fpf}
Any Tate-fpqc covering family $\{X_i\to X\}_i$ of $X\in\Rig_E$ can be refined by a finite faithfully flat cover of a Tate-\'{e}tale cover of $X$.
\end{prop}
\begin{proof}
By the definition of the Tate-fpqc topology, there exists an admissible cover $X=\bigcup_{j} U_j$ with $U_j$ affinoid opens such that over each $U_j$, there exist an fpqc map $Z_j\to U_j$ and a $U_j$-map $Z_j\to \coprod_i (X_i \times_X U_j)$.
Since $U_j$ is quasi-compact, so is $Z_j$.
Therefore, the image of $Z_j\to \coprod_i (X_i \times_X U_j)$ meets only finitely many $X_i$, say $i\in I(j)$ for a finite subset $I(j)\subset I$.
Replacing $Z_j$ by the disjoint union of a finite admissible affinoid cover, we may assume that $Z_j$ is affinoid.
Applying Lemma \ref{lem:Berkovich} to the fpqc map $Z_j\to U_j$ of affinoid rigid spaces, we have affinoid rigid spaces $V_j, W_j$ with maps $f_j: V_j\to Z_j$ such that each composition $V_j\to Z_j \to U_j$ factors as the composition of a finite faithfully flat map $g_j:V_j\to W_j$ and a Tate-\'{e}tale surjection $h_j: W_j \to U_j$.
Therefore, for $f := \coprod_j f_j : \coprod_j V_j \to \coprod_j Z_j$, $g:= \coprod_j g_j$ and $h:=\coprod_j h_j$, we have a commutative diagram
$$
\begin{tikzcd}
\coprod_j V_j \arrow[d, "f"] \arrow[r, "g"] & \coprod_j W_j \arrow[r, "h"]             & X                                  \\
\coprod_j Z_j \arrow[r]                           & \coprod_j \coprod_{i\in I(j)} (X_{i}\times_X U_j) \arrow[r] & \coprod_i X_i \arrow[u, two heads]
\end{tikzcd}
$$
with $g$ being a finite faithfully flat map and $h$ being an \'{e}tale and strongly surjective map, such that $\{V_j \xrightarrow{h_j\circ g_j} X\}_{j\in J}$ refines the Tate-fpqc cover $\{X_i\to X\}_i$.
\end{proof}

The theory of {\it Weil restriction of scalars} for rigid spaces and adic spaces has been developed in \cite{Bertapelle_Weil} and \cite{Wahle_Weil}. 
For any finite faithfully flat morphism $S'\to S$ of rigid spaces, and any quasi-analytic space $X$ over $S'$, we can define $\Res_{S'/S}(X)$ as a presheaf on $\Rig_E$, in the same way as  \cite[\href{https://stacks.math.columbia.edu/tag/05Y8}{Tag 05Y8}]{stacks-project}, which is actually a Tate-fpqc sheaf since $X$ is.

\begin{lemma}
\label{lem:Weil_QAn}
Let $S'\to S$ be a finite faithfully flat map of rigid spaces, and let $Z'$ be a quasi-analytic space over $S'$.
Then, the sheaf $\Res_{S'/S}(Z')$ is also a quasi-analytic space.
\end{lemma}
\begin{proof}
The question is Tate-local over $S$, hence we may assume that $S$ is an affinoid rigid space.
Since $Z'$ is a quasi-analytic space, we can choose an \'{e}tale surjection $p':U'\to Z'$ from a rigid space $U'$. 
Replacing $U'$ by the disjoint union $\coprod_{i\in I} U'_i$ for an admissible cover $U'=\bigcup_{i\in I} U_i'$ by affinoid opens $U_i'$, we assume that $U'$ is a disjoint union of affinoid rigid spaces, and hence separated over $E$.
We claim that 
$$
p:=\Res_{S'/S}(p'): \Res_{S'/S}(U') \to \Res_{S'/S}(Z')
$$
is representable by rigid spaces, \'{e}tale, and surjective.
Note that $U'\to Z'\to S'$ is separated, hence $\Res_{S'/S}(U')$ is represented by a rigid space by Lemma \ref{lem:Weil_repr}(i).
Let $T\to \Res_{S'/S}(Z')$ be a morphism of sheaves, which amounts to a pair $(a,b)$ of morphisms, where $a:T\to S$ and $b:T\times_{a,S}S'\to Z'$ \cite[\href{https://stacks.math.columbia.edu/tag/05Y8}{Tag 05Y8}]{stacks-project}.
A morphism from a rigid space $T'$ to the fiber product
$
T\times_{\Res_{S'/S}(Z')} \Res_{S'/S}(U)
$
is the data of a pair $(f,g)$ of morphisms, where $f:T'\to T$ and $g: T'\times_{S}S'\to U'$ such that $p'\circ g = b\circ (f\times_{S} S')$, which is equivalent to a pair $(c,d)$ of morphisms, where $c=f:T'\to T$ and $d=(f,g):T'\times_{c,T}(T\times_S S') \to W$ for the fiber product $W:=(T\times_{S}S')\times_{Z'}U'$, and $W$ is a rigid space as $Z'$ is a quasi-analytic space.
Such a pair is exactly a morphism of sheaves $T'\to \Res_{T\times_{S}S'/T}(W)$, thus we have shown that
$
T\times_{\Res_{S'/S}(Z')} \Res_{S'/S}(U) \cong \Res_{T\times_S S'/T}(W).
$
Since the map $W\to T\times_S S'$ is the base change of the separated \'{e}tale surjection $U'\to Z'$, by Lemma \ref{lem:Weil_repr}(i), $T\times_{\Res_{S'/S}(Z')} \Res_{S'/S}(U) \cong \Res_{T\times_S S'/T}(W)$ is represented by a rigid space, and $W\to T\times_S S'$ is a Tate-\'{e}tale cover of rigid spaces.
By Lemma \ref{lem:Weil_repr}(ii), the structure map $\Res_{T\times_S S'/T}(W) \to T$ is a Tate-\'{e}tale cover.
Therefore, $p=\Res_{S'/S}(p')$ is representable, \'{e}tale, and surjective, which, by Lemma \ref{lem:aux_1}, allows us to conclude that $\Res_{S'/S}(Z')$ is a quasi-analytic space.
\end{proof}

\begin{lemma}
\label{lem:fpf-desc}
Let $q:S'\to S$ be a finite faithfully flat map of rigid spaces over $E$, and let $Z'$ be a quasi-analytic space over $S'$ equipped with a descent datum relative to $q:S'\to S$.
Then, $Z'$ descends to a quasi-analytic space over $S$.
\end{lemma}
\begin{proof}
Let $S'':=S'\times_S S'$, which is finite faithfully flat over $S'$ and $S$, equipped with two projections $p_1,p_2: S''\rightrightarrows S$.
We write $\varphi:p_1^*Z' \xrightarrow{\simeq} p_2^*Z' =: Z''$ for the descent datum.
By \cite[4.2.11, Exercise 2.D(b)]{Olsson_Alg-spaces-stacks} we descend $Z'$ to a Tate-fpqc sheaf $Z$ over $S$.
Explicitly, cf. \cite[\href{https://stacks.math.columbia.edu/tag/04TR}{Tag 04TR}]{stacks-project}, we have that for any $S$-rigid space $T$,
\begin{align*}
Z(T) 
&= 
\set{f\in \Hom_{S'}(T\times_S S',Z')}{\varphi(p_1^*f)=p_2^*f}    
\\&= \eq(\Res_{S'/S}(Z')\rightrightarrows \Res_{S''/S}(Z'')) (T)
\\&\cong \left(\Res_{S'/S}(Z')\times_{(\varphi\circ p_1^*,p_2^*), \Res_{S''/S}(Z'') \times \Res_{S''/S}(Z''), \Delta} \Res_{S''/S}(Z'')\right)(T).
\end{align*}
By Lemma \ref{lem:Weil_QAn}, $Z\cong \eq(\Res_{S'/S}(Z')\rightrightarrows \Res_{S''/S}(Z''))$ is a quasi-analytic space.
\end{proof}

\begin{lemma}
\label{lem:Weil_repr}
Let $q:S'\to S$ be a finite locally free morphism of rigid spaces over $E$, and $f:X'\to S'$ be a separated morphism of rigid spaces over $E$.
\begin{enumerate}[(i)]
\item 
Then, $\Res_{S'/S}(X')$ is represented by a separated $S$-rigid space, and $\Res_{r(S')/r(S)}(r(X'))$ is represented by a locally of finite type separated $r(S)$-adic space such that 
$$
r(\Res_{S'/S}(X'))\cong \Res_{r(S')/r(S)}(r(X'))
$$
as adic spaces over $r(S)$.
\item 
If $f:X'\to S'$ is furthermore a Tate-\'{e}tale cover, so is $\Res_{S'/S}(f):\Res_{S'/S}(X') \to S$.
\end{enumerate}
\end{lemma}
\begin{proof}
(i)
Let $\{S_i\to S\}_i$ be an admissible cover of $S$ by affinoid opens such that the base change $S'_i:=S'\times_S S_i$ is finite free over $S_i$ (in particular, affinoid), and we write $X'_i:=X'\times_{S'}S'_i.$
The map
$r(f_i):r(X'_i)\to r(S_i')$ is a separated morphism of adic spaces locally of finite type over $\Spa(E)$.
By \cite[4.3.6]{Zavyalov_quotients}, we can apply \cite[2.3.3]{Wahle_Weil} to obtain that $\Res_{r(S'_i)/r(S_i)}(r(X_i'))$ is represented by an adic space locally of finite type over $r(S_i)$.
Then, $\Res_{r(S'_i)/r(S_i)}(r(X_i'))$ is separated over $r(S_i)$ because $r(X_i')$ is separated over $r(S_i)$, and the functor $\Res_{r(S'_i)/r(S_i)}$  commutes with fiber products \cite[(W3)]{Bertapelle_Weil} and also preserves closed immersions \cite[Lemma 2.2.8 and discussion afterward]{Wahle_Weil}.
The universal property of Weil restriction implies that these $\Res_{r(S'_i)/r(S_i)}$ glue over the affinoid cover $\{r(S_i)\to r(S)\}_i$ to an adic space over $r(S)$, which represents $\Res_{r(S')/r(S)}(r(X'))$, and it is locally of finite type and separated over $r(S)$.
Under Huber's equivalence of categories \cite[(1.1.11)]{Hu96}, there exists a rigid space $R$ separated over $S$ such that $r(R)\cong \Res_{r(S')/r(S)}(r(X'))$ as locally of finite type separated $r(S)$-adic spaces.
We claim that $R$ represents the functor $\Res_{S'/S}(X')$ on $S$-rigid spaces.
Indeed, for any rigid $S$-space $T\to S$, we have
\begin{align*}
\Hom_{\Rig/S}(T, R)
&=
\Hom_{\Ad/r(S)}(r(T), r(R))
\\&=
\Hom_{\Ad/r(S)}(r(T), \Res_{r(S')/r(S)}(r(X')))
\\&=
\Hom_{\Ad/r(S')}(r(T)\times_{r(S)}r(S'), r(X'))
\\&=
\Hom_{\Rig/S'}(T\times_S S', X'),
\end{align*}
which shows that $R\cong \Res_{S'/S}(X')$.
Consequently, $r(\Res_{S'/S}(X'))\cong \Res_{r(S')/r(S)}(r(X'))$.
\vskip8pt

(ii) By \cite[2.1.2]{Hu96}, $r(f):r(X')\to r(S')$ is an \'{e}tale surjective morphism of adic spaces, and by (i), it suffices to show that the induced map $\Res(r(f)):\Res_{r(S')/r(S)}(r(X'))\to r(S)$ is \'{e}tale and surjective.
This is locally of finite type since it comes from a map of rigid spaces.
For any affinoid space $\Spa A = \Spa(A^\rhd,A^+)$, any ideal $I$ of $A^\rhd$ with $I^2=(0)$, and any morphism $\Spa A\to r(S)$, the square-zero closed immersion $\Spa A/I \hookrightarrow \Spa A$ induces 
$$
\Hom_{r(S)}(\Spa A, \Res_{r(S')/r(S)}(r(X'))) \to \Hom_{r(S)}(\Spa A/I, \Res_{r(S')/r(S)}(r(X'))),
$$
which is bijective.
Indeed, by the adjunction defining Weil restriction, it corresponds to 
$$
\Hom_{r(S')}(\Spa A \times_{r(S)} r(S'), r(X')) \to \Hom_{r(S')}(\Spa A/I \times_{r(S)} r(S'), r(X'))
$$
induced by the square-zero closed immersion $\Spa A/I \times_{r(S)} r(S') \hookrightarrow \Spa A \times_{r(S)} r(S')$ of affinoid spaces\footnote{Recall that $r(q):r(S')\to r(S)$ is a finite morphism.}.
Since $r(X')\to r(S')$ is \'{e}tale, it follows that the displayed maps above are bijective, and hence $\Res_{r(S')/r(S)}(r(f))$ is \'{e}tale \cite[Definition 1.6.5]{Hu96}.
\vskip8pt

To show that 
$$
|\Res_{r(S')/r(S)}(r(f))|:
|\Res_{r(S')/r(S)}r(X')| \twoheadrightarrow |r(S)|
$$
is surjective, for any $s\in |r(S)|$, fix a {\it completed} algebraic closure $\ovk(s)$ of the residue field $k(s)$ and let $\ovk(s)^+$ be a valuation ring of $\ovk(s)$ that extends $k(s)^+$, cf. \cite[(2.5.2)]{Hu96}. 
The inclusion $k(s)\hookrightarrow\ovk(s)$ induces a map 
$$
\overline{s}: \Spa  \overline{\kappa}(s):=\Spa(\overline{k}(s),\overline{k}(s)^+) \to r(S)
$$ 
that sends the unique closed point of $\Spa \overline{\kappa}(s)$ to $s\in r(S)$. 
Since $r(S')\to r(S)$ is finite locally free, the fiber $r(S')_{\overline{s}}:=r(S')\times_{r(S)}\Spa \overline{\kappa}(s)$ is finite free over $\Spa(\overline{\kappa}(s))$, and hence of the form $\Spa B = \Spa(B^\rhd, B^+)$ for a finite dimensional $\overline{k}(s)$-algebra $B^\rhd$, and $B^+$ is the integral closure of $\overline{k}(s)^+$ in $B^\rhd$ \cite[(1.4.2)]{Hu96}.
The Artinian $\ovk(s)$-algebra $B^\rhd=\prod_{i=1}^r B_i^\rhd$ decomposes into the product of its localizations $B_i^\rhd$ at maximal ideals $\frm_i^\rhd$.
With $B_i^+$ being the integral closure of $\ovk(s)^+$ in $B_i^\rhd$ and with $B_i:=(B_i^\rhd, B_i^+)$, we have
$
\Spa B = \coprod_{i=1}^r \Spa B_i.
$
Since $\ovk(s)$ is algebraically closed, each residue field $B_i^\rhd/\frm_i^\rhd$ is isomorphic to $\ovk(s)$, hence $B_i^\rhd$ is a strictly henselian local ring as it is local Artinian with algebraically closed residue field.
Note that $B_i^+=\set{b\in B_i^\rhd}{b \pmod{\frm_i^\rhd} \in \ovk(s)^+}$ because $B_i^+$ is the integral closure of $\ovk(s)^+$ in $B_i^\rhd$ and $\frm_i^\rhd$ is nilpotent. 
Thus by \cite[2.5.7]{Hu96}, each $\Spa(B_i)$ has a unique closed point $x_i$, which is given by the maximal ideal $\frm_i^\rhd$ and the valuation on $\overline{k}(s)^+$.
Thus by \cite[2.5.6]{Hu96}, since condition (iii) in {\it loc. cit.} is just shown to be satisfied by $\Spa(B_i)$, the equivalent condition (i) in {\it loc. cit.} holds true for $\Spa(B_i)$, which implies that for any \'{e}tale morphism $g:Y\to \Spa(B_i)$ of adic spaces and any point $y\in Y$ such that $g(y)=x_i$, there exists a section $h_i:\Spa(B_i)\to Y$ of $g$ with $h_i(x_i)=y$.
Consider the morphism
$$
g:
r(X')\times_{r(S)} \Spa \overline{\kappa}(s) \to r(S')\times_{r(S)} \Spa \overline{\kappa}(s)= \coprod_i \Spa(B_i),
$$
which is an \'{e}tale surjection, as it is the base change of the \'{e}tale surjection $r(X')\to r(S')$ along the adic morphism $\overline{s}: \Spa \overline{\kappa}(s) \to r(S)$.
Hence, choosing preimages $y_i$  of $x_i$ under $g$ for each $i$, we have a section $h: r(S')\times_{r(S)} \Spa \overline{\kappa}(s)\to r(X')\times_{r(S)} \Spa \overline{\kappa}(s)$ of $g$.
Further composing $h$ with the projection $r(X')\times_{r(S)} \Spa \overline{\kappa}(s)\to r(X')$ gives a morphism
$$
r(S')\times_{r(S)} \Spa \overline{\kappa}(s) \to r(X')
$$
that is compatible with the maps to $r(S')$ because $h$ is a section of $g$.
By the adjunction defining Weil restriction, we get an $r(S)$-morphism
$
\Spa(\overline{\kappa}(s))\to \Res_{r(S')/r(S)}(r(X')).
$
Since the closed point of $\Spa \overline{\kappa}(s)$ lies over $s$, the image of this point in $\Res_{r(S')/r(S)}(r(X'))$ lies over $s\in r(S)$.
Thus, 
$
\Res_{r(S')/r(S)}(r(f)):\Res_{r(S')/r(S)}(r(X'))\to r(S)
$ 
is surjective.
\end{proof}

\begin{cor}
\label{cor:QAn_desc}
Let $q:S'\to S$ be a Tate-fpqc cover in $\Rig_E$.
Let $X'$ be a quasi-analytic space over $S'$ equipped with a descent datum relative to $q:S'\to S$.
Then, $X'$ descends to a quasi-analytic space $X$ over $S$.
\end{cor}
\begin{proof}
Consider the category of quasi-analytic spaces fibered over $\Rig_E$, which we view as a pseudofunctor on $\Rig_E$.
Given \cite[Lemma 2.5]{Daylies_descente}, we could argue as in \cite[2.2.1]{Daylies_descente} to conclude that, by Proposition \ref{prop:fppf=et+fpf} and Lemma \ref{lem:fpf-desc}, it suffices to prove the desired effective descent statement assuming $q:S'\to S$ is a Tate-\'{e}tale cover.

\vskip8pt
Assuming $q:S'\to S$ is a Tate-\'{e}tale cover, by \cite[4.2.11]{Olsson_Alg-spaces-stacks} we can descend $X'$ to a Tate-\'{e}tale sheaf $X$ over $S$.
The projection $X'\to X$ is representable in rigid spaces, \'{e}tale, and surjective because it is the base change of the Tate-\'{e}tale cover $q:S'\to S$ of rigid spaces.
Let $p:U\to X'$ be an \'{e}tale chart of $X'$.
By Lemma \ref{lem:aux_1}, the composite $U\to X'\to X$ is an \'{e}tale chart of $X$, making $X$ a quasi-analytic space over $S$.
\end{proof}

\begin{lemma}
\label{lem:smooth-grpd}
Let $(U,R,s,t,c)$ be a smooth groupoid in quasi-analytic spaces. 
Then, the natural map $p:U\to [U/R]$ to the Tate-fpqc quotient stack is representable in quasi-analytic spaces, smooth, and surjective.
Moreover, $[U/R]$ is a rigid analytic Artin stack.
\end{lemma}
\begin{proof}
By Lemma \ref{lem:aux_2}, it suffices to show that $p:U\to X:=[U/R]$ is representable by quasi-analytic spaces and smooth.
Let $s:T\to X$ be a morphism with $T\in \Rig_E$.
Since $p$ is by definition Tate-fpqc surjective, there exists a Tate-fpqc cover $T'\to T$ such that the section $s|_{T'}$ factors through the map $p:U\to X$ by $s':T'\to U$.
To show that $U\times_X T$ is a quasi-analytic space, observe that over $T'$,
$$
(U\times_X T)\times_T T' \cong X\times_{X\times X}(U\times T') \cong R\times_{U\times U}(U\times T'),
$$ 
cf. \cite[\href{https://stacks.math.columbia.edu/tag/044V}{Tag 044V}]{stacks-project}, which is a quasi-analytic space over $T'$.
By Corollary \ref{cor:QAn_desc}, $U\times_X T$ is a quasi-analytic space over $T$.
Fix an \'{e}tale chart $V\to U\times_X T$ with $V\in \Rig_E$.
As for smoothness of $U\times_X T\to T$, observe that $U\times_X T'\to T'$ is isomorphic to 
$$
U\times_{X,p\circ s'} T' \cong (U\times_{X,p} U) \times_{U,s'} T' \cong R\times_{U,s'}T' \to T',
$$ 
which is smooth, being the base change of the smooth map $t:R\to U$.
Thus, the composite $V \to U\times_X T \to T$ is a map of rigid spaces such that $V\times_T T' \to T'$ is smooth over the Tate-fpqc cover $T'$ of $T$.
Since smoothness is Tate-fpqc local on the target (Lemma \ref{lem:local_property_P}), $V\to T$ is smooth, i.e., $U\times_X T\to T$ is a smooth map of quasi-analytic spaces.
Hence, $p:U\to[U/R]$ is smooth.
\end{proof}

\begin{thm}
\label{thm:Artin_etale=fpqc}
Let $X_{et}$ be a Tate-\'{e}tale rigid analytic Artin stack over $E$.
Let $p_{et}:U\to X_{et}$ be a map representable in quasi-analytic spaces, smooth, and Tate-\'{e}tale surjective, from a rigid space $U\in \Rig_E$.
Let $X_{fpqc}$ be the Tate-fpqc stackification of $X_{et}$.
Then, the induced map $p_{fpqc}:U\to X_{fpqc}$ is representable in quasi-analytic spaces, smooth, and Tate-fpqc surjective.
So, $X_{fpqc}$ is a Tate-fpqc Artin stack over $E$.
In fact, $X_{et}\cong X_{fpqc}$.
\end{thm}
\begin{proof}
Let $R:=U\times_{X_{et}}U$.
Since $X_{et}$ is a Tate-\'{e}tale Artin stack, $R$ is a quasi-analytic space equipped with natural projection maps $s,t:R\rightrightarrows U$  such that $[U/R]_{et}=X_{et}$.
Then, $X_{fpqc}$ is the Tate-fpqc stackification of the Tate-\'{e}tale quotient stack $[U/R]_{et}$, which is the Tate-fpqc\footnote{Here we implicitly use Theorem \ref{thm:qan_etale=fpqc} to guarantee that $U$ and $R$ are Tate-fpqc sheaves.} quotient stack $[U/R]_{fpqc}$.
By Lemma \ref{lem:smooth-grpd}, $X_{fpqc}$ is a Tate-fpqc Artin stack that admits $p_{fpqc}$ as a smooth Tate-fpqc surjection over $\Rig_{E}$.
By the existence of \'{e}tale quasi-sections to smooth surjections (\cite[Remark 3.29]{Oh-Shimizu_moduli}), $p_{fpqc}$ is essentially surjective in the Tate-\'{e}tale topology.
We claim that the natural map $\alpha: X_{et}\to X_{fpqc}$ is an equivalence.
From $p_{fpqc}=\alpha \circ p_{et}$, it follows that $\alpha$ is essentially surjective.
That $U\times_{[U/R]_{et}} U = R = U\times_{[U/R]_{fpqc}} U$ implies that $\alpha$ is fully faithful.
Thus, $\alpha:X_{et}\xrightarrow{\simeq} X_{fpqc}$ is an equivalence.
\end{proof}

{
}

\vskip8pt

\textbf{Auxiliary lemmas.}

\begin{lemma}
\label{lem:aux_1}
Let $U$ be a rigid space and $X$ be a sheaf of sets on $\Rig_E$. 
Let $p: U\to X$ be a map of sheaves that is representable in rigid spaces, \'{e}tale, and surjective.
Then, $R:=U\times_X U$ is a rigid space, where the two  projections $s,t:R\rightrightarrows U$ are representable in rigid spaces, \'{e}tale, and surjective.
The maps $s,t:R\rightrightarrows U$ define an \'{e}tale equivalence relation, $p$ induces an isomorphism $U/R\cong X$, and $X$ is a quasi-analytic space.
\end{lemma}
\begin{proof}
Since $p:U\to X$ is a surjection of sheaves, it induces $U/R\simeq X$ for $R:=U\times_X U$.
Since $p$ is representable in rigid spaces, \'{e}tale, and surjective, it remains to show that $\Delta: X\to  X\times X$ is representable in rigid spaces.
\vskip8pt

Let $T\in \Rig_E$ and $x,y:T\to X$ be any two maps of sheaves.
We need to show that the fiber product $P:=X\times_{\Delta,X\times X, (x,y)}T\to T$ is a rigid space over $T$.
Firstly, \'{e}tale-locally on $T$, $P$ is representable. 
This is because, as in \cite[Remark 3.29]{Oh-Shimizu_moduli}, the (Tate-fpqc) smooth surjection $p:U\to U/R=X$ is Tate-\'{e}tale surjective by the existence of \'{e}tale quasi-sections to smooth surjections.
So, there is a Tate-\'{e}tale cover $\{T_i\to T\}_{i\in I}$ such that $x|_{T_i},y|_{T_i}:T_i \to X=U/R$ admit local sections $x_i,y_i:T_i\to U$.
Thus, over $T_i$ we have, cf. \cite[\href{https://stacks.math.columbia.edu/tag/044V}{Tag 044V}]{stacks-project},
$$
P_i:=
X\times_{\Delta,X\times X, (x|_{T_i},y|_{T_i})}T_i \cong (U\times_X U)\times_{U\times U,(x_i,y_i)} T_i = R \times_{(s,t),U\times U,(x_i,y_i)} T_i\,,
$$
which is a rigid space.
Let $T':=\coprod_{i\in I} T_i$ and $P':=\coprod_{i\in I}P_i$.
Then, $P'=T'\times_T P$, and $P' \to P$ is a map that is representable in rigid spaces, \'{e}tale, and surjective from a rigid space, since it is the base change of the Tate-\'{e}tale cover $T'\to T$ of rigid spaces.
Since $X$ is a sheaf, its diagonal is a monomorphism.
Hence, the base change $P\to T$ is a monomorphism, and $P$ has a representable diagonal since $T$ does.
It follows that $P$ is a quasi-analytic subsheaf of $T\in \Rig_E$.
For any affinoid admissible open cover $\{T_\alpha \to T\}_{\alpha}$, the restriction $P_\alpha:= P\times_T T_\alpha =  X\times_{X\times X}T_\alpha\to T_\alpha$ is a separated quasi-analytic subspace of $T_\alpha$ since $T_\alpha$ is separated.
By \cite[Remark 5.1.9(i)]{EmertonGeeHellmann}, every separated quasi-analytic space is representable by a rigid space, so $P_\alpha \in \Rig_E$. 
Gluing over the admissible cover $\{T_\alpha\to T\}_\alpha$, we conclude that $P\in \Rig_E$.
\end{proof}

\begin{lemma}
\label{lem:aux_2}
Let $U$ be a quasi-analytic space and $X$ be a stack on $\Rig_E$.
Let $p: U\to X$ be a map of stacks that is representable in quasi-analytic spaces, smooth, and surjective.
Then, $R:=U\times_X U$ is a quasi-analytic space, where the natural projections $s,t:R\rightrightarrows U$ are representable in quasi-analytic spaces, smooth, and surjective.
Let $c: R\times_{s,U,t}R =  (U\times_X U)\times_{U}(U\times_X U)\to R$ be the projection map $(u_1,u_2,u_2,u_3)\mapsto (u_1,u_3)$.
Then $(U,R,s,t,c)$ is a smooth groupoid in quasi-analytic spaces, $p$ induces $[U/R]\simeq X$, and $X$ is a rigid analytic Artin stack.
\end{lemma}

\begin{proof}
Since $p:U\to X$ is a surjection of stacks, it induces $[U/R]\simeq X$ for $R:=U\times_X U$.
Since $p$ is representable in quasi-analytic spaces, smooth, and surjective, it remains to show that $\Delta: X\to  X\times X$ is representable in quasi-analytic spaces.

\vskip8pt

Let $T\in \Rig_E$ and $x,y:T\to X$ be any two maps of stacks.
We need to show that the fiber product $P:=X\times_{\Delta,X\times X, (x,y)}T\to T$ is a quasi-analytic space over $T$.
Firstly, \'{e}tale-locally on $T$, $P$ is representable. 
This is because, as in \cite[Remark 3.29]{Oh-Shimizu_moduli}, the (Tate-fpqc) smooth surjection $p:U\to [U/R]=X$ is Tate-\'{e}tale surjective by the existence of \'{e}tale quasi-sections to smooth surjections.
So, there is a Tate-\'{e}tale cover $\{T_i\to T\}_{i\in I}$ such that $x|_{T_i},y|_{T_i}:T_i \to X=[U/R]$ admit local sections $x_i,y_i:T_i\to U$ for the map $p:U\to X$.
Thus, over $T_i$ we have, cf. \cite[\href{https://stacks.math.columbia.edu/tag/044V}{Tag 044V}]{stacks-project},
$$
P_i:=
X\times_{\Delta,X\times X, (x|_{T_i},y|_{T_i})}T_i \cong (U\times_X U)\times_{U\times U,(x_i,y_i)} T_i = R \times_{(s,t),U\times U,(x_i,y_i)} T_i\,,
$$
which is a quasi-analytic space.
Let $T':=\coprod_{i\in I} T_i$ and $P':=\coprod_{i\in I}P_i$.
Then, $P'=T'\times_T P$, and $P' \to P$ is a map that is representable in rigid spaces, \'{e}tale, and surjective from a quasi-analytic space.
To conclude that $P$ is a quasi-analytic space, it suffices to show that its diagonal is representable in rigid spaces.
For $S\in \Rig_E$ and any two maps $a,b:S\to P$, the base change $S\times_{P\times P} P \to S$ is a monomorphism because $P\to P\times P$ is a monomorphism.
Thus $S\times_{P\times P} P$ has representable diagonal because $S\in \Rig_E$.
Denote by $a',b':S\to P\to T$.
Then, $S\times_{P\times P} P \cong (S\times_{(a',b'),T\times T}T) \times_{P\times_T P}P$, where $S\times_{(a',b'),T\times T}T\in \Rig_E$ admits two maps to $P$ that agree when further composed with $P\to T$.
The base change along the cover $T'\to T$ is $(S\times_{(a',b'),T\times T}T') \times_{P'\times_{T'} P'}P' = (S\times_{(a',b'),T\times T}T') \times_{P'\times P'}P'\in \Rig_E$, since $P'$ is a quasi-analytic space.
Thus, $(S\times_{P\times P} P)\times_T T'\to S\times_{P\times P} P$ is a representable \'{e}tale surjection from a rigid space.
We conclude that $S\times_{P\times P} P$ is a quasi-analytic space by Lemma \ref{lem:aux_1}.
For the monomorphism $S\times_{P\times P} P \hookrightarrow S$, take an affinoid admissible open cover $\{S_\alpha\to S\}_\alpha$.
Then, $S_\alpha \times_{P\times P} P \hookrightarrow S_\alpha$ is separated since $S_{\alpha}$ is separated.
By \cite[Remark 5.1.9(i)]{EmertonGeeHellmann}, every separated quasi-analytic space is representable by a rigid space, so $S_\alpha \times_{P\times P} P\in \Rig_E$. 
Thus, by gluing over the admissible cover $\{S_\alpha\to S\}_\alpha$, we conclude that $S\times_{P\times P} P\in \Rig_E$.
\end{proof}

\subsection{Structure sheaf and global functions on a stack}
\label{structure-sheaf}
Let $p:\frX\to \Rig_E$ be a stack in groupoids over $\Rig_E$.
Then, $\frX$ inherits a canonical $G$-topology from the Tate-fpqc topology on $\Rig_E$, cf. \cite[\href{https://stacks.math.columbia.edu/tag/06NT}{Tag 06NT}]{stacks-project}. 
Namely, a collection of morphisms $\{x_i\to x\}$ in $\frX$ is a cover if and only if $\{p(x_i)\to p(x)\}$ is a Tate-fpqc cover in $\Rig_E$.
Since rigid spaces are Tate-fpqc sheaves, the functor $\cO: T\mapsto \Gamma(T,\cO_T)\cong \Hom_{\Rig_E}(T,\bbA^{1,\rig}_E)$ defines a sheaf on the site $\Rig_{E,fpqc}$.
The map $p:\frX \to \Rig_E$ induces a map of topoi $p:\Shv(\frX)\to \Shv(\Rig_E)$ such that for $\cO_{\frX}:=p^{-1}\cO$, if $x\in \frX$ lies over $T\in \Rig_E$, then we have $\cO_{\frX}(x)=\Gamma(T,\cO_T)$, cf. \cite[\href{https://stacks.math.columbia.edu/tag/06NW}{Tag 06NW}]{stacks-project}.
We refer to $\cO_{\frX}$ as the {\bf structure sheaf} of $\frX$.
The {\bf global sections} of $\cO_{\frX}$ are given by
$
\Gamma(\frX,\cO_{\frX}):=\Hom_{\Shv(\frX)}(*, \cO_{\frX}),
$
where $*$ is the final object in $\Shv(\frX)$ which associates to every object a singleton \cite[\href{https://stacks.math.columbia.edu/tag/06UN}{Tag 06UN}]{stacks-project}.
Therefore, we have
$$
\Gamma(\frX,\cO_{\frX}) = \Hom_{\Psh(\frX)}(*, \cO_{\frX}) \cong \varprojlim_{x\in \frX} \cO_{\frX}(x) = \varprojlim_{\stackrel{x:T\to\frX}{T\in \Rig_E}} \Gamma(T,\cO_T),
$$
where the last equality follows from the 2-Yoneda lemma \cite[\href{https://stacks.math.columbia.edu/tag/04SS}{Tag 04SS}]{stacks-project}.
Observe that
\begin{align*}
\varprojlim_{\stackrel{x:T\to\frX}{T\in \Rig_E}} \Gamma(T,\cO_T) 
&\cong \varprojlim_{\stackrel{x:T\to\frX}{T\in \Rig_E}} \Hom_{\Rig_E}(T,\bbA^{1,\rig}_E)
\\&= \Hom_{\Psh(\Rig_E)}(F_{\frX}^{pre}, \bbA^{1,\rig}_E)
\\&= \Hom_{\Shv(\Rig_E)}(F_{\frX}, \bbA^{1,\rig}_E)
= \Hom_{\Stk(\Rig_E)}(\frX, \bbA^{1,\rig}_E)
\end{align*}
by the 2-Yoneda lemma and the universal property of the coarse sheaf $F_{\frX}$ associated to $\frX$ \eqref{para:coarse_sheaf}.
We deduce by applying $\Hom_{\Stk(\Rig_E)}(-, \bbA^{1,\rig}_E)$ that if $\frX \simeq \coeq(R\rightrightarrows U)$ is presented as the coequalizer of two stacks, then $\Gamma(\frX,\cO_{\frX})=\eq(\Gamma(U,\cO_U)\rightrightarrows \Gamma(R,\cO_R))$ is an equalizer.
Replacing $\Rig_{E,fpqc}$ by $\Spec(\cO_E)_{fppf}$ and replacing $\bbA^{1,\rig}_E$ by $\bbA^1_{\cO_E}$ in the above argument, we get the same conclusion for (formal) stacks $\cX=[U/R]$ over $\cO_E$.

\vskip8pt

The clopen substacks of $\frX$ are then in bijection with the idempotent elements in $\Gamma(\frX,\cO_\frX)$. Namely, given a clopen substack $\frY \sub \frX$ we have $\Gamma(\frX,\cO_\frX) = \Gamma(\frY,\cO_\frY) \oplus \Gamma(\frY',\cO_{\frY'})$ with $\frY' = \frX \setminus \frY$, and one associates to $\frY$ the identity element $1_{\Gamma(\frY,\cO_\frY)} \in \Gamma(\frX,\cO_\frX)$. Conversely, given an idempotent $e \in \Gamma(\frX,\cO_\frX)$, the substack $\frY$ where $e = 1$ is a clopen substack.

\section{On connected components}
\label{section4_conncomp}

\subsection{Connected components of \texorpdfstring{$\cX_{n,\red}$}{}}

\label{subsection4.1_formalcomp}

We denote by $\cX_n = \cX_{n,K}$ the Emerton-Gee stack of projective \'etale $(\varphi,\Gamma)$-modules over $\bA_K$ of rank $n$, as introduced and studied in \cite{EmertonGee_modulicrystalline}. Let $k$ be the residue field of $K$ of size $q = p^f$. 
Let $T \subset \GL_{n/k}$ denote the diagonal subgroup. 
We will sometimes write $T_1$ for the subgroup $T$ in the case $n=1$. 
Let $Y(T)$ be the cocharacter lattice $\Hom(\bbG_m,T)$ of $T$. 
Let $\Phi^\vee \subset Y(T)$ denote the set of coroots. 
Let $\omega: G_K \ra k^\times = \bbG_m(k)$ denote the mod $p$ reduction of the Lubin--Tate character. 
Fix an embedding $\iota: k \hookrightarrow \overline{\bbF}_p$. 
For $\lambda \in Y(T)$, let $\overline{\rho}_\lambda: G_K \ra T(k) \subset \GL_n(k) \subset \GL_n(\overline{\bbF}_p)$ denote the composition $\lambda(k) \circ \omega$. 

\vskip8pt

Recall that the set $\mathcal{X}_n(\overline{\bbF}_p)$ is in bijection with conjugacy classes of continuous representations $\overline{\rho}: G_K \ra \GL_n(\overline{\bbF}_p)$. 
By an abuse of notation, we will also write $\overline{\rho}$ for the corresponding element of $|\mathcal{X}_n|$. 

\begin{thm}
\label{thm:reduced-EG}
    There is a unique bijection $\pi_0(\mathcal{X}_n)\rightarrow Y(T)/((q-1)Y(T)+\bbZ\Phi^\vee)$ that sends the connected component containing $\overline{\rho}_\lambda$ to the coset of $\lambda \in Y(T)$. Moreover, these bijections are compatible with determinants, i.e., these bijections fit in the commutative diagram 
    \begin{center}
    \begin{tikzcd}
\pi_0(\mathcal{X}_n) \arrow[r, "\det"] \arrow[d] & \pi_0(\mathcal{X}_1) \arrow[d] \\
Y(T)/((q-1)Y(T)+\bbZ\Phi^\vee) \arrow[r, "\det"] & Y(T_1)/(q-1)Y(T_1),                      
\end{tikzcd}
    \end{center}
    of bijections, where the horizontal maps are induced by the determinant map. 
\end{thm}
\begin{proof}
    The case $n=1$ follows from \cite[Proposition 7.2.17]{EmertonGee_modulicrystalline}. 
    In fact, in this case, each connected component is irreducible. 
    Now let $n$ be arbitrary. 
    There are determinant maps $\mathcal{X}_n \rightarrow \mathcal{X}_1$ and $Y(T) \rightarrow Y(T_1)$. 
    Moreover, the induced map $Y(T)/((q-1)Y(T)+\bbZ\Phi^\vee) \rightarrow Y(T_1)/(q-1)Y(T_1)$ is an isomorphism. 
    Thus there is a unique commutative diagram 
    \begin{center}
    \begin{tikzcd}
\mathcal{X}_n \arrow[r, "\det"] \arrow[d] & \mathcal{X}_1 \arrow[d] \\
Y(T)/((q-1)Y(T)+\bbZ\Phi^\vee) \arrow[r, "\det"] & Y(T_1)/(q-1)Y(T_1),                      
\end{tikzcd}
    \end{center}
    where the horizontal maps are given by the determinant and the right vertical map is induced from the bijection for $n=1$. 
    Moreover, all maps are continuous, where the objects in the bottom row are given the discrete topologies. 
    Moreover, one easily verifies that $\overline{\rho}_\lambda$ maps to $\lambda+(q-1)Y(T)+\bbZ\Phi^\vee$. 
    It suffices to show that the map $\mathcal{X}_n \rightarrow Y(T)/((q-1)Y(T)+\bbZ\Phi^\vee)$ induces a bijection $\pi_0(\mathcal{X}_n) \rightarrow Y(T)/((q-1)Y(T)+\bbZ\Phi^\vee)$. 
    Note that this map is surjective since the determinant map $\mathcal{X}_n\rightarrow \mathcal{X}_1$ is surjective on points. 

\vskip8pt

    Before continuing with the proof, we need to recall some results on the geometry of $\mathcal{X}_n$. 
    Recall that by \cite[Theorem 1.2.1]{EmertonGee_modulicrystalline}, the irreducible components of $\mathcal{X}_n$ are in bijection with irreducible representations of $\GL_n(k)$ over $\Fpbar$ (also known as Serre weights). 
    For a Serre weight $\sigma$, write $\mathcal{X}_n^\sigma \subset \mathcal{X}_n$ for the corresponding irreducible component. 
    By \cite[Theorem 6.5.1]{EmertonGee_modulicrystalline}, every irreducible component of $\mathcal{X}_n$ contains a completely reducible $\overline{\rho}$. 
    Let $D: \mathcal{X}_1^n \rightarrow \mathcal{X}_n$ be the map given by direct sum. 
    By the $n=1$ case, we see that the image of an irreducible/connected component of $\mathcal{X}_1^n$ under $D$ contains both $\overline{\rho}$ and $\overline{\rho}_\lambda$ for some $\lambda\in Y(T)$. 
    Thus every irreducible component of $\mathcal{X}_n$ contains $\overline{\rho}_\lambda$ for some $\lambda\in Y(T)$. 

\vskip8pt

    We now consider the case $n=2$. 
    In this case, the one-dimensional Serre weights are the $q-1$ distinct powers of $\iota\circ \det$ and an arbitrary Serre weight is a twist by a character of $\otimes_{j=0}^{f-1} \Sym^{s_j} k^2 \otimes_{k,\iota \circ \varphi^j} \Fpbar$ where $\varphi$ denotes the arithmetic Frobenius and $0 \leq s_j \leq p-1$ for all $j$. 
    We say that a Serre weight is Steinberg if it is a twist of $\otimes_{j=0}^{f-1} \Sym^{p-1} k^2 \otimes_{k,\iota \circ \varphi^j} \Fpbar$ and is non-Steinberg otherwise. 
    We similarly call irreducible components of $\mathcal{X}_2$ Steinberg or non-Steinberg. 
    Now every connected component (even every irreducible component) of $\mathcal{X}_2$ contains $\overline{\rho}_\lambda$ for some $\lambda\in Y(T)$. 
    Since $\overline{\rho}_\lambda$ is peu ramifie\'e (i.e.~not tr\`es ramifie\'e), $\overline{\rho}_\lambda$ is contained in a non-Steinberg irreducible component by \cite[Lemma A.4]{CEGS_LocalGeometry} and \cite[Theorem 1.2]{CEGS_ANT}. 
    We conclude that every connected component contains a non-Steinberg irreducible component. 

\vskip8pt

    We will now show that $\# \pi_0(\mathcal{X}_2) \leq q-1$ from which the result for $n=2$ follows. 
    The $q-1$ possible central characters give a product decomposition of the category of $\Fpbar[\GL_2(k)]$-modules. 
    Each subcategory with a fixed central character further decomposes into two blocks, a so-called generic block whose simples are all non-Steinberg Serre weights and a non-generic block containing a single indecomposable isomorphism class which is a Steinberg Serre weight. 
    To show the desired inequality, it suffices to show that if $\sigma$ and $\sigma'$ are non-Steinberg Serre weights belonging to a single (generic) block, then $\mathcal{X}_2^\sigma$ and $\mathcal{X}_2^{\sigma'}$ are contained in the same connected component of $\mathcal{X}_2$. 
    Since $\sigma$ and $\sigma'$ are in the same block, there is an $m\in \mathbb{N}$ and a sequence $\sigma_0 = \sigma,\sigma_1,\ldots \sigma_{m-1},\sigma_m = \sigma'$ such that there is a non-trivial $\GL_2(k)$-extension between $\sigma_{i-1}$ and $\sigma_i$ for all $1\leq i \leq m$. 
    By \cite[Theorem 1.2]{Kansal_Codim-1-intersections}, $\mathcal{X}_2^{\sigma_{i-1}} \cap \mathcal{X}_2^{\sigma_i}$ is nonempty. 
    We conclude that for all $1\leq i \leq m$, $\mathcal{X}_2^{\sigma_{i-1}}$ and $\mathcal{X}_2^{\sigma_i}$, and thus $\mathcal{X}_2^\sigma$ and $\mathcal{X}_2^{\sigma'}$, are contained in the same connected component. 

\vskip8pt

Finally, we consider the case $n>2$. 
    We will show that any fiber of the map $\mathcal{X}_n \rightarrow Y(T)/((q-1)Y(T)+\bbZ\Phi^\vee)$ is connected. 
    It suffices to show that $\overline{\rho}_\lambda$ and $\overline{\rho}_\mu$ are in the same connected component if $\lambda - \mu \in (q-1)Y(T)+\bbZ\Phi^\vee$. 
    We will break this into two cases. 
    First suppose that $\lambda - \mu \in (q-1)Y(T)$. 
    Then it is easy to see that $\overline{\rho}_\lambda$ and $\overline{\rho}_\mu$ are equal using that $\omega^{q-1}$ is the trivial character. 
    Second, suppose that $\lambda-\mu$ is a simple coroot $\alpha^\vee$. 
    Consider the standard Levi subgroup $\bG_m^{n-2} \times \GL_2 \subset \GL_n$ generated by $T$ and the root subgroups for $\pm \alpha$. 
    Then the representations $\overline{\rho}_\lambda$ and $\overline{\rho}_\mu$ factor through the $\Fpbar$-points of this Levi subgroup. 
    Let $\overline{\rho}_{\lambda,\alpha},\overline{\rho}_{\mu,\alpha}: G_K \rightarrow \GL_2(\Fpbar)$ be the corresponding projections to the $\GL_2$-factor. 
    Then $\overline{\rho}_\lambda/\overline{\rho}_{\lambda,\alpha}$ and $\overline{\rho}_\mu/\overline{\rho}_{\mu,\alpha}$ are isomorphic and $\overline{\rho}_{\lambda,\alpha}$ and $\overline{\rho}_{\mu,\alpha}$ lie in the same connected component of $\mathcal{X}_2$ by the $n=2$ case. 
    Consider the map $\mathcal{X}_{2,\Fpbar} \rightarrow \mathcal{X}_{n,\Fpbar}$ defined by $M \mapsto M \oplus M_0$ where $M_0$ is the $(\varphi,\Gamma)$-module corresponding to $\overline{\rho}_\mu/\overline{\rho}_{\mu,\alpha} \cong \overline{\rho}_\lambda/\overline{\rho}_{\lambda,\alpha}$. 
    The image of the connected component of $\mathcal{X}_{2,\Fpbar}$ containing $\overline{\rho}_{\lambda,\alpha}$ and $\overline{\rho}_{\mu,\alpha}$ contains (the conjugacy classes of) $\overline{\rho}_\lambda$ and $\overline{\rho}_\mu$. 
    We conclude that $\overline{\rho}_\lambda$ and $\overline{\rho}_\mu$ lie in the same connected component of $\mathcal{X}_n$. 
    Finally, if $\lambda-\mu \in (q-1)Y(T)+\bbZ\Phi^\vee$, then an easy induction argument using the two cases above shows that $\overline{\rho}_\lambda$ and $\overline{\rho}_\mu$ are in the same connected component of $\mathcal{X}_n$.
\end{proof}

\subsection{Connected components of \texorpdfstring{$\mathcal{X}_{n}^{\rig}$}{}}
\label{subsection4.2_rigidcomp}

\begin{thm}
\label{thm:rigid-EG-general}
Assume $p>2$. 
For any finite extension $K/\bbQ_p$, by restricting the determinant of an \'{e}tale $(\varphi,\Gamma_K)$-module of rank $n$ to the torsion subgroup of $K^\times$ via local class field theory, the connected components of $\mathcal{X}_n^{\rig}$ correspond bijectively to characters $\mu(K)\rightarrow \mathcal{O}_E^\times$, where $\mu(K)$ is the set of roots of unity in $K$.
\end{thm}

In this subsection, we give a proof of Theorem \ref{thm:rigid-EG-general} in several steps.

\begin{para}{\it First reductions.}
Let $\mu_{p^\infty}(K)$ be the $p$-power roots of unity in $K$.
Then, $\mu(K) = \mu(K)^p \times \mu_{p^\infty}(K)$, and $\mu(K)^p$ is isomorphic to the units in the residue field of $K$.
For any character $\psi: \mu(K)^p \ra \cO_E^\times$, let $\cX_n^\psi$ be the corresponding connected component of $\cX_n$ given in Theorem \ref{thm:reduced-EG}.
For any character $\chi: \mu_{p^\infty}(K)\ra \cO_E^\times$, consider the base change
$$
\begin{tikzcd}
\cX_n^{\psi,\chi} \arrow[r, hook] \arrow[d] & \cX_n^\psi \arrow[d, "\det(\rho)|_{\mu_{p^\infty}(K)}"]\\
\Spf(\cO_E) \arrow[r, "\chi", hook] & \Spf(\cO_E[\mu_{p^\infty}(K)])
\end{tikzcd}
$$
where the right vertical map is induced by the determinant $\det: \cX_n \ra \cX_1$, followed by restriction of the character to the $p$-power torsion subgroup $\cX_1 \ra \Spf(\cO_E[\mu_{p^\infty}(K)])$, and the bottom map is induced by the character $\chi:\mu_{p^\infty}(K)\ra \cO_E^\times$.
The fiber product $\cX_n^{\psi,\chi}$ is a closed substack of $\cX_{n}^\psi$ whose reduction coincides with the reduction of $\cX_{n}^\psi$, because 
$$
\Spf(\cO_E)_\red = \Spec(k) \xrightarrow{\chi \pmod{\varpi}} \Spec(k) = \Spec(k[x]/(x^{p^r}-1))_{\red} = \Spf(\cO_E[\mu_{p^\infty}(K)])_\red
$$ 
is the identity map on $\Spec(k)$, where we write $\mu_{p^\infty}(K) \cong \Z/p^r \Z$ for some $r\geq 0$. 
Thus, $|\cX_n^{\psi,\chi}| = |\cX_n^\psi|$, and hence $\cX_n^{\psi,\chi}$ is also connected.
\vskip8pt 

Passing to rigid generic fibers\footnote{This functor commutes with fiber products, as $(-)^{\rig}_{\rm pre}$ commutes with presheaf fiber products by definition, and sheafification commutes with finite limits.}, we get a pullback diagram of Tate-fpqc stacks on $\Rig_E$,
$$
\begin{tikzcd}
\cX_n^{\psi,\chi,\rig} \arrow[r, hook] \arrow[d] & \cX_n^{\psi,\rig} \arrow[d, "\det(\rho)|_{\mu_{p^\infty}(K)}"]\\
\Spf(\cO_E)^\rig \arrow[r, "\chi", hook] & \Spf(\cO_E[\mu_{p^\infty}(K)])^\rig
\end{tikzcd}
$$
where the bottom map is 
$$
\Spf(\cO_E)^\rig = \Sp(E) \xrightarrow{\chi} \coprod_{\mu_{p^\infty}(K) \ra \cO_E^\times } \Sp(E) = \Sp(E[\mu_{p^\infty}(K)]) =  \Spf(\cO_E[\mu_{p^\infty}(K)])^\rig
$$
the inclusion of the closed point corresponding to the character $\chi$ to the maximal spectrum of the semisimple algebra $E[\mu_{p^\infty}(K)] = \prod_{\mu_{p^\infty}(K)\ra \cO_E^\times} E$.
Thus, $\cX_n^{\psi,\chi,\rig}$ is a clopen substack of $\cX_n^{\psi,\rig}$, and we can decompose $\cX_n^\rig$ into a disjoint union of  clopen substacks
$$
\cX_n^\rig = \coprod_{\mu(K)^p\xra{\psi} \cO_E^\times} \coprod_{\mu_{p^\infty}(K)\xra{\chi} \cO_E^\times} \cX_n^{\psi,\chi,\rig}
$$
indexed by characters $\mu(K)\to \cO_E^\times$.
It remains to show that $\cX_n^{\psi,\chi,\rig}$ is connected.

\end{para}

\vskip8pt

\begin{para}{\it Basic setup.}
Fixing characters $\psi,\chi$ as above, we write $\cX:=\cX_n^{\psi,\chi}$ for simplicity.
Since $\cX$ is a closed substack of the stack $\cX_n$, it is a noetherian formal algebraic stack with $\cX_\red$ being an algebraic stack of finite type over $k$ by \cite[Theorem 1.2.1]{EmertonGee_modulicrystalline}, hence it is locally formally of finite type over $\Spf(\cO_E)$ by Lemma \ref{lem:finite_type_red}.
Fix a smooth surjection
$$
\pi: \cU := \coprod_{i\in I} \Spf(A_i) \twoheadrightarrow \cX
$$
as in \S\ref{para:presentation}.
By \cite[Proposition 3.4.12]{EmertonGee_modulicrystalline} and base change, the diagonal of $\cX$ is affine of finite presentation.
Thus,
$$
\cR:=\cU\times_{\cX}\cU \cong \coprod_{j\in J}\Spf(B_j)
$$
is a disjoint union of formal schemes $\Spf(B_j)\in \FS_{\cO_E}$. 
We have a coequalizer diagram
\[
\begin{tikzcd}
    \cR \arrow[r, shift left=1ex, "s"] \arrow[r, shift right=1ex, "t"] & \cU \arrow[r, "\pi"] & \cX
\end{tikzcd}
\]
with $\pi$ smooth surjective and $s,t$ smooth.
So, we can take the rigid generic fiber $\cX^\rig$, which sits in the following coequalizer diagram
\[
\begin{tikzcd}
    \cR^\rig \arrow[r, shift left=1ex, "s^\rig"] \arrow[r, shift right=1ex, "t^\rig"] & \cU^\rig \arrow[r, "\pi^\rig"] & \cX^\rig
\end{tikzcd}.
\]

\end{para}

\vskip8pt

\begin{para}{\it Input from (Galois) deformation theory.}
We study geometric properties of the rings $A_i$ and $B_j$. 
The goal is to conclude that after possibly refining $I$ and $J$, we may assume both $A_i$ and $B_j$ are $\cO_E$-flat noetherian normal domains.
\vskip8pt

Let us first recall the definition of {\it versal rings} at finite type points of stacks in groupoids on the big fppf site of an affine noetherian base scheme $S=\Spec(\Lambda)$, which we will take to be $\Spec(\cO_E)$ in applications.
Let $p:\cX\ra S_{fppf}$ be a category fibered in groupoids, and let $x_0:\Spec(k)\ra \cX$ (which, by 2-Yoneda lemma, is just an object $x_0\in \cX(k)$) be a finite type point, i.e., there is a finite ring map $\Lambda \ra k$ where $k$ is a field.
Following \cite[\href{https://stacks.math.columbia.edu/tag/06GC}{Tag 06GC}]{stacks-project}, we define the base category $\cC_\Lambda$ to be the category whose objects are local Artin $\Lambda$-algebras $(A,\frm_A)$ with a $\Lambda$-algebra isomorphism $A/\frm_A\cong k$, and whose morphisms are local $\Lambda$-algebra maps compatible with the identifications.
This sits as a full subcategory in its ``completion" $\widehat{\cC}_\Lambda$  \cite[\href{https://stacks.math.columbia.edu/tag/06GW}{Tag 06GW}]{stacks-project} defined similarly with its objects being complete noetherian local $\Lambda$-algebras $(R,\frm_R)$ with $R/\frm_R\cong k$.
Following \cite[\href{https://stacks.math.columbia.edu/tag/07T2}{Tag 07T2}]{stacks-project}, let 
$$
p_{x_0}:\cF_{\cX,k,x_0} \ra \cC_\Lambda
$$
be the category cofibered in groupoids over $\cC_\Lambda$ whose objects are ``deformations of $x_0$ to local Artin $\Lambda$-algebras" and whose morphisms signify compatible deformations.
This is a {\it predeformation category} in the sense of \cite[\href{https://stacks.math.columbia.edu/tag/06GS}{Tag 06GS}]{stacks-project} by \cite[\href{https://stacks.math.columbia.edu/tag/07T5}{Tag 07T5}]{stacks-project}.
To the category $\cF_{\cX,k,x_0}$, or more generally, to any category $\cF\ra \cC_\Lambda$ cofibered in groupoids over $\cC_\Lambda$, we can associate the completion category $\widehat{\cF}\ra \widehat{\cC}_\Lambda$ cofibered over $\widehat{\cC}_\Lambda$, consisting of {\it formal objects} of $\cF$, \cite[\href{https://stacks.math.columbia.edu/tag/06H3}{Tag 06H3}]{stacks-project}: a formal object $\xi = (R,\xi_n,f_n)$ of $\cF$ consists of an object $R\in \widehat{\cC}_\Lambda$, a collection indexed by $n\in \bbN$ of objects $\xi_n \in \cF(R/\frm_R^n),$ and morphisms $f_n:\xi_{n+1}\ra \xi_n$ lying over $R/\frm_R^{n+1}\twoheadrightarrow R/\frm_R^n$.
For the category fibered in groupoids $p:\cX\ra S_{fppf}$ over $S_{fppf}$, we similarly define {\it formal objects} of $\cX$, \cite[\href{https://stacks.math.columbia.edu/tag/07X4}{Tag 07X4}]{stacks-project}, to be $\xi = (R,\xi_n,f_n)$, where $(R,\frm_R)$ is a noetherian complete local $S$-algebra such that $R/\frm_R$ is a field of finite type over $S$, $\xi_n\in \cX(R/\frm_R^n)$, and $f_{n}:\xi_{n}\ra\xi_{n+1}$ are morphisms of $\cX$ lying above $\Spec(R/\frm_R^n)\ra\Spec(R/\frm_R^{n+1})$.
Then, for the choice of $k:=R/\frm_R$ and $x_0:=\xi_1$, the formal object $(R,\xi_n,f_n)$ of $\cX$ gives a formal object $(R,\xi_n,f_n)$ of $\cF_{\cX,k,x_0}$.
We say that {\it $R$ is a versal ring to $\cX$ at $x_0$} (with the data $\xi=(R,\xi_n,f_n)$ implicitly understood) if the formal object $(R,\xi_n,f_n)$ of $\cF_{\cX,k,x_0}$ is a {\it versal formal object}, \cite[\href{https://stacks.math.columbia.edu/tag/06HR}{Tag 06HR}]{stacks-project}, which means that the induced map, whose precise description is explicated in \cite[\href{https://stacks.math.columbia.edu/tag/06HC}{Tag 06HC}]{stacks-project}, 
$$
\underline{\xi}: \underline{R}|_{\cC_\Lambda} \ra \cF_{\cX,k,x_0}
$$
is {\it smooth}, \cite[\href{https://stacks.math.columbia.edu/tag/06HG}{Tag 06HG}]{stacks-project}, as a morphism of categories cofibered in groupoids over $\cC_\Lambda$.
By \cite[\href{https://stacks.math.columbia.edu/tag/06T5}{Tag 06T5}]{stacks-project}, for any predeformation category $\cF$ that admits at least one versal formal object, it has a unique (up to isomorphism) minimal versal formal object such that any versal object is the pushforward of a minimal versal object along a power series ring extension. 
In particular, versal rings of $\cX$ at $x_0$, corresponding to versal formal objects of the predeformation category $\cF_{\cX,k,x_0}$, are unique up to smooth equivalence when they exist, in the sense that there is a ``minimal" versal ring $R_{\min}$ of $\cX$ at $x_0$ such that any other versal ring $R$ of $\cX$ at $x_0$ is a power series ring $R\cong R_{\min}[[t_1,\dots,t_r]]$ for some $r\geq 0$.
Consequently, if $R_1,R_2$ are two versal rings of $\cX$ at $x_0$, then 
\begin{numequation}
\label{eqn:smooth-equiv}
\text{either} \quad R_1=R_2[[t_1,\dots,t_r]] \quad \text{ or }\quad R_2=R_1[[t_1,\dots,t_r]] .
\end{numequation}
\vskip8pt

Back to our formal algebraic stack $\cX=\cX_{n}^{\psi,\chi}$.
For any closed point $u\in \Spf(A_i)$ corresponding to a maximal ideal $\frm_u$ of $A_i$, let $x$ denote the image of $u$ under the map 
$
\pi: \Spf(A_i) \ra \cX.
$
Then, 
$
x: \Spec (\kappa(u)) \xrightarrow{u} \Spf(A_i) \xrightarrow{\pi} \cX
$
is a finite type point of $\cX$, i.e., $\kappa(u):= A_i/\frm_u$ is a finite type $\cO_E$-algebra, because $A_i/I_i$ is a finite type $k$-algebra that surjects onto $\kappa(u)= A_i/\frm_u$. 
Since $\pi: \Spf(A_i) \ra \cX$ is smooth, the induced map between predeformation categories $\pi: \cF_{\Spf(A_i),\kappa(u),u} \ra \cF_{\cX,\kappa(u),x}$ is smooth, \cite[\href{https://stacks.math.columbia.edu/tag/07WK}{Tag 07WK}]{stacks-project}.
Since $\Spf(A_i) = \varinjlim_n \Spec(A_i/I_i^n)$ as a sheaf on $\Spec(\cO_E)_{fppf}$, it is easy to check that the complete noetherian local ring $(A_{i,\frm_u})^{\wedge}_{\frm_u}$ prorepresents $\cF_{\Spf(A_i),\kappa(u),u}$.
So, $(A_{i,\frm_u})^{\wedge}_{\frm_u}$ is a versal ring of $\cX$ at $x$.
Let $\overline{\rho}_x: G_K \ra \GL_n(\kappa(u))$ 
be the Galois representation attached to $x$.
Another versal ring of $\cX=\cX_n^{\psi,\chi}$ at $x$ is given by the ring $R_{\overline{\rho}_x}^{\square,\chi}:= R_{\overline{\rho}_x}^{\square} \otimes_{\cO_E[\mu_{p^\infty}(K)],\chi} \cO_E$, where $R_{\overline{\rho}_x}^\square$ is the framed Galois deformation ring that is versal to $\cX_n$ at $x$ \cite[3.6.3]{EmertonGee_modulicrystalline}.
Indeed, by definition $R_{\overline{\rho}_x}^{\square,\chi}$ prorepresents the framed deformations $\rho_A$ of $\overline\rho_x$ to local Artin $\cO_E$-algebras $A$ with residue field $\kappa(u)$ such that $\det(\rho_A)|_{\mu_{p^\infty}(K)} = \chi$. On the other hand,  since $A$ is local of residual characteristic $p$, the prime-to-$p$ integer $m:=|\mu(K)^p|$ is invertible in $A$.
For any $\zeta\in \mu(K)^p$, let $u:=\det(\rho_A)(\zeta)/\psi(\zeta)\in 1+\frm_A$. Then, since $u^m=1$, 
$$
0= u^{m}-1 = (u-1)(1+u+\dots +u^{m-1}),
$$
where $(1+u+\dots +u^{m-1})\equiv m \pmod{\frm_A}$ is a unit.
Thus, $u-1=0$, i.e., $\det(\rho_A)|_{\mu(K)^p}=\psi$.
Hence, we have a morphism $\Spf(R_{\overline{\rho}_x}^{\square,\chi})\cong \Spf(R_{\overline\rho_x}^\square)\times_{\cX_n} \cX_{n}^{\psi,\chi}\to\cX_{n}^{\psi,\chi}$ that is versal at $x$, which is the base change of the morphism $\Spf(R_{\overline{\rho}_x}^{\square})\to \cX_n$ in \cite[3.6.3]{EmertonGee_modulicrystalline}.

\vskip8pt

Since any two versal rings are smoothly equivalent in the sense of \eqref{eqn:smooth-equiv}, we have
$$
\text{either}\quad 
R_{\overline{\rho}_x}^{\square,\chi} = (A_{i,\frm_u})^\wedge_{\frm_u}[[t_1,\dots,t_r]]
\quad \text{ or }\quad (A_{i,\frm_u})^\wedge_{\frm_u} =  R_{\overline{\rho}_x}^{\square,\chi}[[t_1, \dots, t_r]]\,,
$$
as local $\cO_E$-algebras for some $r\geq 0$.
Under the assumption that $p>2$, $R_{\overline{\rho}_x}^{\square,\chi}$ is $\cO_E$-flat and normal by \cite[Corollaries 4.5, 4.19]{BIP_def}.
Thus by Lemma \ref{lem:versal_nrm}, in either case we have that $\widehat{\cO}_{\Spf(A_i),u} = (A_{i,\frm_u})^\wedge_{\frm_u}$ is $\cO_E$-flat and normal.

\begin{lemma}
\label{lem:versal_nrm}
Suppose $R$ is a noetherian ring.
Then for any $n\in \bbN$, $R$ is normal if and only if $R[[t_1,\dots,t_n]]$ is normal.
\end{lemma}
\begin{proof}
For noetherian $R$, the inclusion $f:R\ra R[[t_1,\dots,t_n]]$ is flat by \cite[\href{https://stacks.math.columbia.edu/tag/05CZ}{Tag 05CZ}]{stacks-project} with an $R$-linear retraction $r: R[[t_1,\dots,t_n]]\twoheadrightarrow R$ so that $r\circ f = \id_R$.
Then, $f$ is faithfully flat: the retraction $r$ splits $R[[t_1,\dots,t_n]] \cong R \oplus (R[[t_1,\dots,t_n]]/R)$ as $R$-module, hence for any $R$-module $M$, the base change 
$$
R[[t_1,\dots,t_n]] \otimes_R M \cong M \oplus ((R[[t_1,\dots,t_n]]/R) \otimes_R M)
$$
is zero if and only if $M=0$.
By \cite[\href{https://stacks.math.columbia.edu/tag/033G}{Tag 033G}]{stacks-project}, $R$ is normal if $R[[t_1,\dots,t_n]] $ is normal.
Conversely, if $R$ is normal, then it is a product of finitely many normal domains $R_1,\dots, R_s$ by \cite[\href{https://stacks.math.columbia.edu/tag/030C}{Tag 030C}]{stacks-project}.
Hence, $R[[t_1,\dots,t_n]]=\prod_{i=1}^s (R_i[[t_1,\dots,t_n]])$.
Since $R_i$ is a domain, $R_i[[t_1,\dots,t_n]]$ is normal by \cite[\href{https://stacks.math.columbia.edu/tag/0BI0}{Tag 0BI0}]{stacks-project}. 
Hence, $R[[t_1,\dots,t_n]]$ is normal.
\end{proof}

We claim that $A_i$ is $\cO_E$-flat. 
Since $\cO_E$ is a DVR, it suffices to show that $A_i$ is $\varpi$-torsionfree. 
Since the completion $A_{i,\frm_u} \ra (A_{i,\frm_u})^\wedge_{\frm_u}$ is faithfully flat, $A_{i,\frm_u}$ is $\varpi$-torsionfree. 
Indeed, consider the multiplication-by-$\varpi$ exact sequence
$$
0 \ra A_{i,\frm_u}[\varpi] \ra A_{i,\frm_u} \xrightarrow{\varpi} A_{i,\frm_u} \,.
$$
Base change along the faithfully flat map $A_{i,\frm_u} \ra (A_{i,\frm_u})^\wedge_{\frm_u}$ yields an exact sequence
$$
0 \ra A_{i,\frm_u}[\varpi] \ot_{A_{i,\frm_u}} (A_{i,\frm_u})^\wedge_{\frm_u}\ra (A_{i,\frm_u})^\wedge_{\frm_u} \xrightarrow{\varpi} (A_{i,\frm_u})^\wedge_{\frm_u} \,,
$$
which implies that $A_{i,\frm_u}[\varpi] \ot_{A_{i,\frm_u}} (A_{i,\frm_u})^\wedge_{\frm_u}=0$, since $(A_{i,\frm_u})^\wedge_{\frm_u}$ is $\varpi$-torsionfree.
By faithful flatness, we get $A_{i,\frm_u}[\varpi] = 0$.
By noetherianness of $A_i$, the $\varpi$-torsion elements $A_i[\varpi]$ in $A_i$ form a coherent sheaf on $\Spec(A_i)$ whose support is closed. 
Its stalks $A_i[\varpi]_{\frm} = (A_{i,\frm})[\varpi] = 0$ are zero at all closed points $\frm\in \Spec(A_i)$, hence the support is empty and $A_i[\varpi] = 0$. 
We conclude that $A_i$ is $\cO_E$-flat.
\vskip8pt

We claim that $A_i$ is normal.
Since $\Spf(A_i)\in \FS_{\cO_E}$, $A_i$ is a quotient of 
$$
A_{n,m} = \cO_E\langle x_1,\dots, x_n \rangle [[y_1,\dots,y_m]].
$$ 
By \cite[p. 476]{Conrad_irred}, $A_{n,m}$ is an excellent ring, so is the quotient $A_i$.
The normal locus in excellent schemes is Zariski-open, \cite[7.8.6]{EGA_IV_2}. 
At every closed point $\frm_u\in \Spec(A_i)$, since $(A_{i,\frm_u})^\wedge_{\frm_u}$ is normal and the completion map of the noetherian local ring $A_{i,\frm_u} \ra (A_{i,\frm_u})^\wedge_{\frm_u}$  is faithfully flat \cite[\href{https://stacks.math.columbia.edu/tag/00MC}{Tag 00MC}]{stacks-project}, $A_{i,\frm_u}$ is normal \cite[\href{https://stacks.math.columbia.edu/tag/033G}{Tag 033G}]{stacks-project}.
Thus, the normal locus is Zariski-open and contains all closed points of $\Spec(A_i)$. So, $A_i$ is normal.
\vskip8pt

Since any noetherian normal ring is a finite product of noetherian normal domains \cite[\href{https://stacks.math.columbia.edu/tag/030C}{Tag 030C}]{stacks-project}, we may and do assume that $A_i$ is an $\cO_E$-flat noetherian normal domain.
\vskip8pt

As for $B_j$, given any $j\in J$ and closed point $v\in \Spf(B_j)$ corresponding to a maximal ideal $\frm_v$ of $B_j$, let $y$ denote the image of $v$ under the smooth map 
$$
\Spf(B_j) \hookrightarrow \cR \xrightarrow{s} \cU \xrightarrow{\pi} \cX
$$
so that $y$ is a finite type point of $\cX$, corresponding to a Galois representation $\overline{\rho}_y: G_K \ra \GL_n(\overline{\bbF}_p)$.
The same argument as above shows that $B_j$ is an $\cO_E$-flat noetherian normal ring, and we may and do assume that $B_j$ is an $\cO_E$-flat noetherian normal domain.
\end{para}

\vskip8pt

\begin{para}{\it Lifting connectivity by de Jong's theorem.}
By the previous step, in the diagram
\[
\begin{tikzcd}
    \coprod_{j\in J} \Spf(B_j) \arrow[r, shift left=1ex, "s"] \arrow[r, shift right=1ex, "t"] & \coprod_{i\in I} \Spf(A_i) \arrow[r, "\pi"] & \cX \,,
\end{tikzcd}
\]
we may assume that $A_i$ and $B_j$ are $\cO_E$-flat normal domains.
So for $\Spf(A_i)$ and $\Spf(B_j)$, the natural map from their rings of formal functions to the rings of bounded analytic functions on their rigid analytic generic fibers, which exists by the construction of $(-)^\rig$ \cite[(7.1.8.2)]{deJongCrystalline}, is an isomorphism by \cite[7.4.1]{deJongCrystalline}.
In particular, as a rigid space $X$ is connected if and only if $\Gamma(X,\cO_X)$ has no nontrivial idempotents,  $\Spf(A_i)^\rig$ and $\Spf(B_j)^\rig$ are connected. 
By the coequalizer diagram
$$
\begin{tikzcd}
\coprod_{j\in J} \Spf(B_j)^\rig \arrow[r, shift left=1ex, "s^\rig"] 
\arrow[r, shift right=1ex, "t^\rig"] & \coprod_{i\in I} \Spf(A_i)^\rig \arrow[r, "\pi^\rig"] 
& {\cX^\rig} 
\end{tikzcd},
$$
an idempotent $e^\rig\in \Gamma(\cX^\rig,\cO_{\cX^\rig})$ is a collection of idempotents on the connected $\Spf(A_i)^\rig$ (hence either 0 or 1) that are compatible when pulled back to the connected $\Spf(B_j)^\rig$.
By \cite[7.4.1]{deJongCrystalline}, this idempotent $e^\rig$ comes from an idempotent $e\in \Gamma(\cX,\cO_{\cX})$.
Since $\cX$ is connected, $e=e^\rig$ is trivial, and hence $\cX^\rig=\cX^{\psi,\chi,\rig}_n$ is also connected. 
\hfill\qedsymbol

\begin{cor}
\label{cor:comp_Robba}
The morphism 
$$
\pi_n: \mathcal{X}_n^\rig \ra \frX_n
$$
defined in \cite[Corollary 5.1.30]{EmertonGeeHellmann} induces a bijection on the connected components of both sides.
In particular, the connected components of $\mathfrak{X}_n$ are in bijection (via the determinant) with the connected components of $\mathfrak{X}_1$ and so with the connected components of $\mathcal{W}$, where $\cW = (\Spf(\cO_E[[\cO_K^\times]]))^\rig$ denotes the rigid analytic space of continuous characters of $\cO_K^\times$.
\end{cor}
\begin{proof}
Clopen substacks of a stack $\frX$ on $\Rig_E$ correspond to idempotents in 
$$
\Gamma(\frX,\cO_{\frX}) = \varprojlim_{\substack{\Sp(A)\ra \frX\\\Sp(A)\in \Rig_E}} \Gamma(\Sp(A),\cO_{\Sp(A)})\,,
$$
and connected components of $\frX$ correspond to primitive idempotents in $\Gamma(\frX,\cO_{\frX})$.
Using $\det:\frX_n\ra \frX_1$, we decompose $\frX_n =\coprod_{\chi\in \widehat{\mu(K)}} \frX_n^\chi$ into a disjoint union of clopen substacks $\frX_n^\chi$ labelled by characters $\chi:\mu(K)\ra E^\times$, where $\frX_n^\chi$ classifies $(\varphi,\Gamma)$-modules $D$ of rank $n$ over $\cR_{K,A}$ whose determinant character $\delta_A:\cO_K^\times \ra A^\times$ restricts to the character $\chi$ on the roots of unity $\mu(K) \subset \cO_K^\times$.
We claim that $\frX_n^\chi$ is connected.
Let $e\in \Gamma(\frX_n^\chi,\cO_{\frX_n^\chi})$ be any idempotent.
The map 
$$
\pi_n: \mathcal{X}_n^\rig \ra \frX_n
$$ 
is compatible with the determinant map $\det$, so it induces a map
$$
\pi_n^\chi: \mathcal{X}_n^{\chi,\rig} \ra \frX_n^\chi
$$
where $\mathcal{X}_n^{\chi,\rig}$ is connected by Theorem \ref{thm:rigid-EG-general}.
Hence, $(\pi_n^{\chi})^*(e)\in \{0,1\} \subset \Gamma(\mathcal{X}_n^{\chi,\rig},\cO_{\mathcal{X}_n^{\chi,\rig}})$ must be a trivial idempotent.
We show that $e\in \Gamma(\frX_n^\chi,\cO_{\frX_n^\chi})$ equals $(\pi_n^{\chi})^*(e) \in \{0,1\}$, which means that for any morphism $f:\Sp(A)\ra \frX_n^\chi$, the pullback $f^*e\in \Gamma(\Sp(A),\cO_{\Sp(A)}) = A$ is constantly $0$ or $1$.
\vskip8pt

Firstly we address the case when $A=L$ is a finite extension of $E$.
Recall Kedlaya's slope filtration theorem, which states that there exists a unique sequence of rational numbers
$$
\mu_1<\dots <\mu_r
$$
and a unique filtration on $D$ by saturated $(\varphi,\Gamma)$-submodules over $\cR_{K,L}$
$$
0 = D_0 \subset D_1\subset \dots \subset D_r = D
$$
whose successive quotients $D_i/D_{i-1}$ are semistable of slope $\mu_i$ for $1\leq i\leq r$; moreover, $D$ is \'{e}tale if and only if $D$ is semistable of slope $0$, which means that $\mu(D)=0$ and any saturated subobject $M\subset D$ has slope $\mu(M)\geq \mu(D)=0$, cf. \cite[\S4.3]{Wang_opw}.
In fact, $D$ and its ``semisimplification" $\bigoplus_{i=1}^r (D_i/D_{i-1})$ lie in the same connected component.
From 
$$
0 \ra D_{r-1} \ra D \ra D/D_{r-1} \ra 0
$$
we can view the isomorphism class of $D$ as a line in the affine space $\Ext_{\varphi,\Gamma}^1(D/D_{r-1},D_{r-1})$ over $L$, which gives a family of $(\varphi,\Gamma)$-modules $D_{t}$ parametrized by $t\in \bbA^{1,\rig}_L$.
When $t\neq 0$, $D_t\cong D$, while for $t=0$ it gives the split extension $D_{r-1}\oplus D/D_{r-1}$.
By induction, we see that $D$ and $\bigoplus_{i=1}^r (D_i/D_{i-1})$ lie in the same connected component of $\frX_n^\chi$.
Enlarging the coefficient field from $L$ to a finite extension $L'$, we can find $\alpha_i \in (L')^\times $ such that 
$$
\frac{1}{f_K} v_p(\alpha_i) = -\mu_i
$$
for $1\leq i\leq r$, where $f_K$ is the inertia degree of $K/\bbQ_p$.
Then, the unramified twist 
$$
\bigoplus_{i=1}^r (D_i/D_{i-1})(\unr(\alpha_i)): \Sp(L')\ra \frX_n^\chi
$$
is semistable of slope $0$ by \cite[Example 4.14, Remark 4.13(ii)]{Wang_opw}, and it belongs to the same connected component as $\bigoplus_{i=1}^r (D_i/D_{i-1})$ since $\bbG_m^r$ is connected.
Thus, we have that $f^*(e) = (\pi_n^\chi)^*(e) \in \{0,1\}$, when $f:\Sp(L)\ra \frX_n^\chi$ classifies a $(\varphi,\Gamma)$-module over $\cR_{K,L}$ for any finite extension $L$ of $E$.
\vskip8pt

Finally we consider the general affinoid case.
Let $f: \Sp(A)\ra \frX_n^\chi$ be any morphism.
By the above, for any $x\in \Sp(A)$, we have $f^*(e)|_x = (\pi_n^{\chi})^*(e) \in \{0,1\}$.
Since affinoid algebras are Jacobson, there exists a nilpotent element $\vep\in A$ such that $f^*(e)= (\pi_n^{\chi})^*(e)+\vep$. 
If $ (\pi_n^{\chi})^*(e)=0$, then $f^*(e)=\vep$ is both nilpotent and idempotent, hence equals $0=(\pi_n^\chi)^*(e)$.
If $ (\pi_n^{\chi})^*(e)=1$, then $f^*(e)=1+\vep$ is idempotent, which implies that $-\vep = 1-f^*(e)$ is both nilpotent and idempotent\footnote{We have $(1-f^*(e))^2=1-2f^*(e)+f^*(e)^2 = 1-2f^*(e)+f^*(e) = 1-f^*(e)$.} and hence $\vep=0$, so $f^*(e)= (\pi_n^{\chi})^*(e) \in \{0,1\}\subset A$.
\end{proof}

\end{para}

\section{Global functions on \texorpdfstring{$\frX_n$}{}}
\label{section5_globalfunc}
\addtocounter{subsection}{1}

\begin{assumption}
\label{hypo:reduced}
The ring of global functions on $\frX_n$ has no non-zero locally nilpotent elements, in the sense that if $f\in \Gamma(\frX_n,\cO_{\frX_n})$ is such that for any affinoid $E$-algebra $A$ and any map $h:\Sp(A)\ra \frX_n$, the pullback $h^*f\in A$ is nilpotent, then $f=0\in\Gamma(\frX_n,\cO_{\frX_n})$.
\end{assumption}

\begin{prop}
Assumption \ref{hypo:reduced} holds if the versal rings at all rigid analytic points $x\in\frX_n$ are reduced.
In particular, Assumption \ref{hypo:reduced} holds when $n=2$ by \cite[Proposition 5.1.27]{EmertonGeeHellmann}. 
The required reducedness is expected in general \cite[Remark 5.1.28]{EmertonGeeHellmann}.
\end{prop}
\begin{proof}
Let $f\in \Gamma(\frX_n)$ be a global function on $\frX_n$ such that for any affinoid $E$-algebra $A$ and any map $h:\Sp(A)\ra \frX_n$, $h^*f\in A$ is nilpotent. 
\vskip8pt

We claim that $f|_A:=h^*f=0$ for all $\Sp(A)$, and therefore $f=0$.
Let $D_A \in \frX_n(A)$ be the rank $n$ $(\varphi,\Gamma)$-module classified by $h$.
Since localization gives an injection
$$
A \hookrightarrow \prod_{x\in \Sp(A)} A_{\frm_x}
$$
as $\frm_x$ varies over maximal ideals of $A$, it suffices to show that the image of $f|_A$ in $A_{\frm_x}$ is zero for all $x\in \Sp(A)$.
Since $A_{\frm_x}$ is local noetherian, by the Krull intersection theorem
$
\bigcap_{i\geq 1} \frm_x^iA_{\frm_x}=0,
$
so it suffices to show that $f|_A \in \frm_x^i$. 
That is to say, $f|_{A/\frm_x^i} = 0$ in $A_{}/\frm_x^i$ for all $x\in \Sp(A)$ and $i\geq 1$.
\vskip8pt 

Let $R$ be a versal ring at the rigid analytic point $x:\Sp(\kappa(x))\ra \frX_n$ with associated map $\Spf(R)\ra \frX_n$, in the sense of formal deformation theory \cite[Definition 3.19]{Wu_translation}.
That is, $(R,\frm_R)$ is a complete noetherian local $\kappa(x)$-algebra, and it comes with $(\varphi,\Gamma)$-modules $D_{R/\frm_R^i}\in \frX_n(R/\frm_R^i)$ for all $i\geq 1$ satisfying compatibility $D_{R/\frm_R^{i+1}} \otimes_{R/\frm_R^{i+1}} R/\frm_R^{i} \cong D_{R/\frm_R^{i}}$, such that they induce a formally smooth map $\Spf(R) \ra \cF_{\frX_n,x}$ to the deformation functor of $D_{x}:=D_{\kappa(x)}$ to local Artin $\kappa(x)$-algebras.
Compatibility implies that $f|_{R/\frm_R^{i+1}} \equiv f|_{R/\frm_R^i} \pmod{\frm_{R}^{i}}$.
So, $f$ defines an element $f_R \in \varprojlim_i R/\frm_R^i = R$.
By \cite[Theorem 5.1.20(ii)]{EmertonGeeHellmann} and smooth equivalence of versal rings \cite[\href{https://stacks.math.columbia.edu/tag/0DQX}{Tag 0DQX}]{stacks-project}, we may and do assume that $R=\widehat{B}_{\frm_y}$ is the complete local ring for a formally smooth map $\Sp(B)\ra \frX_n$ at a point $y\in \Sp(B)$ such that $\Sp(\kappa(y))\ra \frX_n$ is isomorphic to $x$.
Then, $f_R \in {\rm Nil}(R)=\{0\}$ by reducedness of $R$.
But since $A/\frm_x^i$ is a local Artinian $\kappa(x)$-algebra and $D_{A/\frm_x^i}$ deforms $D_{x}$, the versality of $R$ (i.e., formal smoothness of the map $\Spf(R) \ra \cF_{\frX_n,x}$) implies that there exist $j\geq 1$ and a map $\psi: \Sp(A/\frm_x^i) \ra \Sp(R/\frm_R^j)$ that realizes $D_{A/\frm_x^i} \cong \psi^*(D_{R/\frm_R^j})$ as a pullback of the formal object.
Hence, $f|_{A/\frm_x^i}=\psi^*(f_R \pmod{ \frm_R^j}) = \psi^*(0) = 0$.
\end{proof}

\begin{thm}
\label{thm:EG_fcns}
Under {Assumption} \ref{hypo:reduced}, the morphism
$$
\Gamma(\frX_n) \xleftarrow{(\omega_n,\det)^*} \Gamma(\WT_{n,E}) \wot_{\Gamma(\WT_{1,E})} \Gamma(\frX_1)
$$
defined in \cite[Conjecture 5.1.36]{EmertonGeeHellmann} is an isomorphism, where $\omega_n:\frX_n \ra \WT_{n,E}$ sends a rank $n$ $(\varphi,\Gamma)$-module $D$ to the characteristic polynomial of its Sen operator $\Theta_D$.
\end{thm}

The rest of this section is devoted to the proof of Theorem \ref{thm:EG_fcns}.

\begin{para}{\it Embedding functions $\Gamma(\frX_n)$ into $\Gamma(\cT^n)$.}
Let $B$ be the upper triangular Borel subgroup of $\GL_n$ with $T\subset B$ the diagonal torus.
Recall we have the morphisms of stacks
$$
\begin{tikzcd}
& \arrow[dl, "\beta_B"'] \frX_B \arrow[dr, "\alpha_B"] & \\
\frX_n &  & \frX_T
\end{tikzcd}
$$
where $\beta_B$ is the map that forgets the chosen triangulation $\Fil^\bullet(D)$ on a trianguline $(\varphi,\Gamma)$-module $D$, and $\alpha_B$ is the map taking the associated graded of $\Fil^\bullet(D)$, cf. \cite[\S5.3]{EmertonGeeHellmann}.

\begin{lemma}
\label{lem:fcn-embed}
On global functions, pullbacks along the above two maps satisfy:
\begin{enumerate}[(i)]
\item 
Under Assumption \ref{hypo:reduced}, $\beta_B^*: \Gamma(\frX_n) \ra \Gamma(\frX_B)$
is injective.
\item 
$\alpha_B^*: \Gamma(\frX_T) \ra \Gamma(\frX_B)$ is an isomorphism.
\end{enumerate}
\end{lemma}
\begin{proof}
For (i), by \cite[Remark 5.3.15]{EmertonGeeHellmann} (the conjectural description of connected components of $\frX_n$ given there is true by Corollary \ref{cor:comp_Robba}), the trianguline locus $\beta_B(\frX_B)$ inside $\frX_n$ is Zariski dense, which implies that the kernel of $\beta_B^* \subset \Gamma(\frX_n)$ consists of locally nilpotent functions.
Thus under Assumption \ref{hypo:reduced}, $\beta_B^*$ is injective.
\vskip8pt

As for (ii), consider the map $\iota_B: \frX_T \ra \frX_B$ taking an ordered $n$-tuple $(\cR(\delta_1),\dots,\cR(\delta_n))$ of  $(\varphi,\Gamma)$-modules of rank $1$ to the split $(\varphi,\Gamma)$-module $\bigoplus_{i=1}^n \cR(\delta_i)$.
We have $\alpha_B \circ \iota_B = \id_{\frX_T}$, so $\iota_B^* \circ \alpha_B^* = \id_{\Gamma(\frX_T)}$.
It suffices to show that $\alpha_B^* \circ \iota_B^* = \id_{\Gamma(\frX_B)}$.
This amounts to showing that for any affinoid algebra $A$, the value of any function $f\in \Gamma(\frX_B)$ on any $\Sp(A)$-valued point $(D_A,\Fil^\bullet(D_A))$ of $\frX_B$ depends only on the associated graded $\alpha_B((D_A,\Fil^\bullet(D_A))) \in \frX_T(A)$.
Since a function on $\Sp(A)$ is determined by its restriction to an admissible cover,  we may assume that $(D_A,\Fil^\bullet(D_A))$ is free of rank $n$ as an $\cR_{K,A}$-module, equipped with a basis $e_1,\dots,e_n$ with respect to which matrices of $\varphi$ and $\gamma\in \Gamma_K$ are upper triangular. 
With
$$
P_t := \diag(t^n,t^{n-1},\dots, t)
$$
for $t\in \bbG^{\rig}_{m,A}$, we can conjugate the matrices $M_\varphi$ to $M_{\varphi,t}:= P_t M_\varphi P_t^{-1} $ and $M_\gamma$ to $M_{\gamma,t}:= P_t M_\gamma P_t^{-1}$ to get a family of $(\varphi,\Gamma)$-modules $D_{A,t}$ on $\bbG^\rig_{m,A}$.
For any upper triangular matrix $M$, the conjugated matrix
$
M_t := P_t M P_t^{-1}
$
satisfies 
$$
(M_t)_{ij} = 
\begin{cases}
M_{ij} \cdot t^{j-i} & \text{ if } i\leq j \\
0 & \text{ if } i>j
\end{cases}\,,
$$
so we can set $t=0$ as well.
We obtain a triangulated family $(D_{A,t}, \Fil^\bullet(D_{A,t}))\in \frX_B(\bbA^{1,\rig}_{\Sp(A)})$ that specializes to  $(D_A,\Fil^\bullet(D_A))$ at $t=1$, and to the associated graded $\alpha_B((D_A,\Fil^\bullet(D_A)))$ at $t=0$.
For $t\neq 0$, $(D_{A,t},\Fil^\bullet(D_{A,t})) \cong (D_A,\Fil^\bullet(D_A))$ as triangulated $(\varphi,\Gamma)$-modules over $A$, so $f|_{\bbG_{m,\Sp(A)}^\rig}$ is constant.
Thus, $f|_{\bbA^{1,\rig}_{\Sp(A)}}$ is also constant.
We conclude that 
$
\alpha_B^* \circ \iota_B^* = \id_{\Gamma(\frX_B)}
$ 
and hence $\alpha_B^*: \Gamma(\frX_T) \ra \Gamma(\frX_B)$ is bijective.
\end{proof}

\end{para}

\begin{para}{\it (Local) companion symmetries of global functions.}
We use the geometric results on $\frX_B$ and its compactification in \cite[\S5.3]{EmertonGeeHellmann} to restrict the possible shapes that functions on $\frX_n$ can take when viewed as functions on $\frX_T$ through the embeddings in Lemma \ref{lem:fcn-embed}.
\vskip8pt 

We have natural isomorphisms
$$
\frX_T \cong \frX_1^n, \quad \frX_1 \cong \cT \times B\bbG_m, \quad \cT \cong \cW \times \bbG_m^\an\,.
$$
Hence, global functions on $\frX_T$ are functions $f$ on $\cW^n \times \bbG_m^{\an,n}$ uniquely of the form
$$
f = \sum_{\underline{m} \in \bbZ^n} c_{\underline{m}}(\kappa_1,\dots,\kappa_n) \underline{\alpha}^{\underline{m}}
$$
where $\underline{m} = (m_1,\dots, m_n)\in \bbZ^n$, $\kappa_i$ are points on $\cW$ for $1\leq i\leq n$, $\alpha_i$ are coordinates on $\bbG_m^\an$ for $1\leq i\leq n$, $\underline{\alpha}^{\underline{m}} = \alpha_1^{m_1}\cdots \alpha_n^{m_n}$, and $c_{\underline{m}}\in \Gamma(\cW^n)$ are analytic functions on the space of continuous characters of $(\cO_K^\times)^n$ such that the infinite series $f$ converges on $\cW^n \times \bbG_m^{\an, n}$.
\vskip8pt

Via the embedding $(\alpha_B^{*})^{-1} \circ \beta_B^* = \iota_B^* \circ \beta_B^*: \Gamma(\frX_n) \hookrightarrow \Gamma(\frX_T)$ under Assumption \ref{hypo:reduced}, we can view $\Gamma(\frX_n)$ as a subspace of $\Gamma(\cW^n \times \bbG_m^{\an, n})$.
In the following, we study the properties that $f\in \Gamma(\cW^n \times \bbG_m^{\an, n})$ enjoys when it actually comes from $\Gamma(\frX_n)$.
\vskip8pt

By \cite[Theorem 5.3.7]{EmertonGeeHellmann}, there is a compactification $j: \frX_B \hookrightarrow \overline{\frX}_B$ of $\frX_B$, where $\overline{\frX}_B$ is a rigid analytic Artin stack admitting maps $\overline{\beta}_B,\overline{\alpha}_B$ to $\frX_n$ and $\frX_T$ respectively that restrict along $j$ to $\beta_B,\alpha_B$ on $\frX_B$, so we have the following commutative diagram
$$
\begin{tikzcd}
               & \overline{\mathfrak{X}}_B \arrow[ld, "\overline{\beta}_B"'] \arrow[rd, "\overline{\alpha}_B"] &                \\
\mathfrak{X}_n & \mathfrak{X}_B \arrow[u, "j", hook] \arrow[l, "\beta_B"'] \arrow[r, "\alpha_B"]               & \mathfrak{X}_T
\end{tikzcd}
$$
Thus, for any $f\in \Gamma(\frX_n)$, if we write $f_T:=(\alpha_B^*)^{-1}\beta_B^{*}(f)$ for its embedded image in $\Gamma(\frX_T)$, then by the injectivity of $j^*:\Gamma(\overline{\frX}_B) \hookrightarrow \Gamma(\frX_B)$ due to the scheme-theoretic density of $\frX_B$ in $\overline{\frX}_B$ and by the identity $\beta_B^*(f) = \alpha_B^*(f_T)$, we deduce that 
\begin{numequation}
\label{eqn:cptif_eq}
\overline{\beta}_B^*(f) = \overline{\alpha}^*_B(f_T).
\end{numequation}

Moreover, the induced map 
$$
\pi:\overline{\frX}_B \xrightarrow{(\overline{\beta}_B,\overline{\alpha}_B)} \frX_n \times \frX_T
$$
is representable by rigid analytic spaces and proper by \cite[Theorem 5.3.12]{EmertonGeeHellmann}.
We write $\mathfrak{X}_{n, {\rm tri }} \subset \mathfrak{X}_n \times \mathfrak{X}_T$ for the scheme-theoretic image of $\overline{\mathfrak{X}}_B$ under the proper map $\pi$.

\begin{thm}
\label{thm:companion}
Let $D\in \frX_n(E)$ be a crystabelline $(\varphi,\Gamma)$-module of regular Hodge-Tate weight $\underline{\lambda}$ and such that the Weil-Deligne representation is of the form $\bigoplus_{i=1}^n \phi_i$ for smooth characters $\phi_i:K^\times \ra E^\times$ satisfying $\phi_i/\phi_j \notin \{1,|\cdot|_K^{\pm 1}\}$ for $i\neq j$.
Then, we have
$$
\begin{aligned} & \left\{\delta \in \mathfrak{X}_T \mid(D, \delta) \in \mathfrak{X}_{n, \operatorname{tri}}\right\} \\ & =\bigcup_{\mathcal{F}_{\bullet}}\left\{\delta_{\mathcal{F}_{\bullet}} z^{w \underline{\lambda}} \mid w=(w_\sigma)_\sigma \in W \text { such that } w_{\mathcal{F}, \sigma} \preceq w_\sigma w_0 \text { for all } \sigma\right\}
\end{aligned}
$$
where the union runs over all Frobenius stable flags $\mathcal{F}$ of $\mathrm{WD}(D)$, $w_0 \in S_n$ is the longest element, $z^{w\underline{\lambda}}$ is the algebraic character
$$
z^{w\underline{\lambda}}: (K^\times)^n \ra E^\times, \quad (x_1,\dots,x_n)\mapsto \prod_{i=1}^n \left(\prod_{\sigma\in \Sigma_K} \sigma(x_i)^{\lambda_{w_\sigma^{-1}(i),\sigma}} \right) \,,
$$
and $\delta_{\cF_\bullet}$ is the smooth character
$$
\delta_{\cF_\bullet}: (K^\times)^n \ra E^\times, \quad (x_1,\dots,x_n) \mapsto \prod_{i=1}^n \phi_{\tau_{\cF}(i)}(x_i)
$$
such that\footnote{Not to be confused with elements of the Weyl group $W=S_n^{\Sigma_K}$ of $\mathrm{Res}_\Qp^K(\GL_{n,K})_E \cong (\GL_{n,E})^{\Sigma_K}$.} $\tau_{\cF}\in S_n$ is the unique reordering of the characters $\phi_1,\dots,\phi_n$ that corresponds to the choice of Frobenius stable flag $\mathcal{F}$ of $\mathrm{WD}(D)$.
\end{thm}
\begin{proof}
This is the generic crystabelline version of \cite[Theorem 5.3.19]{EmertonGeeHellmann}, as described in \cite[Remark 5.3.20(ii)]{EmertonGeeHellmann}.
\end{proof}

In general, given any smooth characters $\phi_1,\dots, \phi_n:K^\times \ra E^\times$, we may use the compact notation $\underline{\phi}=(\phi_1,\dots,\phi_n)$ for this $n$-tuple, and define a smooth character as follows
$$
\delta_{\underline{\phi}}: (K^\times)^n \ra E^\times, \quad (x_1,\dots,x_n) \mapsto \prod_{i=1}^n \phi_i(x_i).
$$

\begin{lemma}
Let $\Sigma_K$ denote the set of $\bbQ_p$-embeddings of $K$ into $E$.
\label{lem:crystabel_crit}
\begin{enumerate}[(i)]
\item 
For any tuple of integers  $\underline{\lambda} = (\lambda_{i,\sigma})_{1\leq i\leq n,\sigma:K\hookrightarrow E} \in (\bbZ^{n})^{\Sigma_K}$ with $\lambda_{1,\sigma}>\lambda_{2,\sigma}>\dots > \lambda_{n,\sigma}$, for any smooth characters $\phi_1,\dots,\phi_n:K^\times \ra E^\times$ satisfying  $\phi_i/\phi_j \notin \{1,|\cdot|_K^{\pm 1}\}$ for $i\neq j$, and for any $w=(w_\sigma)_{\sigma\in \Sigma_K}\in S_n^{\Sigma_K}$, there exists a crystabelline $(\varphi,\Gamma)$-module $D$ over $R_{K,E}$ of rank $n$ and Hodge-Tate weight $\underline{\lambda}$ and Weil-Deligne representation $\phi_1\oplus \dots\oplus \phi_n$, such that the refinement $\cF = (\phi_1,\dots,\phi_n)$ is in relative position $w_\sigma$ with respect to the $\sigma$-Hodge filtration of $D_{\rm dR}(D)$ for all $\sigma \in \Sigma_K$.
\item 
In the situation of (i), if the smooth characters $\phi_i = \unr(\varphi_i)$ are unramified with $\varphi_i\in E^\times$ satisfying $\varphi_i/\varphi_j\notin \{1,q^{\pm 1}\}$, then the crystabelline $(\varphi,\Gamma)$-module $D$ is crystalline.
\end{enumerate}
\end{lemma}
\begin{proof}
Clearly, (ii) follows from (i).
For (i), given the smooth characters $\phi_1,\dots,\phi_n$, there is a finite abelian extension $L/K$ such that $\phi_i|_{I_L} = 1$ for all $i$.
By local class field theory, each $\phi_i$ gives a rank one $L_0\otimes_\Qp E$-module $D_i$ equipped with semilinear $(\varphi,G(L/K))$-actions, cf. \cite[Proposition 4.1]{BreuilSchneider}.
Set $D_0:=\bigoplus_{i=1}^n D_i$ to be the direct sum containing $n!$ $(\varphi,G(L/K))$-stable complete flags, also known as {\it refinements} in the literature, 
$$
\cF_{\tau(\bullet)} := (0\subset \cF_{\tau(1)} \subset \cF_{\tau(2)} \subset \dots \subset \cF_{\tau(n)} = D_0)
$$
one for each $\tau\in S_n$ where $\cF_{\tau(j)} := \bigoplus_{i=1}^j D_{\tau(i)}$ is a rank $j$ subobject of $D_0$.
\vskip8pt

Recall that an $E$-filtered $(\varphi,G(L/K))$-module over $K$ is the data of such a $D_0$ together with a separated exhaustive filtration on $D_L:=L\otimes_{L_0} D_0$ by $G(L/K)$-stable $L\otimes_{\Qp}E$-modules, and there is a tensor equivalence between the category of $(\varphi, \Gamma_K)$-modules over $\mathcal{R}_{K, E}$ that become crystalline over $L$ and the category of $E$-filtered $(\varphi, G(L/K))$-modules over $K$.
By Galois descent, the natural map 
$
L\otimes_K (D_L)^{G(L/K)} \cong D_L
$
is an isomorphism. 
Thus, a filtration by $G(L/K)$-stable $L\otimes_{\Qp}E$-submodules on $D_L$ is the same thing as a filtration by $K\otimes_{\Qp}E$-submodules on $(D_L)^{G(L/K)} =: D_{\rm dR}$.
Via the canonical isomorphism $K\otimes_\Qp E \cong \prod_{\sigma\in \Sigma_K} E, x\otimes y \mapsto (\sigma(x)y)_{\sigma\in \Sigma_K}$, the module $D_{\rm dR} \cong \prod_{\sigma\in\Sigma_K} D_{\dR,\sigma}$ decomposes into a direct product of $E$-vector spaces, and any filtration $\Fil^i D_\dR = \prod_{\sigma\in \Sigma_K} \Fil^i D_{\dR,\sigma}$ also decomposes into $\sigma$-components.
\vskip8pt
Via the construction $(\pr_\sigma)\circ(L\otimes_{L_0}(-))^{G(L/K)}$, any $(\varphi,G(L/K))$-stable flag $\cF_{\tau(\bullet)}$ on $D_0$ gives rise to a complete flag $\cF_{\tau(\bullet),\sigma}$ of $E$-vector spaces on the $\sigma$-component $D_{\dR,\sigma}$ with a basis $(e_{\tau(1),\sigma},\dots, e_{\tau(n),\sigma})$.
For $\tau =\id \in S_n$ and any $\sigma\in \Sigma_K$, we define a Hodge filtration
$$
\Fil^a (D_{\dR,\sigma}) := \bigoplus_{\set{i}{a\leq -\lambda_{w_{\sigma}^{-1}(i),\sigma}}} E.e_{i,\sigma}
$$
so that the $(-\lambda_{w_{\sigma}^{-1}(i),\sigma})$-th graded piece of this filtration $\Fil^\bullet (D_{\dR,\sigma})$ 
$$
\Fil^{ -\lambda_{w_{\sigma}^{-1}(i),\sigma}} (D_{\dR,\sigma}) / \Fil^{ -\lambda_{w_{\sigma}^{-1}(i),\sigma}+1} (D_{\dR,\sigma}) = E.\overline{e_{i,\sigma}} 
$$
is the 1-dimensional $E$-vector space generated by the image of $e_{i,\sigma}$ for $1\leq i \leq n$.
Then, the resulting $E$-filtered $(\varphi,G(L/K))$-module over $K$ satisfies the requirement.
\end{proof}

\begin{cor}
\label{cor:crystabel_sym}
If $f_T\in \Gamma(\cT^n)$ comes from a function $f\in \Gamma(\frX_n)$, then for any tuple of integers  $\underline{\lambda} = (\lambda_{i,\sigma})_{1\leq i\leq n,\sigma:K\hookrightarrow E} \in (\bbZ^{n})^{\Sigma_K}$ with $\lambda_{1,\sigma}>\lambda_{2,\sigma}>\dots > \lambda_{n,\sigma}$ and for any smooth characters $\phi_1,\dots,\phi_n:K^\times \ra E^\times$ satisfying  $\phi_i/\phi_j \notin \{1,|\cdot|_K^{\pm 1}\}$ for $i\neq j$, we have
\begin{enumerate}[(i)]
\item 
The value
$$
f_T(\delta_{\underline{\phi}}z^{w\underline{\lambda}})
=
f_T\left(\prod_{i=1}^n  \left( \phi_i(x_i)\prod_{\sigma\in \Sigma_K} \sigma(x_i)^{\lambda_{w_\sigma^{-1}(i),\sigma}} \right) \right) \in E
$$
is independent of the choice of $w=(w_\sigma)_\sigma \in S_n^{\Sigma_K}$.
\item 
The value 
$$
f_T(\delta_{\tau(\underline{\phi})}z^{\underline{\lambda}})
=
f_T\left(\prod_{i=1}^n  \left( \phi_{\tau(i)}(x_i)\prod_{\sigma\in \Sigma_K} \sigma(x_i)^{\lambda_{i,\sigma}} \right) \right) \in E
$$
is independent of the choice of $\tau \in S_n$.
\end{enumerate}

\end{cor}
\begin{proof}
(i). Given any dominant regular Hodge-Tate weights $\underline{\lambda}$ and generic smooth characters $\underline{\phi}$, by Lemma \ref{lem:crystabel_crit} one can construct a crystabelline $(\varphi,\Gamma)$-module $D$ of Hodge-Tate weight $\underline{\lambda}$ with Weil-Deligne representation $\bigoplus_{i=1}^n \phi_i$ such that the refinement $\cF = \underline{\phi}$ satisfies $w_{\cF,\sigma}=\id$ for all $\sigma\in \Sigma_K$.
Thus by Theorem \ref{thm:companion}, the set 
$$
\set{(D,\delta_{\underline{\phi}}z^{w\underline{\lambda}})}{w=(w_\sigma)_\sigma \in W} \subset \frX_n \times \frX_T
$$
is contained in $\frX_{n,{\rm tri}}(E)$.
Since $\frX_{n,{\rm tri}}$ is the scheme-theoretic image of $\pi: \overline{\frX}_B\ra \frX_n\times \frX_T$, which is a proper map, we have $\pi(\overline{\frX}_B)=\frX_{n,\mathrm{tri}}$. 
Thus, there exist a finite extension $E'/E$ and points $y_w\in \overline{\frX}_B(E')$ such that $(D,\delta_{\underline{\phi}}z^{w\underline{\lambda}}) = \pi(y_w) 
$
for each $w\in W$.
By \eqref{eqn:cptif_eq}, 
$$
f(D) = f_T(\delta_{\underline{\phi}}z^{w\underline{\lambda}})\,,
$$
which shows the independence of $w$ for the expression $f_T(\delta_{\underline{\phi}}z^{w\underline{\lambda}})$.

(ii). It is easy to build a non-critical crystabelline $(\varphi,\Gamma)$-module $D$ of Hodge-Tate weights $\underline{\lambda}$ and Weil-Deligne representation $\bigoplus_{i=1}^n \phi_i$ all of whose $n!$ refinements $\cF_\tau$ (in bijection with $\tau\in S_n$) satisfy $w_{\cF_\tau,\sigma}=w_0$ for all $\sigma\in \Sigma_K$.
Then, in the statement of Theorem \ref{thm:companion}, for each $\tau\in S_n$ we choose $w=(w_{\sigma})_\sigma = \id\in W$ because then $w_{\cF_\tau,\sigma} \preceq w_\sigma w_0$ becomes $w_0\preceq \id \circ w_0$, which is true.
By the same argument as in (i), we complete the proof.
\end{proof}

\begin{lemma}
\label{lem:diagonal}
Suppose $f\in \Gamma(\frX_n)$ and write $f_T$ for its image in $\Gamma(\cT^n)$ of the form
$$
f_T = \sum_{\underline{m} \in \bbZ^n} c_{\underline{m}}(\kappa_1,\dots,\kappa_n) \underline{\alpha}^{\underline{m}}.
$$
Then, $c_{\underline{m}} = 0 \in \Gamma(\cW^n)$ unless $\underline{m}=(m,\dots,m)\in \bbZ^n$ is diagonal for some integer $m\in \bbZ$.
\end{lemma}
\begin{proof}
We first prove the vanishing of off-diagonal coefficients of $f$ by testing on  crystalline $(\varphi,\Gamma)$-modules.
For any dominant regular weight $\underline{\lambda}$, consider generic unramified characters $\underline{\phi}=(\phi_1,\dots,\phi_n)$ with $\phi_i:=\unr(\varphi_i)$ for
generic $\underline{\varphi}=(\varphi_1,\dots,\varphi_n) \in (E^\times)^n$, in which case we use this tuple of numbers $\underline{\varphi} \in (E^\times)^n$ to represent $\underline{\phi}$.
By Corollary \ref{cor:crystabel_sym}, we have 
$$
f_T(\delta_{\underline{\varphi}}z^{\underline{\lambda}}) = f_T(\delta_{\underline{\varphi}}z^{w\underline{\lambda}})
$$
for any $w=(w_\sigma)_\sigma \in W$, which explicitly means that
$$
\sum_{\underline{m}} c_{\underline{m}} \left(\left.\prod_{\sigma}\sigma(z_1)^{\lambda_{w_{\sigma}^{-1}(1),\sigma}}\right|_{\cO_K^\times},\dots, \left.\prod_{\sigma}\sigma(z_n)^{\lambda_{w_{\sigma}^{-1}(n),\sigma}}\right|_{\cO_K^\times}\right) \cdot \prod_{i,\sigma} \sigma(\varpi_K)^{\lambda_{w_{\sigma}^{-1}(i),\sigma}m_i} \cdot \underline{\varphi}^{\underline{m}}\,,
$$
is independent of the choice of $w$.
Then, for each fixed $\underline{\lambda}$, we get rigid functions on
$$
\set{(\varphi_1,\dots,\varphi_n)\in (E^\times)^n}{\forall i\neq j, \varphi_i/\varphi_j\notin \{1,q\}}
$$
in the variables $\underline{\varphi}$, indexed by $w\in W$ but independent of the choice of $w\in W$.
Hence, for each fixed $\underline{m}\in \bbZ^n$, by equating the coefficients in front of $\underline{\varphi}^{\underline{m}}$ we obtain that
$$
c_{\underline{m}} \left(\left.\prod_{\sigma}\sigma(z_1)^{\lambda_{w_{\sigma}^{-1}(1),\sigma}}\right|_{\cO_K^\times},\dots, \left.\prod_{\sigma}\sigma(z_n)^{\lambda_{w_{\sigma}^{-1}(n),\sigma}}\right|_{\cO_K^\times}\right) \cdot \prod_{i,\sigma} \sigma(\varpi_K)^{\lambda_{w_{\sigma}^{-1}(i),\sigma}m_i} 
$$
is independent of the choice of $w=(w_\sigma)_\sigma \in W$, for all dominant regular weights $\underline{\lambda}$.
\vskip8pt

We claim that $c_{\underline{m}} = 0 \in \Gamma(\cW^n)$ whenever $\underline{m}=(m_1,\dots,m_n)\in \bbZ^n$ is not of the diagonal form $(m,\dots,m)$ for some $m\in \bbZ$.
Indeed, for such non-diagonal $\underline{m}$, choose $j\in \{1,\dots, n-1\}$ such that $m_j\neq m_{j+1}$.
Consider the element $w=(w_\sigma)_\sigma\in W$ with $w_\sigma=(j \mbox{ }j+1)$ being the transposition for all $\sigma$, and consider also the identity element $w'=1\in W$.
Comparing the resulting expressions and using independence of $w$, we obtain that
\begin{align*}
& \mbox{ }c_{\underline{m}}(\underline{\lambda}) \cdot \prod_{i=1}^n \prod_{\sigma\in \Sigma_K} \sigma(\varpi_K)^{\lambda_{i,\sigma}m_i} 
\\&=
\left( \prod_{i\notin \{j,j+1\}} \prod_{\sigma\in \Sigma_K} \sigma(\varpi_K)^{\lambda_{i,\sigma}m_i} \right) \cdot \left(\prod_{\sigma\in \Sigma_K} \sigma(\varpi_K)^{\lambda_{j,\sigma}m_j+\lambda_{j+1,\sigma}m_{j+1}}\right)
\cdot c_{\underline{m}}(\underline{\lambda})
\\&=
\left( \prod_{i\notin \{j,j+1\}} \prod_{\sigma\in \Sigma_K} \sigma(\varpi_K)^{\lambda_{i,\sigma}m_i} \right) \cdot \left(\prod_{\sigma\in \Sigma_K} \sigma(\varpi_K)^{\lambda_{j+1,\sigma}m_j+\lambda_{j,\sigma}m_{j+1}}\right)
\cdot c_{\underline{m}}((j\mbox{ }j+1)\underline{\lambda}) 
\end{align*}
which simplifies to 
\begin{numequation}
\label{eqn:j_j+1}
c_{\underline{m}}(\underline{\lambda}) = c_{\underline{m}}((j\mbox{ }j+1)\underline{\lambda})\cdot \prod_{\sigma}\sigma(\varpi_K)^{(m_{j+1}-m_j)(\lambda_{j,\sigma}-\lambda_{j+1,\sigma})} \,.
\end{numequation}
Now we use a $p$-adic approximation argument to show that \eqref{eqn:j_j+1} implies the vanishing of $c_{\underline{m}}$ on the weight space $\cW^n$ as follows.
By \cite[Proposition II.5.7(i)]{NeukirchI}, there exists a finite free $\bbZ_p$-submodule $\Gamma$  of rank $[K:\Qp]$  contained in the 1-units $U_K^{(1)}$ such that
$$
\cO_K^\times = \mu(K)\times \Gamma  \cong \mu(K) \times \bbZ_p^{[K:\Qp]}
$$
as topological abelian groups.
For any $\underline{\nu}=(\nu_{i,\sigma})_{1\leq i\leq n,\sigma\in \Sigma_K} \in (\bbZ_p^{n})^{\Sigma_K}$, the assignment 
$$
x^{\underline{\nu}}: (x_1,\dots,x_n) \mapsto \prod_{i=1}^n \prod_{\sigma\in \Sigma_K} \sigma(x_i)^{\nu_{i,\sigma}}
$$
defines a continuous character on $\Gamma^n \subset (\cO_K^\times)^n$.
The weight space $\cW^n$ decomposes
$$
\cW^n = \left( \coprod_{\eta\in \widehat{\mu(K)}} \bbB_{\eta}\right)^n = \coprod_{(\eta_1,\dots,\eta_n )\in (\widehat{\mu(K)})^n} (\bbB_{\eta_1}\times \dots \times  \bbB_{\eta_n})
$$
into a disjoint union of polydiscs indexed by $n$-tuples of characters of $\mu(K)$, where each 
$$
\bbB_{\eta}\cong \Spf(\cO_E[[\Gamma]])^\rig \cong \Spf(\cO_E[[\bbZ_p^{[K:\Qp]}]])^\rig
$$ 
is a polydisc of dimension $[K:\Qp]$ over $E$.
For $\underline{\eta}=(\eta_1,\dots,\eta_n)\in (\widehat{\mu(K)})^n$, we abbreviate $\bbB_{\underline{\eta}}:=\bbB_{\eta_1}\times \dots \times  \bbB_{\eta_n}$.
Each connected component $\bbB_{\underline{\eta}}$ of $\cW^n$ contains a Zariski-dense subset
$
\set{\underline{\eta}\cdot x^{\underline{\nu}}}{\underline{\nu}\in (\bbZ_p^{n})^{\Sigma_K}}
$
where, if via the identification $\cO_K^\times = \mu(K)\times \Gamma$ we write $z=\zeta\cdot x$ for each $z\in \cO_K^\times$, $\zeta\in \mu(K)$, and $x\in \Gamma$, then the character $\underline{\eta}\cdot x^{\underline{\nu}}$ is given by
$$
(\underline{\eta}\cdot x^{\underline{\nu}})(z_1,\dots,z_n) = \prod_{i=1}^n \left( \eta_i(\zeta_i) \prod_{\sigma\in \Sigma_K} \sigma(x_i)^{\nu_{i,\sigma}}\right)\,.
$$
Let $N$ be the order of the group $\mu(K)\cong \bbZ/N\bbZ$, and choose tuples $\underline{a}=(a_{1,\sigma},\dots, a_{n,\sigma})_\sigma \in (\{0,1,\dots, N-1\}^n)^{\Sigma_K}$ such that $\eta_i(\zeta)=\prod_{\sigma\in \Sigma_K}\sigma(\zeta)^{a_{i,\sigma}}$ for all $\zeta\in \mu(K)$.
Consider
$$
\sL_{\underline{\eta}} := \set{\underline{\eta}\cdot x^{\underline{\nu}}}{\underline{\nu}\in \underline{a}+N\bbZ^{n\Sigma_K}} = \set{z^{\underline{\nu}}}{\underline{\nu}\in \underline{a}+N\bbZ^{n\Sigma_K}}\,,
$$
which is Zariski-dense in the component $\bbB_{\underline{\eta}}$, where for $z_i=\zeta_i\cdot x_i\in \cO_K^\times$, we simply have
$$
(\underline{\eta}\cdot x^{\underline{\nu}})(z_1,\dots,z_n) = \prod_{i=1}^n \left( \prod_{\sigma\in \Sigma_K}\sigma(\zeta_i)^{a_{i,\sigma}} \prod_{\sigma\in \Sigma_K} \sigma(x_i)^{\nu_{i,\sigma}}\right) 
= z^{\underline{\nu}}\,.
$$
We now use dominant regular weights $z^{\underline{\lambda}}$ to approximate any point $z^{\underline{\nu}}\in \sL_{\underline{\eta}}$ as follows.
Let $\underline{\nu}\in \underline{a}+N\bbZ^{n\Sigma_K}$ be given.
For each $1\leq i\leq n$ and $\sigma\in \Sigma_K$, we have
$
\nu_{i,\sigma} = a_{i,\sigma} + Nv_{i,\sigma}
$
for some $v_{i,\sigma}\in \bbZ$.
For each $k\in \bbZ_{\geq 1}$, set
$
\lambda^{(k)}_{i,\sigma} := \nu_{i,\sigma}+N(n-i)p^k\,.
$
Then, $\underline{\lambda}^{(k)}=(\lambda^{(k)}_{1,\sigma},\dots,\lambda^{(k)}_{n,\sigma})_{\sigma\in \Sigma_K} \in (\bbZ^{n})^{\Sigma_K}$ becomes dominant regular for sufficiently large $k\gg 0$ because
$$
\lambda^{(k)}_{i,\sigma}-\lambda^{(k)}_{i+1,\sigma} 
= \nu_{i,\sigma}-\nu_{i+1,\sigma}+Np^k
\xrightarrow{k\ra\infty} \infty \,.
$$
Moreover, $\lim_{k\ra \infty} \underline{\lambda}^{(k)} = \underline{\nu}$ in $ \bbZ_p^{n\Sigma_K}$, and thus $\lim_{k\ra \infty}c_{\underline{m}}(\underline{\lambda}^{(k)}) = c_{\underline{m}}(\underline{\nu})$ by continuity.
We have two cases to discuss:
\begin{itemize}
\item[\underline{$m_j<m_{j+1}$}:]
applying \eqref{eqn:j_j+1} to the sequence $\underline{\lambda}^{(k)}$ yields, for sufficiently large $k$, that
\begin{align*}
|c_{\underline{m}}(\underline{\lambda}^{(k)})|_p 
&= \left|c_{\underline{m}}((j\mbox{ }j+1)\underline{\lambda}^{(k)})\cdot \prod_{\sigma}\sigma(\varpi_K)^{(m_{j+1}-m_j)(\lambda^{(k)}_{j,\sigma}-\lambda^{(k)}_{j+1,\sigma})}\right|_p
\\&\leq 
(|c_{\underline{m}}((j\mbox{ }j+1)\underline{\nu})|_p+1)\cdot \left|\prod_{\sigma}\sigma(\varpi_K)^{(m_{j+1}-m_j)(\lambda^{(k)}_{j,\sigma}-\lambda^{(k)}_{j+1,\sigma})}\right|_p \xrightarrow{k\ra \infty} 0
\end{align*}
because $\lambda^{(k)}_{j,\sigma}-\lambda^{(k)}_{j+1,\sigma}\xrightarrow{k\ra\infty} \infty$ and $m_{j+1}>m_j$.
Therefore, $c_{\underline{m}}(\underline{\nu}) = \lim_k c_{\underline{m}}(\underline{\lambda}^{(k)})=0$.
Since $c_{\underline{m}}$ vanishes on the dense subspace $\sL_{\underline{\eta}}$ of the component $\bbB_{\underline{\eta}}$, we conclude that $c_{\underline{m}}$ vanishes on the entire component.

\item[\underline{$m_j>m_{j+1}$}:] 
applying \eqref{eqn:j_j+1} to the sequence $\underline{\lambda}^{(k)}$ yields, for sufficiently large $k$, that
\begin{align*}
|c_{\underline{m}}((j\mbox{ }j+1)\underline{\lambda}^{(k)})|_p 
&= \left|c_{\underline{m}}(\underline{\lambda}^{(k)})\cdot \prod_{\sigma}\sigma(\varpi_K)^{(m_{j}-m_{j+1})(\lambda^{(k)}_{j,\sigma}-\lambda^{(k)}_{j+1,\sigma})}\right|_p
\\&\leq 
(|c_{\underline{m}}(\underline{\nu})|_p+1)\cdot \left|\prod_{\sigma}\sigma(\varpi_K)^{(m_{j}-m_{j+1})(\lambda^{(k)}_{j,\sigma}-\lambda^{(k)}_{j+1,\sigma})}\right|_p \xrightarrow{k\ra \infty} 0\,.
\end{align*}
We conclude that $c_{\underline{m}}((j \mbox{ }j+1)\underline{\nu})=0$, which shows that $c_{\underline{m}}$ vanishes on the component $\bbB_{(j\mbox{ }j+1)\underline{\eta}}$ indexed by the permuted character $(j\mbox{ }j+1)\underline{\eta}$, in which $(j\mbox{ }j+1)\sL_{\underline{\eta}}$ is a Zariski-dense subset.\footnote{If we choose the tuple for $(j \mbox{ }j+1)\underline{\eta}$ to be $(j\mbox{ }j+1)\underline{a}$, then $(j\mbox{ }j+1)\sL_{\underline{\eta}}$ is precisely $\sL_{(j\mbox{ }j+1)\underline{\eta}}$.}
\end{itemize}
As the choice of the component $\underline{\eta}$ was arbitrary, we conclude that $c_{\underline{m}} = 0 \in \Gamma(\cW^n)$ unless $\underline{m}=(m,\dots,m)\in \bbZ^n$ is diagonal for some integer $m\in \bbZ$.
\end{proof}

We can now write
$$
f_T = \sum_{m\in \bbZ} c_m(\kappa_1,\dots,\kappa_n)\left(\prod_{i=1}^n \alpha_i\right)^m
$$
with $c_m\in\Gamma(\cW^n)$, and there are two kinds of symmetries that follow from Corollary \ref{cor:crystabel_sym}.
\begin{enumerate}[(i)]
\item {\it Symmetry in Sen weights.}
We claim that for any smooth characters $\psi_i:\cO_K^\times \ra E^\times$ for $1\leq i\leq n$, we have
\begin{numequation}
\label{eqn:sym_Sen}
c_m(\delta_{\underline{\psi}}z^{\underline{\lambda}}) = c_m(\delta_{\underline{\psi}}z^{w\underline{\lambda}})\,,
\end{numequation}
for all dominant regular weights $\underline{\lambda}$ and all $w=(w_\sigma)_\sigma \in W$.
\begin{proof}
Indeed, for any $n$-tuple 
$$
(\varphi_1,\dots,\varphi_n)\in (E^\times)^n
$$ 
satisfying $\varphi_i/\varphi_j\notin \{1,q^{\pm1}\}$ for $i\neq j$, we can extend $\psi_i:\cO_K^\times \ra E^\times$ to a generic tuple of smooth characters $\phi_i:K^\times \ra E^\times$ by sending the uniformizer $\varpi_K$ to $\varphi_i$.
By $w$-independence in Corollary \ref{cor:crystabel_sym}(i) and vanishing of off-diagonal terms, we have
$$
\sum_{m\in\bbZ} c_m(\delta_{\underline{\phi}}z^{\underline{\lambda}}) \prod_\sigma \sigma(\varpi_K)^{m\sum_{i}\lambda_{i,\sigma}} (\prod_i \varphi_i)^m 
= 
\sum_{m\in\bbZ} c_m(\delta_{\underline{\phi}}z^{w\underline{\lambda}}) \prod_\sigma \sigma(\varpi_K)^{m\sum_{i}\lambda_{w_\sigma^{-1}(i),\sigma}} (\prod_i \varphi_i)^m\,,
$$
for all $w\in W$, equating the $m$-th coefficient of which proves \eqref{eqn:sym_Sen}.
\end{proof}
\item {\it Symmetry in smooth parts.}
We claim that for any smooth characters $\psi_i:\cO_K^\times \ra E^\times$ for $1\leq i\leq n$, we have
\begin{numequation}
\label{eqn:sym_smooth}
c_{\underline{m}}(\psi_1\kappa_1,\dots,\psi_n\kappa_n) = c_{\underline{m}}(\psi_{\tau(1)}\kappa_1,\dots,\psi_{\tau(n)}\kappa_n)
\end{numequation}
for all $\tau \in S_n$. 
\begin{proof}
Since the translation map
$$
(\kappa_1,\dots,\kappa_n) \mapsto (\psi_1\kappa_1,\dots,\psi_n\kappa_n)
$$ 
is an analytic automorphism of $\cW^n$ and $c_{\underline{m}}$ is rigid analytic on $\cW^n$, it suffices to show that \eqref{eqn:sym_smooth} holds as $(\kappa_1,\dots,\kappa_n)$ runs over a Zariski-dense subset of $\cW^n$, e.g., the set of dominant regular weights $\underline{\lambda}$, which is shown to be Zariski-dense in (each component of) $\cW^n$ during the proof of Lemma \ref{lem:diagonal}.
For any generic $n$-tuple 
$$
(\varphi_1,\dots,\varphi_n)\in (E^\times)^n
$$ 
satisfying $\varphi_i/\varphi_j\notin \{1,q^{\pm1}\}$ for $i\neq j$, we can extend $\psi_i:\cO_K^\times \ra E^\times$ to a generic tuple of smooth characters $\phi_i:K^\times \ra E^\times$ by sending the uniformizer $\varpi_K$ to $\varphi_i$.
Then, by Corollary \ref{cor:crystabel_sym}(ii), we have
$$
\sum_{m\in\bbZ} c_m(\delta_{\underline{\phi}}z^{\underline{\lambda}}) \prod_\sigma \sigma(\varpi_K)^{m\sum_{i}\lambda_{i,\sigma}} (\prod_i \varphi_i)^m 
= 
\sum_{m\in\bbZ} c_m(\delta_{\tau(\underline{\phi})}z^{\underline{\lambda}}) \prod_\sigma \sigma(\varpi_K)^{m\sum_{i}\lambda_{i,\sigma}} (\prod_i \varphi_{\tau(i)})^m\,.
$$
Equating the $m$-th coefficients of both sides yields
$$
c_m(\delta_{\underline{\phi}}z^{\underline{\lambda}}) = c_m(\delta_{\tau(\underline{\phi})}z^{\underline{\lambda}})
$$
for all dominant regular weights $\underline{\lambda}$ and all $\tau\in S_n$, proving \eqref{eqn:sym_smooth}.
\end{proof}
\end{enumerate}

\end{para}

\begin{para}{\it Descent along the derivative map.}
Recall the notion of Sen weights from \cite[6.1.6]{KPX_cohomology}.
For any affinoid $E$-algebra $A$, the {\it weight} of any continuous character $\delta: K^\times \ra A^\times$ is the unique element $\wt(\delta) = (\wt_\sigma(\delta))_{\sigma\in \Sigma_K}\in K\otimes_{\Qp}A\cong \prod_{\sigma\in \Sigma_K} A$ such that
$$
\lim_{\substack{x\ra 0\\ x\in \cO_K}} \frac{|\delta(1+x)-(1+\sum_{\sigma\in \Sigma_K} \wt_\sigma(\delta)\sigma(x))|_A}{|x|_p} = 0
$$
for any Banach algebra norm $|\cdot|_A$ defining the topology on $A$.
This definition also applies to continuous characters $\kappa:\cO_K^\times \ra A^\times$.
We also write the limit statement informally as 
$$
\delta(1+x) = 1 + \sum_{\sigma\in \Sigma_K} \wt_\sigma(\delta)\sigma(x) + o(x) \,.
$$
One can understand the weight of $\kappa$ as the derivative: choose $m=m(K)\in \bbN_{>0}$ such that for $U_K^{(m)}:=1+\varpi^m_K\cO_K$, the $p$-adic logarithm and exponential
$$
\log : U_K^{(m)} \longrightarrow \frm_K^m, \quad \exp: \frm_K^m \longrightarrow U_K^{(m)}
$$
converge and are inverses to each other, from which we draw a commutative diagram
$$
\begin{tikzcd}
U_K^{(m)} \arrow[rr, "\kappa"]                           &  & 1+A^{\circ\circ} \arrow[d, "\log"] \\
\mathfrak{m}_K^m \arrow[u, "\exp"] \arrow[rr, "d\kappa"] &  & A                                 
\end{tikzcd}
$$
where $A^{\circ\circ}$ denotes the ideal of topologically nilpotent elements of $A$ and $d\kappa:=\log\circ \kappa \circ \exp$ is $\bbZ_p$-linear as $\kappa$ is a continuous character.
Since $A$ is an $E$-algebra and $\frm_K^m$ is $p$-torsionfree, we can equivalently view $d\kappa$ as an element of 
$$
\Hom_{\Qp}(\frm_K^m[1/p],A) = \Hom_{\Qp}(K,A) = \Hom_E(K\otimes_\Qp E, A) = \Hom_{E}(\prod_{\sigma\in \Sigma_K} E,A) = \prod_{\sigma\in \Sigma_K} A \,,
$$
where all isomorphisms are canonical.
Concretely, to $\ua = (a_\sigma)_\sigma\in \prod_{\sigma\in \Sigma_K} A$, we associate the $\Qp$-linear functional $\ell_{\ua}:K \ra A, x\mapsto \sum_\sigma a_\sigma \sigma(x)$.

\begin{prop}
The derivative $d\kappa: K\ra A$ is given by $x\mapsto \sum_\sigma \wt_\sigma(\kappa)\sigma(x)$.    
\end{prop}
\begin{proof}
Consider 
$
d\kappa=\log\circ \kappa\circ \exp: \frm_K^m \ra A.
$
For small $x\in \frm_K^m$, 
$$
\exp(x)-1 = x+O(x^2) = x +o(x)
$$ is also small as $x\ra 0$.
The definition of Sen weight of $\kappa$ implies that
\begin{align*}
\kappa(1+(\exp(x)-1)) 
&= 
1 +\sum_{\sigma \in \Sigma_K} \wt_\sigma(\kappa)\sigma(\exp(x)-1)+o(\exp(x)-1)
\\&=
1+\sum_{\sigma\in \Sigma_K} \wt_\sigma(\kappa)\sigma(x +o(x))+o(x)
\\&=
1+\sum_{\sigma\in \Sigma_K} \wt_\sigma(\kappa)\sigma(x)+o(x)\,.
\end{align*}
Then we have
\begin{align*}
d\kappa(x) 
&= 
\log(\kappa(\exp(x))) = \log(\kappa(1+(\exp(x)-1))) 
\\&=
\log\left(1+\sum_{\sigma\in \Sigma_K} \wt_\sigma(\kappa)\sigma(x)+o(x)\right)
\\&= \sum_{\sigma\in \Sigma_K} \wt_\sigma(\kappa)\sigma(x)+o(x)
\end{align*}
because $\log(1+z)=z-\frac{z^2}{2}+\frac{z^3}{3} -\dots =z+O(z^2)$.
By linearity of $d\kappa$, we conclude that
$$
d\kappa(x) =\sum_{\sigma\in \Sigma_K} \wt_\sigma(\kappa)\sigma(x)
$$
as desired.
\end{proof}

\begin{prop}
\label{prop:wt_surj}
The weight map 
$$
\wt: \cW \longrightarrow\bbA^{\Sigma_K,\rig}_E, \quad \kappa \mapsto \wt(\kappa) = (\wt_\sigma(\kappa))_\sigma
$$
is surjective in the étale topology on $\Rig_E$. 
\end{prop}
\begin{proof}
For any affinoid $E$-algebra $A$, and any morphism $\Sp(A)\ra \bbA_E^{\Sigma_K,\rig}$ of rigid analytic spaces corresponding to a $[K:\Qp]$-tuple $\ua=(a_\sigma)_\sigma \in A^{\Sigma_K}$, choose  large $m'\geq m(K)$ such that the exponential map $\exp$ on $A$ converges and is invertible on the image of 
$$
\ell_{\ua}\circ \log : U_K^{(m')} \longrightarrow A, \quad u\mapsto \sum_\sigma a_\sigma. \sigma(\log(u))\,.
$$
Then, define $\kappa':= \exp \circ \ell_{\ua} \circ \log: U_K^{(m')} \longrightarrow A^\times$.
By construction, $d\kappa' = \ell_{\ua}$.
If $\kappa' = \kappa|_{U_K^{(m')}}$ is the restriction of a continuous character $\kappa$ of $\cO_K^\times$, then we know that $\wt(\kappa)=\ua$.
Consider
$$
0 \ra U_K^{(m')} \ra \cO_K^\times \ra \cO_K^\times/U_K^{(m')} \cong \prod_{i=1}^s \Z/r_i\Z \ra 0 \,,
$$
where $\cO_K^\times/U_K^{(m')} \cong \prod_{i=1}^s \Z/r_i\Z$ is a finite abelian group in its invariant factor decomposition.
We claim that $\kappa':U_K^{(m')}\ra A^\times$ extends to a continuous character $\kappa: \cO_K^\times \ra B^\times$ for a finite étale $A$-algebra $B$.
Indeed, pick lifts $x_i\in \cO_K^\times$ of the generators $1\in \Z/r_i\Z$, and set
$$
B := A[T_1,\dots, T_s]/(T_i^{r_i}-\kappa'(x_i^{r_i}))_{1\leq i\leq s}
$$  
which is a finite free étale $A$-algebra, as $\kappa'(x_i^{r_i})\in A^\times$.
This makes sense because $x_i^{r_i}\in U_K^{(m')}$ where $\kappa'$ is defined.
We can extend\footnote{This map $\kappa$ is a group homomorphism, for the same reason why for every abelian extension $$0\ra A \ra E \ra G \ra 0$$ of a finite abelian group $G$ of exponent $n$ by an abelian group $A$, any character of $A$ extends to a character of $E$, after possibly extending the coefficients to extract certain $n$-th roots.} $\kappa'$ to $\kappa:\cO_K^\times \ra B^\times$ by setting $\kappa(x_i):= T_i\in B^\times$.
Then, $\wt(\kappa) = \wt(\kappa')=\ua$, so the weight map $\wt$ is étale locally surjective.
\end{proof}

Consider the map $q:=(\wt_1,\dots,\wt_n,\det) : \cW^n \ra (\bbA_E^n)^{\Sigma_K}\times_{(\bbA_E^1)^{\Sigma_K}} \cW$ defined by
$$
q: (\kappa_1,\dots, \kappa_n) \mapsto (\wt(\kappa_1),\dots, \wt(\kappa_n),\prod_{i=1}^n\kappa_i)
$$
of taking ordered Sen weights and the determinant, where the fiber product is defined by the trace map
$$
(\bbA_E^n)^{\Sigma_K} \ra (\bbA_E^1)^{\Sigma_K}, \quad (a_{i,\sigma})_{1\leq i\leq n,\sigma\in \Sigma_K}\mapsto (\sum_{i=1}^n a_{i,\sigma})_{\sigma\in \Sigma_K}
$$
and the weight map studied in Proposition \ref{prop:wt_surj},
$$
\wt: \cW \ra \bbA^{\Sigma_K}_E\,.
$$

Observe that we have an isomorphism
$$
\pr:
(\bbA_E^n)^{\Sigma_K}\times_{(\bbA_E^1)^{\Sigma_K}} \cW \xrightarrow{\simeq} (\bbA_E^{n-1})^{\Sigma_K} \times \cW
$$
projecting a pair $\displaystyle ((a_{i,\sigma})_{1\leq i\leq n}^{\sigma\in \Sigma_K}, \kappa)$ with $(\wt_\sigma(\kappa))_\sigma = (\sum_{i} a_{i,\sigma})_\sigma$ to the pair $((a_{i,\sigma})_{1\leq i\leq {n-1}}^{\sigma\in \Sigma_K}, \kappa)$, because $a_{n,\sigma} = \wt_\sigma(\kappa)-\sum_{i=1}^{n-1} a_{i,\sigma}$.
In particular, $(\bbA_E^n)^{\Sigma_K}\times_{(\bbA_E^1)^{\Sigma_K}} \cW$ is reduced.

\begin{lemma}
\label{lem:descent_wt}
Suppose $f\in \Gamma(\frX_n)$ and write $f_T$ for its image in $\Gamma(\cT^n)$ of the form
$$
f_T = \sum_{m\in \bbZ} c_{m}(\kappa_1,\dots,\kappa_n) \left(\prod_{i=1}^n\alpha_i\right)^m
$$
for $c_m\in \Gamma(\cW^n)$,
thanks to Lemma \ref{lem:diagonal}.
Then, $c_m = q^*(\widetilde{c}_m)$ descends to a global function $\wtc_m\in \Gamma((\bbA_E^n)^{\Sigma_K}\times_{(\bbA_E^1)^{\Sigma_K}} \cW)$ along the map
$$
q:=(\wt_1,\dots,\wt_n,\det) : \cW^n \ra (\bbA_E^n)^{\Sigma_K}\times_{(\bbA_E^1)^{\Sigma_K}} \cW\,.
$$
\end{lemma}

\begin{proof}
For any affinoid subdomain $\Sp(A)\subset (\bbA_E^n)^{\Sigma_K}\times_{(\bbA_E^1)^{\Sigma_K}} \cW$, viewed as a family 
$$
((a_{i,\sigma})_{i,\sigma} \in (A^{n})^{\Sigma_K}, \kappa_0:\cO_K^\times \ra A^\times) \quad \text{ such that }\quad \big(\sum_{i=1}^n a_{i,\sigma}\big)_{\sigma\in \Sigma_K} = \wt_\sigma(\kappa_0)_{\sigma\in \Sigma_K}
$$
over $\Sp(A)$, by Proposition \ref{prop:wt_surj} there exists a finite étale cover $\Sp(B)\twoheadrightarrow \Sp(A)$ together with continuous characters $\kappa_1,\dots,\kappa_{n-1}: \cO_K^\times \ra B^\times$ such that $\wt_\sigma(\kappa_i)=a_{i,\sigma}$ for each $\sigma\in \Sigma_K$.
Then, $(\sum_{i=1}^n a_{i,\sigma})_{\sigma\in \Sigma_K} = \wt_\sigma(\kappa_0)_{\sigma\in \Sigma_K}$ means that for $\kappa_n := \kappa_0(\kappa_1\cdots \kappa_{n-1})^{-1} \in \cW(\Sp(B))$, we get a local section $s: \Sp(B) \ra \cW^n$ to $q$ defined by $s = (\kappa_1,\dots,\kappa_n)$:
$$
\begin{tikzcd}
                                                                                 & \mathcal{W}^n \arrow[r, "q"]                                    & (\bbA_E^n)^{\Sigma_K}\times_{(\bbA_E^1)^{\Sigma_K}} \cW \\
\Sp(B\otimes_A B) \arrow[r, "\pr_1", shift left] \arrow[r, "\pr_2"', shift right] & \Sp(B) \arrow[r, "{\text{finite \'{e}tale}}", two heads] \arrow[u, "{s=(\kappa_1,\dots,\kappa_n)}", dashed] & \Sp(A) \arrow[u, hook, "{((a_{i,\sigma}),\kappa_0)}"]                                 
\end{tikzcd}
$$
In order to descend $c_m\in \Gamma(\cW^n)$ along $q$, we claim that $s^*(c_m)\in B$ comes from $\Sp(A)$. 
By faithfully flat descent, this is equivalent to the claim that $s_1^*(c_m) = s_2^*(c_m)$ for $s_i:=s\circ \pr_i$, where $i=1,2$.
As we have seen that $\Sp(A)$ is reduced and $\Sp(B)\ra \Sp(A)$ is finite étale, $\Sp(B)$ is reduced by \cite[\href{https://stacks.math.columbia.edu/tag/033B}{Tag 033B}]{stacks-project}.
Then, $\Sp(B\otimes_AB)$ is again reduced because it is finite étale over the reduced rigid space $\Sp(B)$ by base change.
Hence, it suffices to check $s_1^*(c_m)|_{\kappa(x)}=s_2^*(c_m)|_{\kappa(x)}$ at every closed point $x\in \Sp(B\otimes_A B)$.
\vskip8pt
At any point $x\in \Sp(B\otimes_A B)$, the residue field $\kappa(x)$ is a finite extension $L$ of $E$.
To show that $s_1^*(c_m)|_{\kappa(x)}=s_2^*(c_m)|_{\kappa(x)}$, it suffices to show that for any two $n$-tuples of characters
$
(\kappa_1',\dots, \kappa_n'), (\kappa_1'',\dots, \kappa_n'') \in \cW^n(L)
$ 
with the same weights $\wt_\sigma(\kappa_{i}')=\wt_\sigma(\kappa_{i}'')$ and product $\prod_{i=1}^n\kappa_i'=\prod_{i=1}^n\kappa_i''$, we have $c_m(\kappa_1',\dots, \kappa_n')=c_m(\kappa_1'',\dots, \kappa_n'')$.
This follows from \eqref{eqn:sym_smooth}.
Indeed, by the definition of weight, $\kappa_i' = \psi_i \kappa_i''$ for a smooth character $\psi_i:\cO_K^\times \ra L^\times$.
The equality of products $\prod_{i=1}^n\kappa_i'=\prod_{i=1}^n\kappa_i''$ implies that $\prod_{i=1}^n \psi_i = 1$.
Iterating \eqref{eqn:sym_smooth} yields
\begin{align*}
c_m(\kappa_1',\dots,\kappa_n') 
&= 
c_m(\psi_1\kappa_1'',\dots, \psi_n \kappa_n'')
\\&=
c_m((\prod_{i=1}^2\psi_i)\kappa_1'',\kappa_2'',\psi_3\kappa_3'',\psi_4\kappa_4'',\dots,\psi_n \kappa_n'')
\\&=
c_m((\prod_{i=1}^3\psi_i)\kappa_1'',\kappa_2'',\kappa_3'', \psi_4\kappa_4'' ,\dots,\psi_n \kappa_n'')
\\&=
c_m((\prod_{i=1}^n\psi_i)\kappa_1'',\kappa_2'',\dots,\kappa_n'') 
\\&= c_m(\kappa_1'',\dots,\kappa_n'')\,.
\end{align*}
For each affinoid $\Sp(A) \subset (\bbA_E^n)^{\Sigma_K}\times_{(\bbA_E^1)^{\Sigma_K}} \cW$, this gives a function $\widetilde{c}_{m,A}$; these functions glue\footnote{Before proving this, we have not even shown that $\wtc_{m,A}$ is independent of the choices of étale cover $\Sp(B)$ and the section $s:\Sp(B)\ra \cW^n$. 
We now prove them together: let $\Sp(B)\twoheadrightarrow \Sp(A)$ and $\Sp(C)\twoheadrightarrow \Sp(A)$ be any finite étale covers, and let $s_B:\Sp(B)\ra \cW^n,s_C:\Sp(C)\ra \cW^n$ be any local sections to $q$. 
We need to identify $s_B^*(c_m)$ and $s_C^*(c_m)$, both of which we know come from $\Sp(A)$, by restricting them on the overlap $\Sp(D):=\Sp(B)\times_{\Sp(A)} \Sp(C)$.
Same argument as before shows that $\Sp(D)\twoheadrightarrow\Sp(A)$ is a finite étale cover with $\Sp(D)$ reduced. 
We again use \eqref{eqn:sym_smooth} over each residue field to equate the restrictions of $s_B^*(c_m)$ and $s_C^*(c_m)$ to $\Sp(D)$ as elements in $D$. 
Since $A\ra D$ is injective by faithful flatness, we know that $s_B^*(c_m)$ and $s_C^*(c_m)$ descend to the same element of $A$, which we denote by $\wtc_{m,A}$.
} across the overlaps to $\wtc_m\in \Gamma((\bbA_E^n)^{\Sigma_K}\times_{(\bbA_E^1)^{\Sigma_K}} \cW)$ such that $q^*\wtc_m = c_m$.
\end{proof}

Consider the map
$$
g: (\bbA_E^n)^{\Sigma_K}\times_{(\bbA_E^1)^{\Sigma_K}} \cW \longrightarrow \WT_{n,E}\times_{\WT_{1,E}} \cW, \quad 
g((a_{i,\sigma})_{i,\sigma},\kappa) = \left(\bigg(\prod_{i=1}^n (T-a_{i,\sigma})\bigg)_{\sigma\in \Sigma_K},\kappa\right)\,,
$$
where $(\prod_{i=1}^n (T-a_{i,\sigma}))_{\sigma\in \Sigma_K}$ is viewed as an element in $(K\otimes_\Qp A)[T]\cong \prod_{\sigma\in \Sigma_K} A[T]$ whose coefficients are taken to be the value of $g$ in $\WT_{n,E}={\rm Res}_{\Qp}^K(\bbA^n_K)_E=\prod_{\sigma\in \Sigma_K}\bbA^n_E$.
\begin{lemma}
\label{lem:descent_GIT}
For any function $f\in \Gamma(\frX_n)$ with 
$$
f_T = \sum_{m\in \bbZ} c_{m}(\kappa_1,\dots,\kappa_n) \left(\prod_{i=1}^n\alpha_i\right)^m
$$
with $c_m=q^*(\wtc_m)$ for a unique $\wtc_m\in \Gamma((\bbA_E^n)^{\Sigma_K}\times_{(\bbA_E^1)^{\Sigma_K}} \cW)$ as in Lemma \ref{lem:descent_wt}, and for each $m\in \bbZ$, there exists a unique $\frc_m\in \Gamma(\WT_{n,E}\times_{\WT_{1,E}} \cW)$ such that $\wtc_m = g^*(\frc_m)$.
\end{lemma}
\begin{proof}
Observe that $\WT_{n,E}\times_{\WT_{1,E}} \cW$ is the GIT quotient of $(\bbA_E^n)^{\Sigma_K}\times_{(\bbA_E^1)^{\Sigma_K}} \cW$ for the permutation action of the finite group $(S_n)^{\Sigma_K}$, and $g$ is the quotient map.
To descend $\wtc_m$ through $g$, it suffices to show that $\wtc_m$ is $(S_n)^{\Sigma_K}$-invariant.
A Zariski-dense subset of $(\bbA_E^n)^{\Sigma_K}\times_{(\bbA_E^1)^{\Sigma_K}} \cW$ $(\cong (\bbA_E^{n-1})^{\Sigma_K} \times \cW)$ is given by
$$
(\underline{\lambda}, \det(\delta_{\underline{\phi}}z^{\underline{\lambda}}))
$$
coming from
those regular integral weights $\underline{\lambda}$ and generic locally algebraic characters $\delta_{\underline{\phi}}z^{\underline{\lambda}}$ valued in $L^\times$ for all finite extensions $L$ of $E$.
Over this locus, the $(S_n)^{\Sigma_K}$-invariance of $\wtc_m$ follows by \eqref{eqn:sym_Sen}. 
Hence, by Zariski-density $\wtc_m$ is $(S_n)^{\Sigma_K}$-invariant, and it comes from a (unique) function $\frc_m\in \Gamma(\WT_{n,E}\times_{\WT_{1,E}} \cW)$ along $g$.
\end{proof}

\end{para}

\begin{para}
We finish the proof of Theorem \ref{thm:EG_fcns} as follows.
\begin{proof}[Proof of Theorem \ref{thm:EG_fcns}]
We need to show that $(\omega_n,\det):\frX_n \ra \WT_{n,E}\times_{\WT_{1,E}} \frX_1$ induces a bijection on global functions by pullback.
We have canonical isomorphisms
$$
\Gamma(\WT_{n,E}\times_{\WT_{1,E}} \frX_1) \cong 
\Gamma(\WT_{n,E}\times_{\WT_{1,E}} (\cW\times \bbG_m)) \cong  \Gamma((\WT_{n,E}\times_{\WT_{1,E}} \cW)\times \bbG_m)\,.
$$

For the injectivity of $(\omega_n,\det)^*$, let us consider the following commutative diagram
$$
\begin{tikzcd}
\mathfrak{X}_n \arrow[rrr, "{(\omega_n,\det)}"]        &                       & & {\WT_{n,E}\times_{\WT_{1,E}} \mathfrak{X}_1}                                               \\
\mathfrak{X}_B \arrow[u, "\beta_B"]                         & \mathfrak{X}_T  \arrow[l,"\iota_B"]   \arrow[rr, "{(\wt_1,\dots,\wt_n,\det)}"] & & (\mathbb{A}^n_E)^{\Sigma_K} \times_{(\mathbb{A}^1_E)^{\Sigma_K}} \mathfrak{X}_1 \arrow[u]\,.
\end{tikzcd}
$$
By Lemmas \ref{lem:fcn-embed} and \ref{lem:descent_GIT}, it suffices to show that the map
$$
\cT^n = \cW^n \times \bbG_m^n\xrightarrow{((\wt_1,\dots,\wt_n,\det),\det)} \left((\bbA_E^n)^{\Sigma_K}\times_{(\bbA_E^1)^{\Sigma_K}} \cW\right)\times \bbG_m = (\mathbb{A}^n_E)^{\Sigma_K} \times_{(\mathbb{A}^1_E)^{\Sigma_K}} \cT\,,
$$
induces an injective map on global functions.
This is true since the map $\det:\bbG_m^n \ra \bbG_m$, $(\alpha_1,\dots,\alpha_n) \mapsto \prod_{i=1}^n \alpha_i,$ has a section $\alpha \mapsto (\alpha,1,\dots, 1)$, while
$(\wt_1,\dots,\wt_n,\det): \cW^n\ra (\bbA_E^n)^{\Sigma_K}\times_{(\bbA_E^1)^{\Sigma_K}} \cW$ admits finite \'{e}tale local sections by Lemma \ref{lem:descent_wt}.
\vskip8pt

The surjectivity of $(\omega_n,\det)^*$ follows from the conclusion of Lemma \ref{lem:descent_GIT} and the observation that for a trianguline $(\varphi,\Gamma)$-module $D$ with triangulation parameters $(\delta_1,\dots,\delta_n)$, we have $\det(D)=\cR_{K}(\delta_{\det(D)})$ for $\delta_{\det(D)}:=\prod_{i=1}^n \delta_i$ and $\prod_{i=1}^n \delta_i(\varpi_K) = \delta_{\det(D)}(\varpi_K)$.
Indeed, given $
f_T = \sum_{m\in \bbZ} c_{m}(\kappa_1,\dots,\kappa_n) \left(\prod_{i=1}^n\alpha_i\right)^m
$, consider its restriction 
$
\gamma := \sum_{m\in \bbZ} c_{m}(\underline{\kappa})t^m \in \Gamma(\cW^n\times \bbG_m^{\an})
$.
The finite \'{e}tale descent argument of Lemma \ref{lem:descent_wt} proves that $\gamma$ descends to $\widetilde{\gamma}\in \Gamma(((\bbA_E^n)^{\Sigma_K}\times_{(\bbA_E^1)^{\Sigma_K}} \cW )\times \bbG^{\an}_m)$.
We deduce that $\widetilde{\gamma} =\sum_{m\in \bbZ} \widetilde{c_m}t^m$ converges.
Since each $\widetilde{c_m}$ is $(S_n)^{\Sigma_K}$-invariant, the uniqueness of Laurent expansion implies that $\widetilde{\gamma}$ descends to $\mathfrak{r}\in \Gamma((\WT_{n,E}\times_{\WT_{1,E}} \cW)\times \bbG_m^{\an}) = \Gamma(\WT_{n,E}\times_{\WT_{1,E}} \frX_1)$, which has to be $\mathfrak{r}=\sum_{m\in \bbZ} \mathfrak{c}_m t^m$.
Then $(g\circ (\wt_1,\dots,\wt_n,\det))^*(\mathfrak{r}) = f_T$, hence  $(\omega_n,\det)^*$ is surjective by Lemma \ref{lem:fcn-embed}.
\end{proof}

\end{para}

\section{Comparing connected components}
\label{section6_comparison}

\addtocounter{subsection}{1}

\begin{para} {\it Kohlhaase's result on the center of the locally analytic distribution algebra.} 
Consider the sequence $\GL_{n,K} \supset B = TN \supset T \supset Z$ of algebraic groups over $K$, where $B$ is the Borel subgroup of upper triangular matrices, $N$ is the unipotent radical of $B$, $T$ is the subgroup of diagonal matrices, and $Z$ is the center of $\GL_{n,K}$.
Denote their Lie algebras by $\frg = \frg\frl_n(K) \supset \frb = \frt \oplus \frn \supset \frt \supset \frz$.
Write $\frn^-$ for the Lie algebra of the unipotent radical of the opposite Borel subgroup $B^-$.
By Weil's restriction of scalars, we get a sequence 
$
\bG  \supset \bB = \bT\bN \supset \bT \supset \bZ
$
of algebraic groups over $\bbQ_p$, where $\bG={\rm Res}^K_{\bbQ_p}(\GL_{n,K})$ and similarly for $\bB,\bN,\bT,$ and $\bZ$.
Their $\bbQ_p$-Lie algebras are $\frg = \frg\frl_n(K) \supset \frb = \frt \oplus \frn \supset \frt \supset \frz$, whose base changes to the coefficient field $E$ are denoted by $\frg_E:=\frg \otimes_{\bbQ_p} E \cong \prod_{\sigma\in\Sigma_K} \frg\frl_{n}(E)$, $\frb_E$, $\frn_E$, $\frt_E$, and $\frz_E$.
Let $Z(\frg_E)$ denote the center of the universal enveloping algebra of $\frg_E$.
Applying Kohlhaase's theorem \cite[2.2.1]{Kohlhaase_Invariant_dist} to the connected $\bbQ_p$-reductive group $\bG$, we have an isomorphism of $E$-algebras:
\[\kappa: Z(D^\la(\GL_n(K))) \isom Z(\frg_E) \widehat{\otimes}_{U(\frz_E)} D^\la(Z(\GL_n(K))) \;.\]
\end{para}

\begin{para}{\it The ring homomorphism $\gamma: \cC_G \ra \Gamma(\frX_{n,K})$.}
{Let us recall the $p$-adic cyclotomic character $\vep: G_K \ra \Z_p^\x$ and the character obtained from it when pre-composed with the reciprocity map ${\rm rec}_K: K^\x \ra G_K^\ab$, i.e., $(\vep \circ {\rm rec}_K)(a) = {\rm N}_{K/\Qp}(a) \cdot |{\rm N}_{K/\Qp}(a)|$. For the sake of brevity, we also denote this map $K^\x \ra \Z_p^\x$ by $\vep$.} 
We define a ring homomorphism $\gamma: \cC_G \ra \Gamma(\frX_{n,K})$ by a normalized combination of the Harish-Chandra isomorphism and the Amice transform. 
For the Harish-Chandra isomorphism, we normalize the half sum of positive roots of the split reductive $E$-group $\bG \times_{\bbQ_p} E \cong \prod_{\sigma\in \Sigma_K} \GL_{n,E}$ to be 
$$
\rho' := (0, -1, \dots, -(n-1))_{\sigma\in \Sigma_K}
$$
following \cite[Remark 6.2.8(i)]{EmertonGeeHellmann}.
Write $W\cong S_n^{\Sigma_K}$ for the Weyl group of $\bG \times_{\bbQ_p} E$.
Viewing $U(\frt_E)$ as regular functions on the $E$-linear dual $\frt_E^*$, the translation $-\rho': \frt_E^*\to \frt_E^*, \bh \mapsto \bh - \rho'$, induces an isomorphism $(-\rho')^*:U(\frt_E)\to U(\frt_E)$.
We have a commutative diagram
$$
\begin{tikzcd}
U(\mathfrak{t}_E)\oplus (\mathfrak{n}^-_EU(\mathfrak{g}_E)+U(\mathfrak{g}_E)\mathfrak{n}_E) \arrow[r, "{\rm pr}"] & U(\mathfrak{t}_E) \arrow[r, "(-\rho')^*"]                                                   & U(\mathfrak{t}_E)                 \\
{Z(\mathfrak{g}_E)} \arrow[u, hook]                                                                       & U(\mathfrak{z}_E) \arrow[u, hook] \arrow[l, hook'] \arrow[r, "(-\rho'|_{\mathfrak{z}_E})^*"] & U(\mathfrak{z}_E) ,\arrow[u, hook]
\end{tikzcd}
$$
where the restriction of $(-\rho')^*\circ \pr$ to $Z(\frg_E)$ induces an isomorphism ${\rm HC}:Z(\frg_E) \to U(\frt_E)^W$.
The commutativity implies the equality ${\rm HC}|_{U(\frz_E)}= (-\rho'|_{\frz_E})^*:U(\frz_E) \to U(\frz_E)$, where $-\rho'|_{\frz_E} = (\frac{n(n-1)}{2})_{\sigma\in \Sigma_K} \in \frz^*_E$.
Since the cyclotomic character $\vep:K^\times \to \bbZ_p^\times$ has $\sigma$-Sen weight $1$ for each $\sigma\in \Sigma_K$, we can extend $(-\rho'|_{\frz_E})^*$ to a commutative diagram of ring maps

$$
\begin{tikzcd}
U(\mathfrak{z}_E) \arrow[d, "{\rm HC}|_{U(\frz_E)}=(-\rho'|_{\mathfrak{z}_E})^*"'] \arrow[r, hook] & D^{\rm la}(Z(\GL_n(K))) \arrow[d, "\left(\vep^{\frac{n(n-1)}{2}}\right)^*"] \\
U(\mathfrak{z}_E) \arrow[r, hook]                                           & D^{\rm la}(Z(\GL_n(K))),                                        
\end{tikzcd}
$$
where $\left(\vep^{\frac{n(n-1)}{2}}\right)^*: D^{\la}(K^\times) \to D^{\la}(K^\times)$ is the linear map induced by duality from
$$
\mbox{}\cdot\vep^{\frac{n(n-1)}{2}}: C^\la(K^\times, E) \to C^\la(K^\times, E), \quad f \mapsto f\cdot\vep^{\frac{n(n-1)}{2}}\,,
$$ 
and $\left(\vep^{\frac{n(n-1)}{2}}\right)^*$ respects convolution products since $\vep^{\frac{n(n-1)}{2}}$ is a character of $K^\times$.
Let 
$$
\sA: D^\la(Z(\GL_n(K))) = D^\la(K^\times) = D^{\la}(\cO_K^\times)[T^{\pm 1}] \xrightarrow{\simeq} \Gamma(\cW)[T^{\pm 1}] \hookrightarrow \Gamma(\cW\times \bbG_m^{\an}) = \Gamma(\mathfrak{X}_{1,K})
$$
be the injective ring map induced by Amice's isomorphism $D^{\la}(\cO_K^\times) \xrightarrow{\simeq} \Gamma(\cW)$.
Then, we define the ring homomorphism $\gamma: \cC_G \ra \Gamma(\frX_{n,K})$ by the following composition
$$
\begin{tikzcd}
\cC_G \ar[dd,"\gamma"] \ar[r,"z_{(-)}"] & Z(D^\la(\GL_n(K))) \ar[r,
"\kappa"]  & Z(\frg_E) \widehat{\otimes}_{U(\frz_E)} D^\la(Z(\GL_n(K))) \ar[d,"{\rm HC} \otimes (\vep^{n(n-1)/2})^*"]\\
& & U(\frt_E)^W \widehat{\otimes}_{U(\frz_E)} D^\la(Z(\GL_n(K)))\ar[d,"\textrm{incl} \otimes \sA"] \\
\Gamma(\frX_{n,K}) &  & \Gamma(\WT_{n,E}) \widehat{\otimes}_{\Gamma(\WT_{1,E})} \Gamma(\mathfrak{X}_{1,K}), \ar[ll,"\omega_n^* \otimes \det^*"] \\ 
\end{tikzcd}
$$
where $\textrm{incl}: U(\frt_E)^W\hookrightarrow \Gamma(\WT_{n,E})$ is the natural inclusion map.
The reason for the chosen normalization by $\vep^{n(n-1)/2}$ is the expected compatibility between the determinant $\det(V)$ of a $G_K$-representation $V$ and the central character $\omega_{\Pi(V)}$ of the locally analytic $\GL_{n}(K)$-representation $\Pi(V)$ conjecturally associated to $V$, i.e., $\omega_{\Pi(V)} = \vep^{\frac{n(n-1)}{2}}\det(V)$, cf. \cite[Proposition 3.26]{Ding_crystabelline} and \cite[Remark 3.11]{Ding_wallcross}\footnote{In \cite{Ding_wallcross}, the half sum of positive roots is normalized to be $\theta=(n-1,\dots,1,0)_{\sigma\in\Sigma_K}$, in which case the compatibility is of the form $\omega_{\Pi(V)} = \vep^{-\frac{n(n-1)}{2}}\det(V)$, which agrees with \cite[Proposition 0.1(i)]{ColmezPoids} when $n=2$ and $K=\bbQ_p$.}.
\end{para}

\vskip8pt

\begin{para} 
\label{para:center_endofunctor}
To $e \in \cC(\cA)$ one associates a functor $\ove: \cA \ra \cA$ by setting $\ove(M) := \im(e_M)$, and for a morphism $f: M \ra N$ in $\cA$, setting $\ove(f) \in \Hom_\cA(\ove(M),\ove(N))$ to be the morphism induced by $f$ from $\ove(M) = \im(e_M)$ to $\ove(N) = \im(e_N)$. One has $\overline{e_1 \circ e_2} = \ove_1 \circ \ove_2$.
\end{para}

\begin{prop} Let $\cA$ be an abelian category and assume that the center $\cC = \cC(\cA)$ of $\cA$ is a ring which can be written as $\cC = \prod_{i=1}^m \cC_i$ with finitely many rings $\cC_i$ whose Zariski spectrum is connected. Let $\pr_i: \cC  \ra \cC_i$ be the projection onto the $i$-th factor, and $e_i = \pr_i(1_\cC)$. 

\vskip8pt

\begin{enumerate}[(i)]
\item $\cA_i := \ove_i \cdot \cA$ is an indecomposable subcategory (i.e., if it is written as the direct product of two abelian subcategories, then one of the factors contains only zero objects).

\vskip5pt

\item $\cA = \bigoplus_{i=1}^m \cA_i$, and for every $i \in \{1, \ldots, m\}$ the inclusion $\iota_i: \cA_i \ra \cA$ and projection functors $\pi_i: \cA \ra \cA_i$ are exact.

\vskip5pt

\item For any distinct $i, j \in \{1, \ldots,m\}$, $M$ in $\cA_i$, and $N$ in $\cA_j$ one has $\Ext^q_\cA(M,N) = 0$ for all $q \ge 0$.
\end{enumerate}
\end{prop}

\begin{proof}
We claim that for an abelian category $\cA$ that contains at least one nonzero object, $\cA$ is isomorphic to the direct product of two (necessarily full) abelian subcategories $\cA = \cA_1 \times \cA_2$ such that both $\cA_1$ and $\cA_2$ contain at least one nonzero object if and only if the center $\cC(\cA)$ has a disconnected Zariski spectrum.
Indeed, if $\cA = \cA_1 \times \cA_2$ is the product of such abelian categories, then we let $p_1:\cA\to \cA_1$ and $p_2:\cA\to \cA_2$ be the projections, and let $i_1:\cA_1 \to \cA$ and $i_2:\cA_2\to \cA$ be the inclusions.
Then, $e_1:=i_1\circ p_1$ and $e_2:=i_2\circ p_2$ are idempotents in $\cC(\cA)$ satisfying $e_1+e_2=1$ and $e_1e_2=0$.
Thus, $\cC(\cA)=e_1\cC(\cA)\times e_2\cC(\cA)$ with $e_j\cC(\cA)\cong \cC(\cA_j)$ for $j=1,2$. Our assumption on $\cA_j$ implies that $e_j\cC(\cA)$ are not the zero ring.
Hence, $\Spec(\cC(\cA))$ is disconnected.
Conversely, suppose $\Spec(\cC(\cA))$ is disconnected. 
Then, $\cC(\cA) = \cC_1 \times \cC_2$ is a direct product of two nonzero rings.
Let $e_1=(1,0)$ and $e_2=(0,1)$ denote the orthogonal idempotents in $\cC(\cA)$ such that $e_i\cC(\cA) = \cC_i$ for $i=1,2$.
Then, consider $\cA_i := \im(\ove_i) \subset \cA$ for $i=1,2$, which are subcategories of $\cA$ by the discussion in \ref{para:center_endofunctor}.
Since $e_i^2 =e_i$, $e_1+e_2 =1$, and $e_1e_2=0$, it follows that each object $X\in \cA$ decomposes $X=\ove_1(X) \oplus \ove_2(X)$ as a direct sum of objects in $\cA_1$ and $\cA_2$, cf. \cite[ch. VIII, sec. 2]{MacLane_working}.
Moreover, any morphism $f:X\to X'$ in $\cA$ decomposes as $f=\ove_1(f)\oplus \ove_2(f)$.
So, $\cC(\cA_i)=\cC_i$ for $i=1,2$, and $\cA_1$ and $\cA_2$ are abelian full subcategories, each containing a nonzero object, such that $\cA=\cA_1\times \cA_2$.

\vskip8pt

Then, (i) and (ii) follow from the claim above.
For (iii), the center $\cC(\cA)$ naturally injects into the center $\cC(D(\cA)) := \{\eta\in \End (1_{D(\cA)}) \mid \eta(X[1]) = (\eta X)[1]\}$ of the derived category $D(\cA)$, and we have $\Ext^q_{\cA}(M,N) = \Hom_{D(\cA)}(M,N[q])$.
Thus,
\begin{align*}
\Hom_{D(\cA)}(M,N[q]) 
&= 
\prod_{l=1}^m \Hom_{D(\cA)}(\ove_l \cdot M,\ove_l \cdot (N[q])) 
\\&= 
\prod_{l=1}^m \Hom_{D(\cA)}(\ove_l \cdot M,(\ove_l \cdot N)[q]) = 0,
\end{align*}
since we have $1=\sum_{l=1}^m e_l$, $M\in \cA_i, N\in \cA_j$, and $i\neq j$. Alternatively, one could work with Yoneda Ext groups\footnote{See \cite[ch. III, Prop. 3.2.2]{Verdier_Cat-derivees} for the equivalence of these with those defined via the derived category.} and the projection functors $\pi_i$.  
\end{proof}

\begin{para} The hypothetical EGH functor 

\[\frA^\rig_{\GL_n(K)}: D^b_{\rm f.p.}(\Rep^\la(\GL_n(K))) \lra D^b_{\rm coh}(\frX_{n,K})\] 

\vskip8pt

from \cite[6.2.4]{EmertonGeeHellmann} should not only be $Z(\frg\frl_{n,K})$-linear, but linear with respect to the action of the Bernstein center $\cC_G$ on the domain and the action of $\Gamma(\frX_{n,K})$ on the target via the morphism $\gamma:\cC_G \ra \Gamma(\frX_{n,K})$ from above. 

\vskip8pt 

Explicitly, for $V$ in $D^b_{\rm f.p.}(\Rep^\la(\GL_n(K)))$ and $e \in \cC_G$ we expect to have 

\begin{numequation}\label{EHG-functor-linearity}
\frA^\rig_{\GL_n(K)}(\ove \cdot V) = \gamma(e) \cdot \frA^\rig_{\GL_n(K)}(V) \;.    
\end{numequation}

\vskip8pt

Let $e$ be an idempotent in the Bernstein center and $\cC_e = \ove \cdot \Rep^\la(\GL_n(K))$ be the corresponding subcategory of $\Rep^\la(\GL_n(K))$. Then this functor maps the subcategory $D^b_{\rm f.p.}(\cC_e)$ to $D^b_{\rm coh}(\frX_{n,K}^{\gamma(e)})$, where $\frX_{n,K}^{\gamma(e)}$ is the connected component of $\frX_{n,K}$ corresponding to the idempotent $\gamma(e) \in \Gamma(\frX_{n,K})$.

\end{para}

\section{Appendix: Topology of Artin stacks on rigid spaces}
\mbox{ }
\addtocounter{subsection}{1}

\begin{para}{\it Quasi-analytic spaces and rigid analytic Artin stacks.} Let $\Rig_{E}$ be the category of rigid analytic spaces over $E$.
As in \cite[5.1.7]{EmertonGeeHellmann} we equip $\Rig_E$ with the Tate-fpqc topology \cite[2.1]{ConradTemkinAnalytification}.
Recall from \cite{EmertonGeeHellmann} the following definitions. 
\begin{enumerate}[(i)]
    \item A \textbf{quasi-analytic space} is a sheaf $X \in \Shv(\Rig_E)$ such that its diagonal $X\to X\times_E X$ is representable by rigid spaces and there exists an étale surjection $Y\twoheadrightarrow  X$ from a representable sheaf $Y\in \Rig_E$, \cite[5.1.9]{EmertonGeeHellmann}.
    \item A \textbf{rigid analytic Artin stack} is a stack $\frX$ on $\Rig_E$ such that its diagonal $\frX \xrightarrow[]{} \frX\times_E \frX$ is representable by quasi-analytic spaces and there exists a smooth surjection $Y\twoheadrightarrow  \frX$ from some $Y \in \Rig_E$,  \cite[5.1.10]{EmertonGeeHellmann}.
\end{enumerate}
\end{para}

\begin{lemma}
\label{lem:local_property_P}
For morphisms of rigid analytic spaces over $E$, the properties of being an open immersion, \'{e}tale, or smooth are Tate-fpqc local on the target.
Moreover, smoothness (resp. \'{e}taleness) is smooth-local (resp. \'{e}tale-local) on the source-and-target.
\end{lemma}

\begin{proof}
The first sentence follows from the faithfully flat descent of property $P$ for rigid spaces \cite[Theorem 4.2.7]{ConradA} together with the description of the Tate-fpqc pretopology in \cite[\S 2.1]{ConradTemkinAnalytification}.
\footnote{
By \cite[Theorem 4.2.7]{ConradA}, if $f:X\to Y$ is a morphism of rigid spaces and $Y'\to Y$ is faithfully flat and admits local fpqc quasi-sections, then $f$ has $P$ if and only if the base change $ f':X\times_Y Y'\to Y' $ has $P$, where $P$ can be open immersion, smooth, or \'{e}tale. 
Now let $\{Y_i\to Y\}_{i\in I}$ be a Tate-fpqc covering family in the sense of \cite[\S 2.1]{ConradTemkinAnalytification}, and suppose that each base change $ f_i:X\times_Y Y_i\to Y_i $ of $f:X\to Y$ has $P$.  
By the description given in {\it loc. cit.}, there is an admissible Tate cover $\{V_\alpha\to Y\}$ such that for each $\alpha$, there exist an fpqc map $ q_\alpha:W_\alpha\to V_\alpha $ and a $V_\alpha$-morphism $ W_\alpha\to \coprod_i (Y_i\times_Y V_\alpha)$.
Pulling $f|_{V_\alpha}$ back to $W_\alpha$, we see that Tate-locally on $W_\alpha$, the resulting morphism is the base change $(X\times_Y Y_i \times_Y V_\alpha \to Y_i \times_Y V_\alpha)$ of some $f_i$.  
Since $P$ is stable under base change and Tate-local on the target, the morphism $ X\times_Y W_\alpha\to W_\alpha $ has $P$. 
By fpqc descent for $P$ along $W_\alpha\to V_\alpha$, it follows that $f|_{V_\alpha}: X\times_Y V_\alpha\to V_\alpha $ has $P$. 
Since the $V_\alpha$'s form an admissible cover of $Y$ and $P$ is Tate-local on the target, we see that $f$ has the property $P$.
}
The second sentence is \cite[Proposition 3.14]{Oh-Shimizu_moduli}.
\end{proof}

\begin{dfn}
\label{defn:source-and-target_property}
Let $f: X\to Y$ be a map of quasi-analytic spaces $X$ and $Y$ on $\Rig_E$.
    \begin{enumerate}
    \item We say that $f$ is \textbf{smooth} if there exist \'{e}tale surjections $V\twoheadrightarrow Y$ and $U\twoheadrightarrow X$ with $V,U\in \Rig_E$ such that the induced map $U\times_Y V\to V$ is a smooth morphism of rigid spaces, \cite[Definition 3.27]{Oh-Shimizu_moduli}.
    \item We say that $f$ is \textbf{\'{e}tale} if there exist \'{e}tale surjections $V\twoheadrightarrow Y$ and $U\twoheadrightarrow X$ with $V,U\in \Rig_E$ such that the induced map $U\times_Y V\to V$ is an \'{e}tale morphism of rigid spaces, \cite[Definition 3.27]{Oh-Shimizu_moduli}.
    \item We say that $f$ is an \textbf{open immersion} if $f$ is a monomorphism of sheaves and $f$ is \'{e}tale, cf. \cite[\href{https://stacks.math.columbia.edu/tag/05W5}{Tag 05W5}]{stacks-project}.
    \end{enumerate}
\end{dfn}

\begin{rem}
By Lemma \ref{lem:local_property_P} and the argument of \cite[Lemma 3.28]{Oh-Shimizu_moduli}, all parts of Definition \ref{defn:source-and-target_property} are independent of the choices of \'{e}tale surjections $V\twoheadrightarrow Y$ and $U\twoheadrightarrow X$.
\end{rem}

\begin{dfn}
    Let $f: \frX\to \frY$ be a morphism of rigid analytic Artin stacks $\frX$ and $\frY$ that is representable in quasi-analytic spaces. 
    Let $P \in$ \{smooth, \'{e}tale, open immersion\}. 
    We say that \textbf{$f$ has property $P$}, if for every morphism $U\to \frY$ with $U$ a quasi-analytic space on $\Rig_E$, the map $\frX\times_{\frY} U \to U$ has property $P$ as a map of quasi-analytic spaces.
\end{dfn}

\begin{dfn}
\begin{enumerate}[(i)]
\item
An \textbf{open subspace} of a quasi-analytic space $X$ on $\Rig_E$ is a subsheaf $W\subset X$ such that the inclusion map $i:W\to X$ is an open immersion.
\item
An \textbf{open substack} of an Artin stack $\frX$ on $\Rig_E$ is a strictly full subcategory $\frY\subset \frX$ such that $\frY$ is also a stack and $\frY\to \frX$ is representable by quasi-analytic spaces and is an open immersion.
\end{enumerate}
\end{dfn}

In what follows we attach to each stack $\frX$ on $\Rig_E$ an {\it underlying topological space} $|\frX|$ such that if $\frX$ is a rigid analytic Artin stack, then open substacks of $\frX$ correspond bijectively to open subsets of $|\frX|$.
\vskip8pt

The category $\Rig_{E}$  admits a fully faithful embedding $r: \Rig_E \to \Ad_E$ into the category of adic spaces over $\Spa(E,\cO_E)$ by \cite[(1.1.11)]{Hu96}. 
If $A$ is an affinoid $E$-algebra, then one has $r(\Sp A) = \Spa(A,A^\circ)$, where $A^\circ$ is the subring of power-bounded elements in $A$.
As in \cite[\mbox{}5.1.7]{EmertonGeeHellmann} we equip $\Rig_E$ with the Tate-fpqc topology \cite[\mbox{}2.1]{ConradTemkinAnalytification}.
We denote by 
$$
|-|: \Rig_{E}\rightarrow \Top
$$
the functor sending a rigid analytic space $X$ to the underlying topological space $|r(X)|$ of the adic space $r(X)$, and we consider its left Kan extension $|-|:\Psh(\Rig_E)\rightarrow \Top$ along the Yoneda embedding $\Rig_E \hookrightarrow \Psh(\Rig_E)$.
The functor $r$ sends Tate-fpqc covering families $(X_i)_{i\in I}$ in \cite[\mbox{}2.1]{ConradTemkinAnalytification} of a rigid space $X$ to open covers $\bigcup_{i\in I} |r(X_i)|$ of $|r(X)|$ by \cite[(1.1.11)(c)]{Hu96} and \cite[Corollary 2.5]{Lazda_note}.
Therefore, the left Kan extension factors through the sheafification functor $L: \Psh(\Rig_E) \rightarrow \Shv(\Rig_E)$ to yield a functor 
$$
|-|:\Shv(\Rig_E) \rightarrow \Top, \quad |Y| = \varinjlim_{\substack{X\in {\Rig_E}\\ \underline{X}\ra Y}} |X|\,,
$$
where the colimit is taken over all representable sheaves $X\in \Rig_E$ with a map $X\to Y$.
We call $|Y|$ the \textit{underlying topological space} of the sheaf $Y$.
For a stack $\frX$ on $\Rig_E$, its \textit{underlying topological space} is defined to be that of the sheaf $F_\frX$ associated to $\frX$, namely $|\frX|:=|F_\frX|$.
Compare with \cite[7.3.2]{RC_SolidNotes} and \cite[4.3.1(2)]{ABLBRCS_deRhamFF}.
\vskip8pt 

\begin{lemma}
\label{lem:presentation_top}
Let $\frX$ be a stack on $\Rig_E$ for which there exists a surjection $\pi: \underline{U} \twoheadrightarrow \frX$ representable in quasi-analytic spaces from a rigid space $U$.
Then, if we write $R:=U\times_{\frX}U$, the morphism $U\to \frX$ induces an equivalence $[U/R]\xrightarrow{\sim}\frX$, and the underlying topological space $|\frX|$ is given by $|U|/|R|:=\coeq_{\Top}(|R|\rightrightarrows|U|)$.
\end{lemma}
\begin{proof}
Since $\pi$ is an epimorphism, $\pi:U\to \frX$ induces an equivalence $[U/R]\simeq \frX$.
Thus, the associated sheaves $F_{[U/R]}\cong F_{\frX}$ are isomorphic.
By adjunctions and a Yoneda argument in $\Shv(\Rig_E)$, we see that the sheaf $F_{[U/R]}$ associated to the quotient stack $[U/R]$ is given by the quotient sheaf $U/R:=\coeq_{\Shv(\Rig_E)}(R\rightrightarrows U)$.
Since $|-|$ is a left adjoint, $|\frX| = |F_{\frX}| = |F_{[U/R]}| = |\coeq_{\Shv(\Rig_E)}(R\rightrightarrows U)| = \coeq_{\Top}(|R|\rightrightarrows |U|) = |U|/|R|$.
\end{proof}

\begin{cor}
\label{cor:conn_open}
Any rigid analytic Artin stack $\frX$ has a locally connected underlying topological space, and hence its connected components are open.
\end{cor}
\begin{proof}
Since any quotient of a locally connected space is locally connected, by Lemma \ref{lem:presentation_top} it suffices to show that for any $A\in \Aff_E$, $|\Spa(A,A^\circ)|$ is locally connected.

\vskip8pt

{\it Claim.} The adic space $X = \Spa(A,A^\circ)$ has finitely many connected components, and these are open. 

\vskip8pt

{\it Proof of the claim.} There is nothing to show if $X$ is connected. If $X$ is not connected, we can write $X = X_1 \sqcup X_2$ as the union of two non-empty disjoint open subsets. If one of these is not connected, we can further decompose. If this process stops after finitely many steps, we are done. If not, we get an infinite sequence of open and closed subspaces $X = U_0 \supsetneq U_1 \supsetneq U_2 \supsetneq \ldots$. If we then define $I_n \sub A$ to be the ideal vanishing on $U_n$, then we have $0 = I_0 \subsetneq I_1 \subsetneq I_2 \subsetneq \ldots$, which is a non-stationary increasing chain of ideals in $A$, contradicting that $A$ is noetherian. \qedsymbol
\end{proof}

\begin{lemma}
\label{lem:fiber_product_surjective}
Let $X\xrightarrow{f} Z\xleftarrow{g} Y$ be morphisms of rigid analytic Artin stacks over $\Sp(E)$.
Then, the natural map 
$
|X\times_Z Y| \to |X|\times_{|Z|}|Y|
$
is surjective.
\end{lemma}
\begin{proof}
{\it Step 1. Rigid spaces over $E$.}
Suppose $X,Y,Z\in \Rig_E$ are rigid spaces.
Recall that any map of rigid spaces $f:X\to Z$ induces a map $r(f):r(X)\to r(Z)$ locally of finite type of analytic adic spaces\footnote{Because $p\in E$ is a topologically nilpotent unit, any adic space over $\Spa(E,\cO_E)$ is analytic.} by \cite[\mbox{}4.3]{Hu94}.
Since $r(Z)$ is analytic, $r(g):r(Y)\to r(Z)$ is adic by \cite[{}3.2(ii)]{Hu94}.
Thus, we can use \cite[Corollary 1.2.3(iii.d)]{Hu96}.
\vskip8pt

{\it Step 2. Quasi-analytic spaces.}
Suppose $X,Y,Z$ are quasi-analytic spaces over $\Rig_E$.
Let $p_Z:U_Z\to Z$ be an \'{e}tale surjection with $U_Z\in \Rig_E$.
Then, with $R_Z:=U_Z\times_Z U_Z$, we get an \'{e}tale groupoid $(U_Z, R_Z, s_Z, t_Z, c_Z)$ with an isomorphism $Z\cong U_Z/R_{Z}$.
So, $|Z| =|U_Z|/|R_Z|$, and we claim that $\im(|R_Z|\xrightarrow{(|s_Z|,|t_Z|)} |U_Z|\times|U_Z|)$ is an equivalence relation on $|U_Z|$.
Indeed, for any groupoid $(U,R,s,t,c)$ in a category, such an image relation is always reflexive by the identity section $e:U\to R$ and is symmetric by the inverse map $i:R\to R$, cf. \cite[\href{https://stacks.math.columbia.edu/tag/0230}{Tag 0230}]{stacks-project}.
The relation $\im(|R_Z|\xrightarrow{(|s_Z|,|t_Z|)} |U_Z|\times|U_Z|)$ is transitive by applying Step 1 to the map $|R_Z\times_{s_Z,U_Z,t_Z}R_Z|\twoheadrightarrow |R_Z|\times_{|s_Z|,|U_Z|,|t_Z|}|R_Z|$ of rigid spaces, and hence an equivalence relation.
Now let $f:X\to Z$ and $g:Y\to Z$ be arbitrary maps of quasi-analytic spaces.
Picking \'{e}tale surjections $U_X \twoheadrightarrow X\times_{f,Z,p_Z} U_Z$ and $U_Y \twoheadrightarrow U_Z\times_{p_Z,Z,g} Y$ of the \'{e}tale base changes, we obtain two  \'{e}tale surjections $p_X:U_X\twoheadrightarrow  X$ and $p_Y:U_Y\twoheadrightarrow  Y$ as the composites, with $U_X,U_Y\in \Rig_E$.
By base change, we have that $U_X\times_{Z} U_Y\twoheadrightarrow X\times_Z Y$ is an \'{e}tale surjection, where
$$
U_X\times_{Z} U_Y = U_X \times_{U_Z} (U_Z \times_Z U_Z) \times_{U_Z} U_Y = U_X \times_{U_Z,s_Z} R_Z \times_{t_Z,U_Z} U_Y \in \Rig_E.
$$
The discussion at the beginning of Step 2, where $Z$ was an arbitrary quasi-analytic space, implies that we have a surjection
$$
|U_X\times_{Z} U_Y| \twoheadrightarrow |U_X| \times_{|U_Z|,|s_Z|} |R_Z| \times_{|t_Z|,|U_Z|} |U_Y| \twoheadrightarrow |X| \times_{|Z|} |Y|
$$
that factors through $|X\times_Z Y| \to |X|\times_{|Z|}|Y|$. 
So, the latter map is also surjective.
\vskip8pt

{\it Step 3. Rigid analytic Artin stacks.}
Let $X,Y,Z$ be rigid analytic Artin stacks on $\Rig_E$.
The results obtained in Step 2 together with Lemma \ref{lem:presentation_top} imply that the argument of Step 2 can be adapted to Step 3, which shows that $|X\times_Z Y| \to |X|\times_{|Z|}|Y|$ is surjective.
\end{proof}

\begin{rem}
\label{rem:image_equivalence}
By Lemma \ref{lem:fiber_product_surjective}, we have that if $[U/R]\simeq \frX$ is a smooth presentation of an Artin stack $\frX$ on $\Rig_E$ coming from a smooth groupoid $(U,R,s,t,c)$ in quasi-analytic spaces, then the image $\im(|R|\xrightarrow{|s|,|t|} |U|\times|U|)$ is an equivalence relation on $|U|$ such that 
$$
|\frX|\cong |U|/\im(|R|\to |U|\times|U|)
$$ 
is the quotient of $|U|$ by this equivalence relation.
\end{rem}

\begin{cor}
\label{cor:open_immer_top}
Let $f: X\to Y$ be an open immersion, i.e., an \'{e}tale monomorphism, of rigid analytic Artin stacks on $\Rig_E$.
Then, the induced map $|f|:|X|\to |Y|$ is an open embedding.
\end{cor}
\begin{proof}
Let $q: V \twoheadrightarrow Y$ be a smooth surjection with $V\in \Rig_E$.
For the fiber product $X\times_Y V$, which is a quasi-analytic space \'{e}tale over $V$, we choose an \'{e}tale surjection $U\twoheadrightarrow X\times_Y V$ with $U\in \Rig_E$. 
Then the composite $g:U\to X\times_Y V\to  V$ is an \'{e}tale map of rigid spaces, and the composite $p:U \to X\times_Y V\to X$ is a smooth surjection.
That is, 
$$
\begin{tikzcd}
U \arrow[r] \arrow[d, "p"] & X \times_Y V \arrow[r] \arrow[ld] & V \arrow[d, "q"] \\
X \arrow[rr, "f"]                  &                                   & Y               
\end{tikzcd}
$$
is a commutative diagram.
Taking the underlying topological space, we have that 
$$
\begin{tikzcd}
{|U|} \arrow[r, two heads] \arrow[d, "{|p|}", two heads] & {|X \times_Y V|} \arrow[r] \arrow[ld, two heads] & {|V|} \arrow[d, "{|q|}", two heads] \\
{|X|} \arrow[rr, "{|f|}"]                          &                                                      & {|Y|}                            
\end{tikzcd}
$$
is a commutative diagram in $\Top$.
Since $U\to V$ is \'{e}tale in $\Rig_E$, $|U|\to |V|$ is an open map.
The vertical maps $|p|$ and $|q|$ are surjective quotient maps.
We claim that $|f|$ is an open map. 
Let $O\subset |X|$ be an open subset.
Then $|f|(O)$ is open in $|Y|$ if and only if $|q|^{-1}(|f|(O))$ is open in $|V|$.
Observe that $|q|^{-1}(|f|(O))=|g|(|p|^{-1}(O))$ by Lemma \ref{lem:fiber_product_surjective}, which is open since $|g|$ is an open map.
So, $|f|:|X|\to |Y|$ is an open map, and hence a quotient map onto its image.
So is the map $|f| \circ |p|: |U|\to |Y|$.
Since $X\hookrightarrow Y$ is a monomorphism, we have $U\times_X U \simeq U\times_Y U$, and hence by Lemma \ref{lem:fiber_product_surjective}
$$
\im(|f|)=\im(|f|\circ |p|)\cong |U|/|U\times_{Y}U| = |U|/|U\times_{X}U| \cong |X|,
$$
which shows that $|f|:|X|\to |Y|$ is an open embedding.
\end{proof}

\begin{lemma}
\label{lem:rigid_open_top}
Let $T\in \Rig_E$.
For any open subset $O\subset |r(T)|$, consider the subfunctor $T_O\subset T$ defined by the formula
$$
S\in \Rig_E \quad \longmapsto \quad T_O(S)= \left\{ f:S\to T \mid |r(f)|(|r(S)|)\subset O \right\} .
$$
Then, $T_O$ is a quasi-analytic space and $T_O\to T$ is an open immersion. 
Moreover, every open quasi-analytic subspace of $T$ is of this form for a unique open subset $O\subset |r(T)|$.
\end{lemma}
\begin{proof}
Let $O\subset |r(T)|$ be an open subset.
We give another characterization of $T_O$.
Choose an admissible affinoid open cover $T=\bigcup_\alpha T_\alpha,$ with $T_\alpha=\Sp(A_\alpha)$.
Let $O_\alpha:=O\cap |r(T_\alpha)|\subset |r(T_\alpha)|$.
Since $r(T_\alpha)$ is locally spectral, we can write $O_\alpha=\bigcup_\lambda O_{\alpha,\lambda}$ with $O_{\alpha,\lambda}\subset |r(T_\alpha)|$ quasi-compact open.
As $T_\alpha$ is affinoid, each $O_{\alpha,\lambda}$ is of the form $O_{\alpha,\lambda}=|r(U_{\alpha,\lambda})|$ for a unique quasi-compact admissible open $U_{\alpha,\lambda}\subset T_\alpha$ by \cite[(1.1.11)]{Hu96}.
Let $U:=\coprod_{\alpha,\lambda} U_{\alpha,\lambda}.$
The natural map $U\to T$ is \'etale whose sheaf image $\im(U\to T)$ we now claim to be $T_O$. 
Indeed, if $f:S\to T$ Tate-fpqc locally factors through $U$, then  $|r(f)|(|r(S)|)\subset O.$
Conversely, suppose $|r(f)|(|r(S)|)\subset O.$
Then, $U\times_T S \to S$ is \'etale, and the map $|r(U\times_T S)|\to |r(S)|$ is surjective by \cite[Corollary 1.2.3(iii.d)]{Hu96}.
By \cite[Lemma 2.1.2]{Hu96}, $U\times_T S\to S$ is an \'etale cover.
On this cover $f$ factors as $U\times_T S\to U \to T$, so $f\in \im(U\to T)(S)$ and hence $T_O=\im(U\to T)$.
Since $T_O$ is a subfunctor of $T$, the diagonal of $T_O$ is representable.
It follows that $U\to T_O$ is representable, \'etale, and surjective.
Hence, $T_O$ is a quasi-analytic space, and $T_O\to T$ is an open immersion.
By Corollary \ref{cor:open_immer_top}, $|T_O|\to |T|$ is an open embedding whose image is $O$ by construction.


\vskip8pt

Conversely, let $i:W\to T$ be an open immersion from a quasi-analytic space $W$. 
Pick an \'etale surjection $U\to W$ with $U\in \Rig_E$.
Then, the composite $U\to W\to T$ is \'etale, hence induces an open map $|r(U)|\to |r(T)|$.  
Let $O:=\im(|r(U)|\to |r(T)|) = \im(|W|\to |r(T)|)$.
Then $O$ is open and homeomorphic to $|W|$ by Corollary \ref{cor:open_immer_top}. 
Since $U\to W$ is an epimorphism and $W\to T$ is a monomorphism, $W$ is the sheaf image of $U\to T$. 
By the argument above, this sheaf image is $T_O$, hence $W\simeq T_O = T_{|i|(|W|)}$.
Uniqueness follows from $|T_O|=O.$
\end{proof}

\begin{lemma}
\label{lem:qan_open_top}
Let $X$ be a quasi-analytic space on $\Rig_E$. 
Then, open subspaces of $X$ correspond bijectively to open subsets of $|X|$ by the map $W\mapsto |W|$.
\end{lemma}
\begin{proof}
Let $i:W\hookrightarrow X$ be an open quasi-analytic subspace. 
By Corollary \ref{cor:open_immer_top}, $|i|:|W|\hookrightarrow |X|$ is an open subset.
We claim that $W$ is the subfunctor $X_{|W|}$ of $X$ defined by 
$$
S\in \Rig_E \quad \longmapsto \quad X_{|W|}(S):= \left\{ f:S\to X \mid |f|(|S|)\subset |W| \right\}.
$$
The inclusion $W\subset X_{|W|}$ is clear.
Conversely, let $f: S\to X$ be a map satisfying the condition that $|f|(|S|)\subset |W|$ for some $S\in \Rig_E$.
Consider the fiber product diagram
$$
\begin{tikzcd}
S\times_X W \arrow[d, hook] \arrow[r] & W \arrow[d, "i", hook] \\
S \arrow[r, "f"]                      & X                     
\end{tikzcd}
$$
where $S\times_X W\hookrightarrow S$ is an open immersion into the rigid space $S$.
By Lemma \ref{lem:rigid_open_top}, $|S\times_X W|\to |S|$ is an open embedding.
Since $|f|$ factors through $|W|$, it follows from Lemma \ref{lem:fiber_product_surjective} that $|S\times_X W|\to |S|$ is also surjective.
So, we have $S\times_X W = S$ by Lemma \ref{lem:rigid_open_top}, which implies that $f:S\to X$ factors through $i:W\hookrightarrow X$.
Thus, $W=X_{|W|}$.
\vskip8pt

Conversely, given an open subset $O\subset |X|$, consider the subfunctor $X_{O}$ of $X$ defined by 
$$
S\in \Rig_E \quad \longmapsto \quad X_{O}(S):= \left\{ f:S\to X \mid |f|(|S|)\subset O \right\},
$$
which is a Tate-fpqc sheaf.
Since $X_O\subset X$, the diagonal map of $X_O$ is representable. 
Pick an \'{e}tale surjection $p:U\to X$ with $U\in \Rig_E$, and let $\wtO:=|p|^{-1}(O)\subset |U|$.
Consider the associated open subspace $U_{\wtO}\subset U$ in Lemma \ref{lem:rigid_open_top}.
We note that $U\times_{X} X_O \cong U_{\wtO}$ over $U$ by construction. 
So, $U\times_{X} X_O$ is a quasi-analytic space.
The pullback of $p:U\to X$ along the inclusion $i:X_O\hookrightarrow X$ is a map $p_i:U\times_{X} X_O\to X_O$ that is representable in rigid spaces, \'{e}tale, and surjective, because by the transitivity of fiber product, the base change of $p_i$ along any map $f:S\to X_O$ is the base change $S\times_X U \to S$ of $p$ along the map $i\circ f: S\to X_O \to X$, and $p$ is an \'{e}tale surjection.
Choosing any \'{e}tale surjection $V\to U\times_X X_O$, we obtain an \'{e}tale surjection $V\to X_O$.
Hence, $X_O$ is a quasi-analytic space.
Moreover, $j:X_O\hookrightarrow X$ is an open immersion, because its base change to the \'{e}tale cover $p:U\to X$ is the open immersion $U_{\wtO}\hookrightarrow U$.
Hence, $X_O\subset X$ is an open quasi-analytic subspace.
By Corollary \ref{cor:open_immer_top}, $|X_O|\subset |X|$ is an open subset, which, by the description $U_{\wtO} = X_O \times_X U$ of the fiber product and Lemma \ref{lem:fiber_product_surjective}, is given by $|p|(\wtO) = O$.
\end{proof}

\begin{rem}
\label{rem:fiber_top}
Let $f:Y\to X$ be a morphism of quasi-analytic spaces, and let $W\subset X$ be an open quasi-analytic subspace. Then, we have $|Y\times_X W|=|f|^{-1}(|W|)$, or simply, 
$$
|f^{-1}(W)|=|f|^{-1}(|W|).
$$
Indeed, for any morphism $g:S\to Y$ with $S\in \Rig_E$, it factors through $f^{-1}(W)\subset Y$ if and only if $f\circ g:S\to X$ factors through $W$, which is the case if and only if $|g|(|S|)\subset |f|^{-1}(|W|)$ by Lemma \ref{lem:qan_open_top}.
By Lemma \ref{lem:qan_open_top}, this shows that $f^{-1}(W)=Y_{|f|^{-1}(|W|)}$, as desired.
\end{rem}

\begin{dfn}
Let $(U,R,s,t,c)$ be a smooth groupoid in quasi-analytic spaces. 
\begin{enumerate}[(i)]
\item
An open quasi-analytic subspace $W\subset U$ is called $R$\textbf{-invariant} if $s^{-1}(W)=t^{-1}(W)$ as open quasi-analytic subspaces of $R$. 
\item
An open subset $O\subset |U|$ is called \textbf{$|R|$-invariant} if $|s|^{-1}(O)=|t|^{-1}(O)$ as open subsets of $|R|$.
\end{enumerate}
\end{dfn}

\begin{lemma}
\label{lem:R-invariant_top}
If $(U,R,s,t,c)$ is a smooth groupoid in quasi-analytic spaces, then under the bijection $W\mapsto |W|$ of Lemma \ref{lem:qan_open_top}, R-invariant open subspaces of U correspond bijectively to $|R|$-invariant open subsets of $|U|$.
\end{lemma}
\begin{proof}
Let $W\subset U$ be an open quasi-analytic subspace.
By Lemma \ref{lem:qan_open_top} and Remark \ref{rem:fiber_top}, we have $s^{-1}(W)=t^{-1}(W)$ if and only if $|s|^{-1}(|W|)=|t|^{-1}(|W|)$, which completes the proof.
\end{proof}

Next, we show that there is a bijective correspondence between open substacks of $\frX$ and open subsets of $|\frX|$.

\begin{thm}
\label{thm:stack_opens}
Let $\frX$ be a rigid analytic Artin stack on $\Rig_E$, and let $ p:U\twoheadrightarrow \frX $ be a smooth surjection from a rigid space. 
Put $R:=U\times_{\frX}U.$
Then the map
$$
\cW\subset \frX
\quad \rightsquigarrow \quad 
\widetilde{\cW}:= \cW\times_{\frX}U\subset U
$$
induces an inclusion-preserving bijection between open substacks of $\frX$ and $R$-invariant open quasi-analytic subspaces of $U$. 
By Lemma \ref{lem:R-invariant_top}, open substacks of $\frX$ correspond bijectively to $|R|$-invariant open subsets of $|U|$, and hence to open subsets of $|\frX|$.
\end{thm}

\begin{proof}
    Let $(U,R,s,t,c)$ be the smooth groupoid in quasi-analytic spaces given by the smooth surjection $p:U\twoheadrightarrow \frX$, which induces $[U/R]\cong \frX$. 
    For any open quasi-analytic subspace $i:\widetilde \cW\hookrightarrow U$, we write $U_{\widetilde\cW} := \widetilde{\cW}$ and consider the restriction of $R$ to $\widetilde \cW$, 
    which is 
    $$
    R_{\widetilde \cW} = ({\widetilde \cW}\times_{U,t}R)\times_R(R\times_{s,U} {\widetilde \cW}) \cong ({\widetilde \cW}\times {\widetilde \cW})\times_{U\times U} R.
    $$
    Then, $(U_{\widetilde\cW},R_{\widetilde\cW},s|_{R_{\widetilde\cW}}, t|_{R_{\widetilde\cW}},c|_{R_{\widetilde\cW}\times_{s,\widetilde{\cW},t}R_{\widetilde\cW}})$ is a groupoid.
    Suppose $\widetilde \cW$ is an $R$-invariant open subspace of $U$.
    Then, $R_{\widetilde\cW} = s^{-1}(\widetilde \cW) = t^{-1}(\widetilde \cW)$ as subspaces of $R$, the quintuple above is a smooth groupoid in quasi-analytic spaces, and one checks that there is a 2-fiber square
    $$
    \begin{tikzcd}
    {U_{\widetilde\cW}} \arrow[r, "p_{\widetilde\cW}"] \arrow[d,"i"] & {[ U_{\widetilde\cW} / R_{\widetilde\cW} ]} \arrow[d, "{[i]}"] \\
        U \arrow[r, "p"] & {[U/R]}.
    \end{tikzcd}
    $$
    Since \'{e}taleness is smooth-local on the target (Lemma \ref{lem:local_property_P}), the smooth surjection $p:U\to [U/R]$ and the open immersion $i:U_{\widetilde\cW} \to U$ together imply that $[i]$ is an \'{e}tale monomorphism, i.e., an open immersion.
    Therefore, the strict essential image of $[ U_{\widetilde\cW} / R_{\widetilde\cW} ]$ under $[i]$ is an open substack $\cW$ of $[U/R]\simeq \frX$ whose pullback along $p:U\to \frX$ is, by the 2-fiber square above, the $R$-invariant open subspace $U_{\widetilde\cW} = \widetilde{\cW}$ of $U$.
    \vskip8pt
    
    Conversely, any open substack $\cW \hookrightarrow \frX\simeq [U/R]$ gives an open quasi-analytic subspace $\widetilde{\cW}:=\cW\times_{\frX} U\hookrightarrow U$ that is $R$-invariant:
    $$
    s^{-1}(\widetilde\cW) = R \times_{s,U}(U\times_\frX \cW) = R\times_{p\circ s,\frX} \cW = R\times_{p\circ t,\frX} \cW = R \times_{t,U}(U\times_\frX \cW) = t^{-1}(\widetilde\cW)\,.
    $$
    Thus, the restriction of $R$ to ${\widetilde\cW}$ is given by $R_{\widetilde\cW} = s^{-1}({\widetilde\cW})=t^{-1}({\widetilde\cW})$.
    Since ${\widetilde\cW}\to \cW$ is a smooth surjection, we have an equivalence $[{\widetilde\cW}/({\widetilde\cW}\times_\cW {\widetilde\cW})]\simeq {\cW}$.
    Since 
    $$
    \begin{aligned} 
    \widetilde{\mathcal{W}} \times_{\mathcal{W}} \widetilde{\mathcal{W}} &=\left(U \times_\frX \mathcal{W}\right) \times_{\mathcal{W}}\left(U \times_\frX \mathcal{W}\right) 
    \simeq
    \left(U \times_\frX U\right) \times_\frX \mathcal{W} 
    =
    R \times_\frX \mathcal{W} 
    =
    R_{\widetilde{\mathcal{W}}} \,,
    \end{aligned}
    $$
    we have $[\widetilde\cW/R_{\widetilde\cW}]\simeq \cW$.
    This establishes the desired bijective correspondence.
\end{proof}

\begin{cor}
Let $\frX$ be a rigid analytic Artin stack on $\Rig_E$.
There is a bijective correspondence between connected components of $\frX$ and connected components of $|\frX|$.
\end{cor}
\begin{proof}
By Corollary \ref{cor:conn_open}, the connected components of $|\frX|$ are open, which by Theorem \ref{thm:stack_opens} allows us to write $\frX$ as a disjoint union of connected open substacks.
These are the connected components of $\frX$.
\end{proof}

\bibliographystyle{alpha}
\bibliography{mybib}

\end{document}